\documentclass[12pt]{amsart}
\usepackage{etex}
\usepackage{graphicx}
\usepackage{epstopdf}
\usepackage{subcaption}
\usepackage{ytableau}
\usepackage{cite}
\usepackage{tikz, tikz-cd, tkz-graph}
\usetikzlibrary{arrows, patterns}

\usepackage[margin=1in]{geometry}
\usepackage{times}

\usepackage{xcolor}

\usepackage{hyperref}
\hypersetup{
  colorlinks,
  linkcolor={red!50!black},
  citecolor={blue!50!black},
  urlcolor={blue!80!black}
}

\usepackage[english]{babel}
\usepackage{amsmath,amssymb,amsthm}
\usepackage{xypic}
\usepackage{longtable}
\usepackage{array}

 \usepackage[enableskew]{youngtab}

\usepackage{epsfig}
\usepackage{hyperref}
\usepackage{enumitem}
\usepackage{booktabs}

\graphicspath{{./images/}}

\newtheorem{thm}{\bf Theorem}[section]
\newtheorem{eg}[thm]{\bf Example}
\newtheorem{prop}[thm]{\bf Proposition}
\newtheorem{cor}[thm]{\bf Corollary}
\newtheorem{mydef}[thm]{\bf Definition}
\newtheorem{lem}[thm]{\bf Lemma}

\newenvironment{claim}[1]{\par\noindent\underline{Claim:}\space#1}{}
\newenvironment{claimproof}[1]{\par\noindent\underline{Proof:}\space#1}{\hfill $\blacksquare$}

\theoremstyle{remark}
\newtheorem{rem}[thm]{\bf Remark}

\newcommand{\CC}{\mathbb{C}}

\newcommand{\ZZ}{\mathbb{Z}}

\newcommand{\NN}{\mathbb{N}}

\newcommand{\br}{\mathbf{r}}

\newcommand{\Mon}{\mathrm{Mon}}

\newcommand{\VG}{V/\!/_\theta G}
\newcommand{\VT}{V/\!/_\theta T}
\newcommand{\LM}{\mathrm{LM}}
\newcommand{\sh}{\mathrm{shape}}

\DeclareMathOperator{\Rep}{\mathrm{Rep}}

\DeclareMathOperator{\add}{\mathrm{add}}
\DeclareMathOperator{\Hom}{\mathrm{Hom}}
\DeclareMathOperator{\GL}{\mathrm{GL}}
\DeclareMathOperator{\Proj}{\mathrm{Proj}}

\DeclareMathOperator{\Gr}{\mathrm{Gr}}
\DeclareMathOperator{\Sym}{\mathrm{Sym}}

\DeclareMathOperator{\Cox}{\mathrm{Cox}}

\DeclareMathOperator{\Pic}{\mathrm{Pic}}

\DeclareMathOperator{\Mat}{\mathrm{Mat}}
\DeclareMathOperator{\sign}{\mathrm{sign}}
\DeclareMathOperator{\SI}{\mathrm{SI}}
\DeclareMathOperator{\id}{\mathrm{id}}

\renewcommand{\emptyset}{\varnothing}

\newcolumntype{C}[1]{>{\centering\let\newline\\\arraybackslash\hspace{0pt}}m{#1}}

\title[]{Quiver semi-invariants and SAGBI bases}

\author[L.~Heuberger]{L.~Heuberger}
\address{Liana Heuberger \newline \indent Institut de Mat\'ematiques de Marseille, Aix-Marseille Universit\'e\newline \indent
3 Pl. Victor Hugo, 13331 Marseille Cedex 3}
\email{liana.heuberger@univ-amu.fr}

\author[E.~Kalashnikov]{E.~Kalashnikov}
\address{Elana Kalashnikov \newline \indent  Department of Pure Mathematics, University of Waterloo,\newline \indent
Waterloo, Canada N2L 3G1}
\email{e2kalash@uwaterloo.ca}

\thanks{LH is supported by Research Project Grant RPG-2021-149 from The Leverhulme Trust. EK is supported by an NSERC Discovery Grant. The authors thank the anonymous referee for very helpful comments.}
\begin{document}
\begin{abstract}
We introduce a new combinatorial structure of \emph{linked tableaux},  which generalize the semi-standard tableaux that index a SAGBI basis of the Pl\"ucker coordinate ring of a flag variety. We show that linked tableaux index Domokos-Zubkov semi-invariants, which span the semi-invariant ring of a quiver. The semi-invariant ring of a quiver coincides in many cases with the Cox ring of an associated quiver moduli space.  We show that these semi-invariants satisfy straightening laws inherited from Pl\"ucker coordinates.  In the case of the generalized Kronecker quiver, we prove that the semi-invariants associated to semi-standard linked tableaux are a (possibly infinite) SAGBI basis.  For the generalized Kronecker quiver with dimension vector $(2,2)$, we show that the SAGBI basis is finite and describe it explicitly. \end{abstract}
\maketitle

\section{Introduction}\label{sec:intro}
Many interesting varieties can be constructed as GIT quotients of vector spaces $V/\!/G$:  important examples include toric varieties, Grassmannians, and flag varieties. When $G=T$ is a torus, the quotient is a toric variety, for which many questions are highly computable. For a general group $G$, and GIT quotient $V/\!/G$, one can replace $G$ with a maximal torus $T$. The associated toric variety $V/\!/T$ is sometimes called the abelianization of $V/\!/G.$ 

Understanding the connection between  $V/\!/G$ and $V/\!/T$  can allow one to extend the combinatorial control of toric varieties to other, more general, varieties. This strategy, which can be called the \emph{Abelian/non-Abelian correspondence} has been applied successfully in many different contexts: cohomology \cite{martin,stromme}, $I$ and $J$ functions \cite{abelian, Webb2018, kalashnikov}, and quantum cohomology \cite{gukalashnikov}. The inspiration for this paper was to ask whether this strategy could give insight in understanding the Cox ring of $V/\!/G.$  The Cox ring of a variety $X$, informally defined, is the algebra generated by all global sections of all line bundles of $X$. 

Examples of Cox rings include such well-studied objects as the Pl\"ucker coordinate ring of the Grassmannian and the homogeneous coordinate ring of a toric variety. Finding generators of  the Cox ring -- or even just generators of sub-algebras of the Cox ring -- gives information about the variety; for example, through understanding maps into projective space. One important application is that a SAGBI basis of the Cox ring induces a toric degeneration of the variety. Toric degenerations extend the combinatorial control of toric varieties to other varieties and also have important applications in mirror symmetry. 

In this paper, we study the Cox ring of quiver moduli spaces via quiver semi-invariants and the Abelian/non-Abelian correspondence. Our main results are:
\begin{itemize}
\item We introduce pairs of linked tableaux as a combinatorial structure indexing Domokos--Zubkov \cite{domokos} generators of the Cox ring. We show that these generators satisfy straightening laws inherited from Pl\"ucker coordinates.
\item In the case of a generalized Kronecker quiver, we show that semi-standard and primitive linked tableaux index a (possibly infinite) SAGBI basis of the Cox ring.
\item We prove that this SAGBI basis is finite in the case of the generalized Kronecker quiver with dimension vector $(2,2)$. 
\end{itemize}
We describe these results in more detail in Theorems \ref{thm:intro1}-\ref{thm:intro3}.

Let $G$ be a reductive group acting on a complex vector space $V$. In this paper, we always work over $\CC$. A stability condition $\theta$, which is a character of $G$, gives a GIT quotient $\VG=V^{ss}/G$ where the semi-stable locus $V^{ss}$ is determined by $\theta$. Given another character $\alpha \in \chi(G)$, we can consider the projection
\[ V^{ss} \times \CC/G \to V^{ss}/G= \VG.\]
The action of $G$ on the product is the diagonal action defined by the representation $V$ of $G$ on the first factor and the character $\alpha$ on the second. When the GIT quotient is smooth, by Kempf's descent lemma $V^{ss} \times \CC/G$ is a line bundle over $\VG$, which we denote $L_\alpha$. Any element $f \in \CC[V]$ such that, for all $g \in G$,
\[g \cdot f = \chi_\alpha(g) f,\]
defines a section of $L_\alpha$. 
Such an $f$ is called a \emph{$G$-semi-invariant} of weight $\alpha.$

In many cases of interest -- such as quiver flag varieties --  the Cox ring of $\VG$ is the algebra of semi-invariants. This holds \cite{crawunpub} whenever the unstable locus $V^{us}=V\backslash V^{ss}$ has codimension at least two and
\[\chi(G) \cong \Pic(\VG).\]
The results in this paper concern the semi-invariant ring, and hence (when they coincide) the Cox ring. 

The Weyl group $W$ acts on $\chi(T)$ and $\chi(G) \cong \chi(T)^W.$   If a polynomial $f \in \CC[V]$ is $G$-semi-invariant, then it is $T$-semi-invariant of the same weight. Moreover, since $T$ scales the monomials in $\CC[V]$, it follows that if $f$ is $G$-semi-invariant, it is a sum of $T$-semi-invariant monomials. 

This suggests three immediate questions:
\begin{enumerate}
\item Given a $T$-semi-invariant monomial with weight in $\chi(T)^W$, can we construct a $G$-semi-invariant polynomial using the Weyl group?
\item Which $T$-semi-invariant monomials appear as summands of a $G$-semi-invariant polynomial?
\item Given a term order, which $T$-semi-invariant monomials appear as the leading term in a $G$-semi-invariant polynomial?
\end{enumerate}
This paper is motivated by these questions in the setting of quiver moduli spaces. A quiver moduli space is a GIT quotient $V/\!/_\theta G$ constructed from a quiver $Q$.  The abelianization $V/\!/_\theta T$ is the quiver moduli space associated to the abelianized quiver $Q^{ab}$ of $Q$.  

We introduce some notation to state our results. Let $Q=(Q_0,Q_1)$ be a quiver, so that $Q_0=\{0,1,\dots,\rho\}$ is the set of vertices, and $Q_1$ the set of arrows. All quivers in this paper are assumed to be acyclic. There are two maps $s,t: Q_1 \to Q_0$ taking arrows to their source and targets respectively. Given a choice of dimension for each vertex $\br=(r_0,\dots,r_\rho) \in \NN^{\rho+1},$ set 
\[\Rep(Q,\br)=\oplus_{a \in Q_1} \Hom(\CC^{r_{s(a)}},\CC^{r_{t(a)}}).\]
There is a natural $\GL(\br)=\prod_{i=0}^\rho \GL(r_i)$ action on $\Rep(Q,\br)$. The GIT quotient 
\[\Rep(Q,\br)/\!/_\theta \GL(\br)\]
is the quiver moduli space associated to $(Q,\br)$ and was first considered in \cite{king}. Semi-invariants of the $\GL(\br)$-action on $\Rep(Q,\br)$ are called \emph{quiver semi-invariants}, and have been amply studied  \cite{schofieldbergh, schofield, derksen, domokos}. 

In the first part of the paper, we explain how to produce the semi-invariants  of Domokos--Zubkov \cite{domokos} from this perspective, answering the first question above. We show how a new combinatorial structure, \emph{linked tableaux pairs}, can be used to index these semi-invariants, which we call DZ semi-invariants. Strictly speaking, these semi-invariants agree with the semi-invariants of \cite{domokos} only in the bipartite case, but are otherwise a simple generalization replacing arrows with paths, see \S \ref{sec:dzcomparison}.  Tableaux are key in understanding the structure of the Cox ring of flag varieties, and linked tableaux pairs play an analogous role for quiver moduli spaces. 

For the definition of a linked tableaux pair, see Definition \ref{def:linked}, but roughly speaking these are pairs $(\underline{T}^-,\underline{T}^+)$, where $\underline{T}^\pm$ is a tuple of rectangular tableaux, one for each vertex in the quiver, together with a \emph{link}. The entries of the tableaux are of the form $Pj$, where $P$ is a path in the quiver and $j$ is an integer. The link is a bijection $\sigma$ between the labels of $\underline{T}^-$ and $\underline{T}^+$, that takes a label $P j$ in the $k^{th}$ row of $T^-_{s(P)}$ to a label $P k$ in the $j^{th}$ row of $T^+_{t(P)}$, as illustrated below: 
\[\begin{tikzpicture}[scale=0.6,every node/.style={scale=0.8}]
\draw (0,0) rectangle (5,5);
\draw (2,1) rectangle (3,2);
\node[] at (2.5,-1) {$T^-_{s(P)}$};
\node[] at (-1,1.5) {$k^{th} \text{ row}$};
\node[] at (2.5,1.5) {$P j$};

\draw (7,0) rectangle (12,5);
\draw (9,3) rectangle (10,4);
\node[] at (9.5,-1) {$T^+_{t(P)}$};
\node[] at (13,3.5) {$j^{th} \text{ row}$};
\node[] at (9.5,3.5) {$P k$};
\draw[->] (3.1,1.5) -- (8.9,3.5) node[midway,above left] {$\sigma$};

\end{tikzpicture}\]
By ordering the paths of the quiver, and then using the lex ordering, we can extend the notion of semi-standard to linked tableaux pairs. 

We exploit the generalization of tableaux to linked tableaux to extend results for Grassmannians to quiver moduli spaces. In particular, we show that the analogue of the straightening laws for Grassmannians hold in this context:
\begin{thm}\label{thm:intro1} Domokos--Zubkov semi-invariants  satisfy straightening laws coming from linked tableaux. Linked tableaux with weakly increasing columns and rows span the semi-invariant ring. 
\end{thm}
In the special case of the generalized Kronecker quiver, we can use linked tableaux to answer the third question as well. The main motivation in considering this question is finding SAGBI bases of quiver semi-invariants. SAGBI bases were introduced by Robbiano and Sweedler \cite{Robbiano1}; for more background, see \cite{Robbiano2}.  SAGBI basis of an algebra $R \subset \CC[V]$ with a term order is a set $S$ of elements in $R$ such that the leading term of any element in $R$ can be written as a product of the leading terms of elements from $S$.  

One powerful application of SAGBI bases is in producing toric degenerations. Toric degenerations extend the combinatorial control we have over toric varieties   to other varieties; for example, they are used in various mirror symmetry constructions. From this perspective, the central example of a toric degeneration is the Gelfand--Cetlin toric degeneration of type A flag varieties, due to Gonciulea--Lakshmibai \cite{GL}. This is a SAGBI basis degeneration where the basis is indexed by semi-standard Young tableaux.  

We show that semi-standard linked tableaux index a (possibly infinite) SAGBI basis for the generalized Kronecker quiver. The generalized Kronecker quiver (henceforth, just Kronecker quiver) is the quiver with two vertices $0$ and $1$, with $K$ arrows from $0$ to $1$:
 \[\begin{tikzcd}
 0 \arrow[r,"K"] & 1
\end{tikzcd} \]
 Fix a dimension vector $\br=(r_0,r_1).$ 
 
 Linked tableaux pairs are simpler in this case, as the tuples of tableaux are of length one and all paths are in fact arrows, which we label from $1,\dots,K$.  Let $\alpha^-$ be the partition of shape $r_0 \times a r_1$ and  $\alpha^+$  the partition of shape $r_1 \times a r_0$, for $a \in \frac{1}{\gcd(r_0,r_1)} \ZZ.$ Let $T^-$ be a tableau of shape $\alpha^-$, filled with entries of the form $ij$, where 
 \[i \in \{1,\dots,K\}, \hspace{5mm} j \in \{1,\dots,r_1\}.\]
  Let $T^+$ be a tableau of shape $\alpha^+$, filled with entries of the form $ij$, where 
  \[i \in \{1,\dots,K\}, \hspace{5mm} j \in \{1,\dots,r_0\}.\]
 
A link $\sigma$ is a bijection between the labels of $T^-$ and the labels of $T^+$ taking a label $iq$ in the $p^{th}$ row of $T^-$ to a label $ip$ in the $q^{th}$ row of $T^+.$ Define an ordering on the labels by setting $ip < jq$ if $i <j$ or $i=j$ and $p < q.$ In this way, we extend the notion of semi-standard to these tableaux with double entries: a linked pair is \emph{semi-standard} if both $T^+$ and $T^-$ are semi-standard. 
\begin{thm}\label{thm:intro2} Let $Q$ be the Kronecker quiver. Then pairs of primitive semi-standard linked tableaux index a SAGBI basis for the algebra of quiver semi-invariants.
\end{thm} 
For the definition of \emph{primitive} semi-standard linked tableaux pairs see Definition \ref{defn:prim}. It is not clear from the combinatorics that the set of primitive semi-standard linked tableaux is finite. We do not know of any examples where it is not finite, and we show that it holds for the following family: 
\begin{thm} \label{thm:intro3} Let $Q$ be the Kronecker quiver with dimensions $r_0=r_1=2$. Then this SAGBI basis is finite.
\end{thm}
%\begin{eg} Let $Q$ be the Kronecker quiver with $k=3$ and $r_0=r_1=2.$ Then if $(T^-,T^+)$ is a primitive semi-standard linked pair, $T^+$ is one of the following tableaux: 
%\[\ytableausetup{centertableaux}
%\begin{ytableau}
%i1 \\
%j2\\
%\end{ytableau}  \hspace{2mm} 1 \leq i \leq j \leq 3,
%\hspace{10mm}
%\begin{ytableau}
%11 &22\\
%21 & 32\\
%\end{ytableau}\]
%\end{eg}

\paragraph{\textbf{Plan of the paper}} In \S \ref{sec:background}, we give the necessary background on quiver moduli spaces and quiver semi-invariants. In \S \ref{sec:orbited}, we describe the DZ semi-invariants and show in \S \ref{sec:straightening} that they satisfy straightening laws arising from Pl\"ucker coordinates. In \S \ref{sec:Kronecker} and \S \ref{sec:sagbik}, we discuss the SAGBI basis of the Kronecker quiver, and describe it explicitly in some small examples.  In \S \ref{sec:2case} we prove that this basis is finite in the $r_0=r_1=2$ case. In the last section of the paper, \S \ref{sec:example}, we use the SAGBI basis to produce a toric degeneration of Fano quiver moduli space, and give a conjectural Laurent polynomial mirror for this variety. 
\section{Background}\label{sec:background}
As in the introduction, let $Q=(Q_0,Q_1)$ be a quiver. Label the vertices $Q_0=\{0,1,\dots,\rho\}$.  All quivers in this paper are assumed to be acyclic, so we can assume that if $a \in Q_1,$ then $s(a)<t(a).$ We fix a dimension vector $\br=(r_0,\dots,r_\rho) \in \NN^{\rho+1}.$ Recall that the set of representations of this quiver of dimension $\br$ is 
\[\Rep(Q,\br)=\oplus_{a \in Q_1} \Hom(\CC^{r_{s(a)}},\CC^{r_{t(a)}}).\]
The change of basis action is given by $\GL(\br)=\prod_{i=0}^\rho \GL(r_i)$. 

Given $\theta \in \chi(\GL(\br))$, there is a GIT quotient
\[M_\theta(Q,\br):=\Rep(Q,\br)/\!/_\theta G.\]
The character group $\chi(\GL(\br))$ is naturally identified with $\ZZ^{\rho+1}$. When $r_0=1$ and 
\[\theta=(-\sum_{i=1}^\rho r_i,1\dots,1)\]
 under this identification, the resulting GIT quotient is a smooth projective fine moduli space, called a \emph{quiver flag variety} \cite{craw}. Examples of quiver flag varieties include type A flag varieties and Grassmannians. 
\subsection{Quiver semi-invariants}
The group $\GL(\br)$ acts naturally on the elements of $\CC[\Rep(Q,\br)].$ If $Q$ is acyclic, there are no non-trivial $\GL(\br)$-invariant polynomials in the ring $\CC[\Rep(Q,\br)].$ Instead, we can consider semi-invariants of this group action.
We choose coordinates on $\CC[\Rep(Q,\br)]$:
\[x^a_{ij}, \hspace{2mm} a \in Q_1, i \in \{1,\dots,r_{s(a)}\}, j \in \{1,\dots,r_{t(a)}\}.\] 
\begin{mydef} Let $f$ be a polynomial in $\CC[\Rep(Q,\br)].$ Then $f$ is a \emph{semi-invariant} of weight $\alpha \in \chi(\GL(\br))$ if, for all $g \in \GL(\br)$,
\[g \cdot f= \chi_\alpha(g) f.\] The ring of semi-invariants is denoted 
\[\SI(Q,\br),\]
and is graded by $\chi(\GL(\br))$. For $\alpha \in \chi(\GL(\br))$, we denote the $\alpha$-graded piece as
\[\SI(Q,\br)_\alpha.\]
\end{mydef}

Apart from semi-invariants being a natural construction to consider in invariant theory, they also play an important role in the study of the related GIT quotients. The GIT quotient $M_\theta(Q,\br)$ is constructed as
\[\Proj(\oplus_{k \in \ZZ_{\geq 0}} \SI(Q,\br)_{k\theta}).\]

As described in the introduction, given a character $\alpha \in \chi(\GL(\br))$, we can consider 
\[ \Rep(Q,\br)^{ss} \times \CC/ \GL(\br) \to \Rep(Q,\br)^{ss}/\GL(\br)=\Rep(Q,\br)/\!/_\theta\GL(\br).\]
When the GIT quotient is smooth, by Kempf's descent lemma this is a line bundle, which we denote $L_\alpha$. For $M_\theta(Q,\br)$ with sufficiently good properties, this induces isomorphisms
\[\chi(\GL(\br)) \cong \Pic(M_\theta(Q,\br)),\]
\[\Gamma(M_\theta(Q,\br),L_\alpha) \cong \SI(Q,\br)_\alpha,\]
and hence
\[\Cox(M_\theta(Q,\br)) \cong \SI(Q,\br).\]
These isomorphisms hold when the unstable locus has codimension at least 2 and there are no strictly semi-stable points \cite{crawunpub}. In particular, they hold for quiver flag varieties \cite{craw}.

Studying the semi-invariant ring can give information on embeddings of $M_\theta(Q,\br)$ arising from ample line bundles. Another application, as will be discussed, is in finding toric degenerations.
\subsection{Abelianization}
Since a quiver GIT quotient is of the form $\VG$, we can consider its abelianization $\VT$, where $T \subset G$ is the diagonal maximal torus. In \cite{kalashnikov}, this is shown to be another quiver GIT quotient: the one associated to the \emph{abelianized} quiver $Q^{ab}$ of $Q$, with dimension vector $(1,\dots,1)$. 

The abelianization $Q^{ab}$ of a quiver $Q$ is formed by replacing a vertex $p$ with dimension $r_i$ with $i$ separate vertices $p_1,\dots,p_{r_i}$ all of dimension one. We call $p_i$ a \emph{lift} of $p$ to $Q^{ab}$. If there is an arrow $a: p \to q$ between vertex $p$ and $q$ in $Q$, there is an arrow $a_{ij}:p_i \to q_j$ in $Q^{ab}$ for each lift $p_i$ of $p$ and $q_j$ of $q$.   For example, consider the quiver
\begin{center}
\includegraphics[scale=0.5]{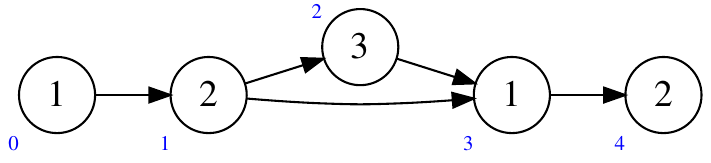}.
\end{center}
where the number inside a vertex $i$ indicates the dimension $r_i$ associated to $i$. The abelianization of this quiver is
\begin{center}
\includegraphics[scale=0.5]{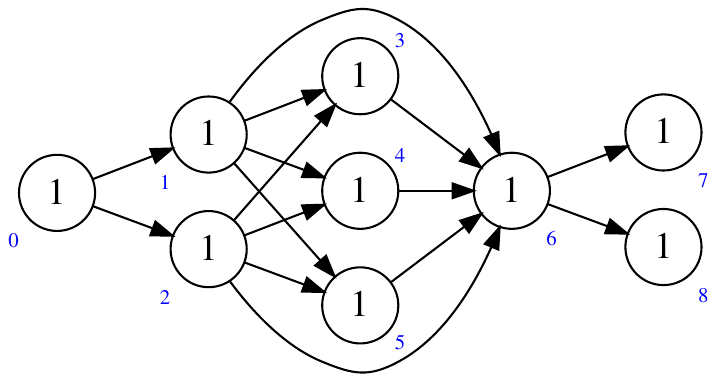}.
\end{center}
The two vector spaces
\[\Rep(Q,\br) =\Rep(Q^{ab},(1,\dots,1))\]
are naturally identified.  A lift of an arrow $a:p\to q$ in $Q$ to an arrow $a_{ij}: p_i \to q_j$ in $Q^{ab}$ corresponds to the coordinate $x^a_{ij}$, in $\CC[\Rep(Q,\br)]$.

Denote the $T$-weight of $x^a_{ij}$ as
\[D^T_{a_{ij}} \in \chi(T).\]
The monomial $\prod_{a_{ij} \in S} x^a_{ij}$ for some collection of arrows $S$ has weight
 \[\sum_{a_{ij} \in S} D^T_{a_{ij}}.\]
 \begin{mydef} Let $S$ be a collection of arrows in $Q^{ab}$. We say that $S$ is \emph{Weyl-invariant} if the difference between the number of arrows in and out of each vertex is constant over the lifts of $p$ for all $p \in Q_0.$ That is, for every vertex $p \in Q_0$,  $n_{p_i}=n_{p_j}$ for all lifts $p_i$ and $p_j$ of $p$ to $Q^{ab}$, where
 \[n_{p_i}:=\#\{a \in S: t(a)=p_i\}-\#\{a \in S: s(a)=p_i\}.\]
 \end{mydef}
The set of monomial semi-invariants of $T$ with Weyl-invariant weight are in one-to-one correspondence with the Weyl-invariant arrow sets. Given an arrow set, we can partition it into paths. 
 \begin{mydef} Let $S$ be a collection of paths in $Q^{ab}$. We say that $S$ is Weyl-invariant if for every vertex $p \in Q_0$, both the number of paths in and the number of paths out of each vertex is constant over the lifts of $p$ for all $p \in Q_0.$ \end{mydef}
 Note that Weyl-invariance for path sets is defined to be stronger than Weyl-invariance for arrow sets. 
 \begin{lem} For every Weyl-invariant arrow set $S$, there is at least one Weyl-invariant path set $T$ such that the collection of all arrows in all paths of $T$ is $S$. 
 \end{lem}
 \begin{proof}
 Consider the lifts of a vertex $p_1,\dots,p_{r_i}$ of a vertex in $Q$.  By connecting the arrows coming in and out each vertex until there are only $n_{p_i}$ paths that have source or target this vertex, we can build a path set that is Weyl-invariant (and in fact bipartite). 
 \end{proof}
 
An arrow $a \in Q_1$ is associated to the matrix of coordinates on $\Rep(Q,\br)$:
\[M_a:=(x^a_{ij}) \in \Mat(r_{t(a)} \times r_{s(a)}).\]
For a path $P$ in $Q$, let 
\begin{equation}\label{eq:pathmatrix}
M_P:=\prod_{a \in P} M_a.
\end{equation}
\section{Determinantal semi-invariants}\label{sec:orbited}
\subsection{The Grassmannian}
Before describing the general case, we consider the simplest example: the Grassmannian. In this section we produce semi-invariants of the Grassmannian by introducing an action of the Weyl group on tableaux. 

The Grassmannian $\Gr(n,r)$ is the quiver flag variety associated to the quiver
\begin{center}
\includegraphics[scale=0.5]{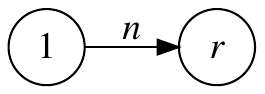}
\end{center}

In order to quotient by an effective group action, we can construct the Grassmannian as the GIT quotient $\Mat(r\times n)/\!/\GL(r)$ (by choosing an identification of $(\CC^*\times\GL(r))/\CC^* \simeq\GL(r)$).

The abelianization of this quiver has one source vertex and $r$ target vertices, each with $n$ arrows from the source. For example, the abelianization of the quiver of $\Gr(5,3)$ is 
\begin{center}
\includegraphics[scale=0.25]{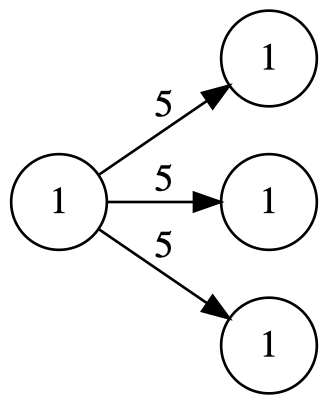}.
\end{center}
 The Weyl group is $\Sym_r$ and permutes the $r$ target vertices. Let $x^i_j$ be the coordinate associated to the arrow $a^i_j$, where $a^i_j$ is the $i^{th}$ arrow into the $j^{th}$ vertex. A collection $S$ is W-invariant if there is the same number of arrows to each target vertex. 
 
 Consider collections with one arrow into each target vertex, i.e. sets $\{a^{i_1}_1,\dots,a^{i_r}_r\}$ for any choice $i_1,\dots,i_r\in \{1,\ldots,n\}.$ This is associated to the monomial 
 \[\prod_{j=1}^r x^{i_j}_j .\]
 The Weyl group acts on this monomial by permuting the lower indices. To make this a semi-invariant, we certainly need it to be semi-invariant under the Weyl group action, so we can consider the sum
  \[\sum_{\sigma \in \Sym_r} \sign(\sigma) \sigma \cdot \prod_{j=1}^r x^{i_j}_j =\sum_{\sigma \in \Sym_r} \sign(\sigma)  \prod_{j=1}^r x^{i_j}_{\sigma(j)}=\det([x^{i_j}_k]). \]
 This sum is now not just semi-invariant for the Weyl group, but actually for all of $G$, as these are of course the Pl\"ucker coordinates of the Grassmannian. One can visualize this on $Q^{ab}$ by drawing the arrows, and then letting $W$ act on them in the natural way. 
 
 Now consider collections with two arrows into each target vertex. Label such a set 
 \[\{a^{i_1}_1,\dots,a^{i_r}_r,a^{k_1}_1,\dots,a^{k_r}_r\}.\]
  We get the product of monomials 
  \[\prod_{j=1}^r x^{i_j}_j  x^{k_j}_j.\]
  We could symmetrize this monomial using the Weyl group action:
\[\sum_{\sigma \in \Sym_r} \sign(\sigma)^2 \sigma \cdot \prod_{j=1}^r x^{i_j}_j  x^{k_j}_j,\]
but this is not a $G$-semi-invariant. However, notice that there exists a natural action of two copies of $W$: one can act by one copy of $W$ on the $a^{i_j}_j$ and one copy on the $a^{k_j}_j.$ Doing this produces a semi-invariant, as what we obtain is a product of two weight one semi-invariants. 

However, we have used more data than just the set of arrows: we have partitioned the arrows into two separate sets, each of which are Weyl-invariant. To record this, we use a tableau. Each column in the tableau is of length $r$, and corresponds to one set that the Weyl group acts on:
\[ \ytableausetup{centertableaux}
\begin{ytableau}
i_1 & k_1\\
i_2 & k_2 \\
\vdots &\vdots \\
i_r & k_r\\
\end{ytableau}
\]
To produce semi-invariants of weight $l$, we write down an $r \times l$ tableau. Each such tableau $T$ gives a monomial $m_T$: a box labeled $i$ in the $j^{th}$ row contributes a variable $x^i_j.$ There is a $W^{\times l}$ action on tableaux of shape $r \times l$, where the first copy of $W$ permutes the labels in the first column, the second in the second column, and so forth.  The semi-invariant associated to $T$ is
\[\sum_{\underline{w} \in W^{\times l}} \sign(\underline{w}) m_{\underline{w} \cdot T}.\]
 
\subsection{The general case}
Let $Q$ be an acyclic quiver with dimension vector $\br$, and let $T$ be the maximal diagonal torus in the group $\GL(\br)$. Let $Q^{ab}$ denote the abelianization of the quiver.  Recall that $T$-semi-invariant monomials of Weyl-invariant weights are in one-to-one correspondence with Weyl-invariant collections of arrows. The goal of this section is to associate semi-invariants of $Q$ to Weyl-invariant collections of arrows in $Q^{ab}$. As in the Grassmannian case, we need tableaux and actions of copies of the Weyl group on tableaux to do this.

We do this in two steps:
\begin{enumerate}
\item Partition the arrows into a set of paths $S_P$.
\item Assign to the set of paths two tuples of partitions that are linked.
\end{enumerate}
\subsubsection{Partitioning the arrows} 
Fix a Weyl-invariant path set $S$. The monomial associated to the underlying arrow set has a Weyl-invariant weight, and let 
\[ (w_p)_{p \in Q_0}= (w^+_p)- (w^-_p), \hspace{3mm} w^\pm_p \ge 0\]
be this weight in $\chi(\GL(\br))$ under the identification $\chi(\GL(\br)) \cong \chi(T)^W.$ Then 
\begin{itemize}
\item $w^-_p=\#\{P \in S: s(P)=p_i \text{ for some lift $p_i$ of $p$}\}$, 
\item $w^+_p=\#\{P \in S: t(P)=p_i \text{ for some lift $p_i$ of $p$}\}$.
\end{itemize}
\begin{lem} Let $S$ be a Weyl-invariant path set with weight $(w_p)_{p\in Q_0}\in \chi(G)$, using the identification $\chi(G)\cong \chi(T)^W$. Then $\sum_{p} r_p w_p=0.$
\end{lem}
\begin{proof}
We observe that
 \[\sum_{p\in Q_0} r_p w_p=\sum_{p_i \in Q^{ab}} n_{p_i},\]
where $n_{p_i}$ is, as before,
 \[n_{p_i}:=\#\{P \in S: t(P)=p_i\}-\#\{P \in S: s(P)=p_i\}.\]
Since each path has exactly one source and one target, it follows that this sum is $0.$
\end{proof}
\subsubsection{Linked pairs of tableaux}
We now define \emph{linked pairs} of tableaux, which will play an analogous role to the tableaux in the Grassmannian case. Here, we need to have pairs of tableaux because there exist non-trivial Weyl group actions on both the source and target of every arrow. 

Let $\mathcal{P}$ be the set of all paths in the quiver $Q$. Choose a labeling of this set $\{P_1,\dots,P_N.\}$ Given a path $P$ in $Q^{ab}$, we can project it onto a path in $Q.$

Choose any weight 
\[ (w_p)_{p \in Q_0}= (w^+_p)- (w^-_p), \hspace{3mm} w^\pm_p \ge 0\]
satisfying $\sum_{p} r_p w_p=0.$ For this weight, we define two tuples of partitions: \emph{target partitions}
\[(r_p \times w_p^+)_{p \in Q_0}\]
and \emph{source partitions},  
\[(r_p \times w_p^-)_{p \in Q_0}.\]
\begin{rem} Note that for a particular $p \in Q_0$, both $w_p^+$ and $w_p^-$ may be non-zero. 
\end{rem}
\begin{lem} The total number of boxes in the source partitions equals the total number of boxes in the target partitions.
\end{lem}
\begin{proof}
The total number of boxes in the source partitions is $\sum_{p} r_p w_p^-$. The total number of boxes in the target partitions is $\sum_{p} r_p w_p^+$. Then 
\[0=\sum_{p\in Q_0} r_p w_p=\sum_{p\in Q_0} r_p w_p^+-\sum_{p\in Q_0} r_p w_p^-\] and the result follows.
\end{proof}

Next, we consider tableaux with shapes given by the target and source partitions. That is, let  $(\underline{T}^+,\underline{T}^-)$ be two tuples of tableaux satisfying
\[\underline{T}^+=(T^+_p)_{p \in Q_0}, \hspace{5mm} \sh(T^+_p)=r_p \times w_p^+,\] 
and
\[\underline{T}^-=(T^-_p)_{p \in Q_0}, \hspace{5mm} \sh(T^-_p)=r_p \times w_p^-.\] 
The labels of $\underline{T}^\pm$ are given by pairs of integers:
\begin{enumerate}
\item $T^+_p$ is filled by entries of the form $ij$, where $i$ is the index of a path $P_i$ such that $t(P_i)=p$ and $j \in \{1,\dots,r_{s(P_i)}\}.$ 
\item$T^-_p$ is filled by entries of the form $ij$, where $s(P_i)=p$ and $j \in \{1,\dots,r_{t(P_i)}\}.$
\end{enumerate}
\begin{mydef}\label{def:linked} Let $\underline{T}^+$ and $\underline{T}^-$ be a pair of tableaux tuples filled as above. We say that $\underline{T}^+$ and $\underline{T}^-$ form a \emph{linked pair} if there is a bijection $\sigma$ from the set of boxes in $\underline{T}^+$ to the set of boxes in $\underline{T}^-$ satisfying the following. If $\sigma$ takes box $B_1$ in $T^+_p$ with label $ij$ to box $B_2$  in $T^-_q$ with label $kl$, then 
\begin{itemize}
\item $i=k$, so that $P_i$ is a path from $q \to p$, and $j \in \{1,\dots,r_{q}\}$, $l \in \{1,\dots,r_{p}\}.$
\item $B_1$ appears in the $l^{th}$ row of $T^+_p$, a tableau of shape $r_p \times w^+_p.$
\item $B_2$ appears in the $j^{th}$ row of $T^-_q$, a tableau of shape $r_q \times w^-_q.$
\end{itemize} 
\end{mydef}

\begin{eg}\label{eg:semiinv}
Consider the quiver
\[\begin{tikzpicture}[scale=1,every node/.style={scale=0.8}]
\filldraw (0,0) circle (3pt);
\filldraw (3,0) circle (3pt);
\filldraw (6,0) circle (3pt);
\draw[->] (0,0) to [bend right=15, above,"1"] (2.9,-0.1);
\draw[->] (0,0) to [bend left=15, "2",near end] (2.9,0.1);
\draw[->] (3,0) to ["7"] (6,0);
\draw[->] (0,0) to [bend right=20,above,"5"] (5.9,-0.1);
\draw[->] (0,0) to [bend left=20,above,"4"] (5.9,0.1);
\draw[->] (0,0) to [bend right=40,above,"6"] (6,-0.1);
\draw[->] (0,0) to [bend left=40,above,"3"] (6,0.1);
\end{tikzpicture}\]
 with dimension vector $(2,2,3)$. Label the arrows from $1,\dots,7$, ordered by their source vertex and then by their target, as drawn above. There are two paths given by concatenating arrows from vertices $0$ to $1$ with the arrow from $1$ to $2$, which we label $8,9$.

Consider the weight $[-2,-1,2]$. Let $\underline{T}^+=(\emptyset,\emptyset, T^+_2$), where
  \[T^+_2= \ytableausetup{centertableaux}
\begin{ytableau}
3 1 & 5 2\\
4 1 & 7 1\\
7 2 &  6 2\\
\end{ytableau}
.\]
Let $\underline{T}^-=(T^-_0,T^-_1,\emptyset)$, where 
  \[T^-_0= \ytableausetup{centertableaux}
\begin{ytableau}
3 1 & 4 2\\
6 3 & 5 1\\
\end{ytableau}
 \hspace{10mm} 
T^-_1 = \ytableausetup{centertableaux}
\begin{ytableau}
7 2 \\
7 3 \\
\end{ytableau}\]
This is a linked pair, with a unique link. 
\end{eg}

\subsubsection{From path sets to linked pairs of tableaux}
Next, we explain how to go from path sets to linked tableaux pairs. Let $S$ be a Weyl-invariant path set of $Q^{ab}$. As above, let $w$ be the Weyl-invariant weight of the associated monomial, and write 
\[w=(w_p)_{p \in Q_0}= (w^+_p)- (w^-_p)\]
where $w^+_p$ is the number of the paths in $S$ into a vertex in $Q^{ab}_0$ lifting $p$  and $w^-_p$ is the number of paths in $S$ out of a vertex in $Q^{ab}_0$ lifting $p$. Since the path set is Weyl-invariant, the choice of lifts doesn't matter. 

The number of paths in $S$ is the same as the total number of boxes in the target partitions for this weight, which is the same as the total number of boxes in the source partitions.  We use the paths to label the source and target partitions corresponding to the weight $w$ as follows. A path $P \in S$ projects to a path $P_i$ in $Q$.  Then the source of $P$ is a vertex $q_j, j \in \{1,\dots,r_q\}$ and the target of $P$ is a vertex $p_l, l \in \{1,\dots,r_p\}$.  The path $P$ then gives a label $i j$ which can be placed in $l^{th}$ row of the $q^{th}$ source tableau, and a label $i l$ which can be placed in the $j^{th}$ row of the $p^{th}$ target tableau. 

That is, for the path set $S$, we consider tuples of tableaux $\underline{T}^+=(T^+_p)$ and $(T^-_q)$ that satisfy the following two conditions:
\begin{itemize}
\item The labels in the $l^{th}$ row of $T^+_p$ are the set 
\[\{ij: P \in S \text{ with projection } P_i \text{ and } s(P)=q_j, t(P)=p_l.\}\]
\item The labels of the $j^{th}$ row of $T^-_q$ are the set
\[\{il: P \in S \text{ with projection } P_i \text{ and } s(P)=q_j, t(P)=p_l.\}\]
\end{itemize}
Then the two tuples are naturally a linked pair, using the path set $P$: the label in the target tableau $\underline{T}^+$ corresponding to $P$ is linked to the label in the source tableau $\underline{T}^-$ corresponding to $P$. 

\begin{eg}
\label{eg:semiinvpaths} Consider the quiver $Q$ and the linked pair of tableaux $\underline{T}^\pm$ from Example \ref{eg:semiinv}. Index the paths in the abelianization of $Q$ by $P^{jl}_{i}$, where $P^{jl}_{i}$ projects onto the path $P_i$ in $Q$, where the source of this path is the $j^{th}$ lift of $s(P_i)$, and its target is the $l^{th}$ lift of $t(P_i)$. The path set underlying the linked pairs of tableaux $\underline{T}^\pm$ is
\[ \{P^{11}_3, P^{12}_4, P^{23}_7, P^{21}_5, P^{12}_7,P^{23}_6\},\] which is Weyl-invariant.
\end{eg}

\subsubsection{Weyl-type group actions on linked pairs of tableaux}
As in the case of the Grassmannian, choosing a tuple of tableaux corresponding to a path set $S$ allows us to construct a $G$-semi-invariant out of the monomial with $W$-invariant weight that is given by $S$. The structure of linked pairs of tableaux allows us to define a more refined action of (copies of the) Weyl group on the monomial.

Let $\underline{T}^\pm$ be a linked tableaux pair, linked by $\sigma.$  Let
 \[ \Sym^+= \prod_{p \in Q_0} \Sym_{r_p}^{w^+_p}.\]
 %and denote elements of $\Sym^+$ as $\alpha^+=(\alpha_{p,i})_{p \in Q_0, i \in \{1,\dots,w^+_p\}}.$  
Let
 \[ \Sym^-= \prod_{p \in Q_0} \Sym_{r_p}^{w^-_p}.\]
 %and denote elements of $\Sym^-$ as $\alpha^-=(\alpha_{p,i})_{p \in Q_0, i \in \{1,\dots,w^-_p\}}.$ 
 Set $\sign(\alpha^\pm)$  to be the product of the signs of the permutations in $\alpha^\pm \in \Sym^\pm.$ 

Then both groups act on each of $\underline{T}^-$ and $\underline{T}^+.$ One action is given by permuting the rows, and one by permuting the labels, as we describe below.
\paragraph{\bf{Permuting the rows of $\underline{T}^\pm$}}
This is an action of $\Sym^\pm$ on $\underline{T}^\pm$.

There are $w^\pm_p$ columns of $T^\pm_p$, and $w^\pm_p$ copies of $\Sym_{r_p}$ in $\Sym^\pm.$ The action is defined by using the $l^{th}$ copy of $\Sym_{r_p}$ to permute the labels of the $l^{th}$ column of $T^\pm_p.$
\paragraph{\bf{Permuting the labels of $\underline{T}^\pm$}} This is an action of $\Sym^\pm$ on $\underline{T}^\mp$. Because $(\underline{T}^+,\underline{T}^-)$ are a linked pair, there are $r_p w_p^-$ labels $ij$ in $\underline{T}^+$ satisfying $s(P_i)=p.$ Each of the $w_p^-$ columns in $T^-_p$ corresponds to a set of labels in $\underline{T}^+$ of the form $\{i_1 1,\dots,i_{r_p} r_p\}$. The action of $\Sym^-$ on $T^+$ is defined by having the $l^{th}$ copy of $\Sym_{r_p}$ in $\Sym^-$ permute the second entry of the $l^{th}$ set of labels $\{i_1 1,\dots,i_{r_p} r_p\}$.
%%Consider the action of $\alpha^+ \in \Sym^+$ on $\underline{T}^+.$ This moves a box $B^+$ with label $ij$ in $T^+_p$ from row $l_1$ to $l_2$ in the same column. Under $\sigma$, box $B^+$ is mapped to some $B^-$ with label $i l_1$.The action of $\alpha^+$ on $\underline{T}^-$ is given by replacing the label $i l_1$ with label $i l_2.$ Similarly, $\Sym^-$ acts on the labels of $\underline{T}^+.$ 

Note that for all $\alpha^+ \in \Sym^+$,
\[(\alpha^+ \cdot \underline{T}^+,\alpha^+ \cdot \underline{T}^-)\]
is naturally a linked pair.
Similarly,
\[(\alpha^- \cdot \underline{T}^+,\alpha^- \cdot \underline{T}^-)\]
is a linked pair for all $\alpha^- \in \Sym^-.$

\begin{eg}
\label{eg:semiinvaction}
Consider the pair of linked tableaux $\underline{T}^\pm$ from Example \ref{eg:semiinv}, where the non-empty tableaux are:
 \[T^+_2= \ytableausetup{centertableaux}
\begin{ytableau}
3 1 & 5 2\\
4 1 & 7 1\\
7 2 &  6 2\\
\end{ytableau}
\]
and 
  \[T^-_0= \ytableausetup{centertableaux}
\begin{ytableau}
3 1 & 4 2\\
6 3 & 5 1\\
\end{ytableau}
 \hspace{10mm} 
T^-_1 = \ytableausetup{centertableaux}
\begin{ytableau}
7 2 \\
7 3 \\
\end{ytableau}\]

The group $\Sym^+=\Sym_3\times \Sym_3$ acts by permuting the columns of $T^+_2$. It also acts on the labels of $T_0^-$ and $T_1^-$ where the first copy of $\Sym_3$ acts by permuting the second digit of the labels in yellow, and the second copy of $\Sym_3$ acts by permuting the second digit of the labels in green:
  \[T^-_0= \ytableausetup{centertableaux}
\begin{ytableau}
*(yellow)3 1 & *(yellow)4 2\\
*(green)6 3 & *(green)5 1\\
\end{ytableau}
 \hspace{10mm} 
T^-_1 = \ytableausetup{centertableaux}
\begin{ytableau}
*(green)7 2 \\
*(yellow)7 3 \\
\end{ytableau}\]

\end{eg}

\begin{rem} Note that the actions of $\Sym^+$ and $\Sym^-$ on $T^+$ ($T^-$) commute. This follows from the definition: a label $ij$ in $T^+$ has its location changed by $\alpha \in \Sym^+$, and its second entry changed by $\beta \in \Sym^-$.

\end{rem}
%Then $\Sym^+$ acts on the labels of $\underline{T}^-$ and $\Sym^-$ on the labels of $\underline{T}^+$. This action is defined using $\sigma.$ Suppose we have a box $B^+$ in $T^+_p$ in the $(c^+)^{th}$ column of $T^+_p$ mapped under $\sigma$ to a box $B^_$  in the $(c^-)^{th}$ column of $T^-_q.$ Then given $\alpha^+ \in \Sym^+$, the label in $\alpha^+ \cdot T^-_q$ of box $B^-$ is given by applying $(q,c^+)$ component of $\alpha^+$ to $j$. Similarly, the label in $\alpha^- \cdot T^+_p$ in box $B^+$ is given by applying the $(p,c^-)$ component of $\sigma^-$ to $l$. 

\subsubsection{From linked pairs of tableaux to semi-invariants}
 $T$-semi-invariant monomials of Weyl-invariant weights are in one-to-one correspondence with Weyl-invariant arrow sets.  Given such an arrow set, we choose a Weyl-invariant path set, and for this path set, we choose a linked tableaux pair. The final stage is to explain how to produce a semi-invariant from a linked pair.   
 
Recall that for a path $P$ in $Q$, there is a matrix $M_P$ (see \eqref{eq:pathmatrix}) of size $r_{t(P)} \times r_{s(P)}.$  Let $(M_P)_{ij}$ be the $ij$ entry of $M_P$. To a tuple of tableaux $\underline{T}^+$, we associate 
\[\Mon(\underline{T}^+)=\prod (M_{P_i})_{lj},\]
where the product is over all labels $ij$ appearing in some $T^+_p$, and $l$ is the row in which this label appears. 

We analogously define 
\[\Mon(\underline{T}^-)=\prod (M_{P_i})_{lj},\]
where the product is over all labels $il$ appearing in some $T^-_p$, and $j$ is the row in which this label appears. 

The following lemma follows from the definition:
\begin{lem} If $\underline{T}^\pm$ is a linked pair, then
\[\Mon(\underline{T}^+)=\Mon(\underline{T}^-).\]
\end{lem}

\begin{mydef}
 Define
 \[f_{\underline{T}^\pm}=\sum_{\alpha^- \in \Sym^-,\alpha^+\in \Sym^+} \sign(\alpha^-)  \sign(\alpha^+) \Mon( \alpha^- \cdot \alpha^+ \cdot \underline{T}^+).\] 
 \end{mydef}
 \begin{rem}
  Since $\alpha^- \cdot \alpha^+ \cdot \underline{T}^+$ and $\alpha^- \cdot \alpha^+ \cdot \underline{T}^-$ are canonically linked, this can be equivalently written as 
   \[f_{\underline{T}^\pm}=\sum_{\alpha^- \in \Sym^-,\alpha^+\in \Sym^+} \sign(\alpha^-)  \sign(\alpha^+) \Mon( \alpha^- \cdot \alpha^+ \cdot \underline{T}^-).\] 
 \end{rem}
As written, it is not clear that this polynomial is a semi-invariant of $Q$. We show this in the following
\begin{thm}\label{thm:isinv} For any linked pair, $f_{\underline{T}^\pm}$ is a semi-invariant of $Q$.
\end{thm}

The proof will follow from a straightforward lemma. 

Let $C_i(M_P)$ be the $i^{th}$ column of $M_P$ and $R_i(M_P)$ the $i^{th}$ row. Given  a tableau $T^-_p$,  $D(T^-_p)$ is given by taking a product of determinants, where there is one determinant for each column of $T^-_p$. Each determinant is the determinant of matrix with rows taken from the path matrices showing up in the first index of the labels. The row contributed by a particular path matrix is determined by the second index in that label. More precisely, if the labels of the $l^{th}$ column of $T^-_p$ are
 \[i^l_m j^l_m, m=1,\dots,r_p,\]
 set
 \[D(T^-_p)=\prod_{l=1}^{r_p} \det([R_{j^l_m}(M_{P_{i^l_m}})]_{m=1,\dots,r_p}).\] 
 Set $D(\underline{T}^-)=\prod_{p \in Q_0} D(T^-_p).$
 
Similarly, given  a tableau $T^+_p$,  $D(T^+_p)$ is given by taking a product of determinants, where there is one determinant for each column of $T^+_p$. Each determinant is the determinant of a matrix with one column taken from each of the path matrices showing up in the first index of the labels in this tableau column. The column contributed by a particular path matrix is determined by the second index in that label. More precisely, if the labels of the $l^{th}$ column of $T^+_p$ are
 \[i^l_m j^l_m,\textup{ for } m=1,\dots,r_p,\]
 set
 \[D(T^+_p)=\prod_{l=1}^{r_p} \det([C_{j^l_m}(M_{P_{i^l_m}})]_{m=1,\dots,r_p}).\]
 As above, set $D(\underline{T}^+)=\prod_{p \in Q_0} D(T^+_p).$

\begin{lem}\label{lem:twoexp}  If $\underline{T}^\pm$ is a linked pair,
\[f_{\underline{T}^\pm} =  \sum_{\alpha^+ \in \Sym^+} \sign(\alpha^+) D(\alpha^+ \underline{T}^-)= \sum_{\alpha^- \in \Sym^-} \sign(\alpha^-) D(\alpha^- \underline{T}^+).\]
\end{lem}
\begin{proof} Follows from the Leibniz formula for the determinant and the fact that the $\Sym^+$ and $\Sym^-$ actions commute. 
\end{proof}

%Then we have, analogously,
% \begin{lem}  If $\underline{T}^\pm$ is a linked pair,
%\[f_{\underline{T}^\pm} =  \sum_{\alpha^- \in \Sym^-} \sign(\alpha^-) D(\alpha^- \underline{T}^+).\]
%\end{lem}
\begin{proof}[Proof of Theorem \ref{thm:isinv}]
Write $\GL(\br)=H^+ \times H^-$, where 
\[H^+=\prod_{p: w_p^+>0} \GL(r_p),\]
and 
\[H^-=\prod_{p: w_p^->0} \GL(r_p).\]
Then the two expressions of $f_{\underline{T}^\pm}$ from Lemma \ref{lem:twoexp} are, respectively, $H^+ \times \id$ and $\id \times H^-$ semi-invariant of weights $w^+$ and $-w^-.$ It follows that $f_{\underline{T}^\pm}$ is semi-invariant for the entire group $\GL(\br)$ with weight $w^+-w^-.$
\end{proof}

\begin{eg}
\label{eg:semiinv1}
Continuing from Example \ref{eg:semiinv}, if $M_i$ denotes the matrix corresponding to the $i^{th}$ path, then we can write $f_{\underline{T}^\pm}$ as 
\[\sum_{\sigma=(\sigma_i) \in S_2^3 }\sign(\sigma) 
\det[C_{\sigma_2(1)}(M_3) |C_{\sigma_3(1)}(M_4)| C_{\sigma_1(2)}(M_7)] 
\det[C_{\sigma_3(2)}(M_5) |C_{\sigma_1(1)}(M_7)| C_{\sigma_2(2)}(M_6)],\]
or equivalently
\[ f_{\underline{T}^\pm}=\sum_{(\tau_1,\tau_2)\in S_3^2} \sign(\tau_1) \sign(\tau_2)
\det{\begin{bmatrix} R_{\tau_2(2)}(M_7) \\ R_{\tau_1(3)}(M_7) \end{bmatrix}} 
\det\begin{bmatrix} R_{\tau_1(1)}(M_3) \\ R_{\tau_2(3)}(M_6) \end{bmatrix}
\det\begin{bmatrix} R_{\tau_1(2)}(M_4) \\ R_{\tau_2(1)}(M_5) \end{bmatrix}.\]
\end{eg}
\subsection{Comparison with Domokos--Zubkov semi-invariants} \label{sec:dzcomparison}
Domokos--Zubkov semi-invariants, as originally defined in \cite{domokos}, are only defined for bipartite quivers. In the case where $Q$ is bipartite, using Proposition 5.4 of \cite{domokos}, it is straightforward to check that the semi-invariants $f_{\underline{T}^\pm}$ are exactly the semi-invariants defined there (up to multiplication by a scalar). Note that because we work over characteristic 0, the part of data defined by $\bar{\lambda}$ in the definition of DZ semi-invariants only changes the semi-invariants by multiplication by a scalar.

The semi-invariants $f_{\underline{T}^\pm}$ are the natural generalization of DZ semi-invariants to the general quiver case. We simply allow the entries of the matrices associated to paths, rather than just the entries of the matrices associated to arrows. However, the proof from \cite{domokos} that these span $\SI(Q,\br)$ only applies in the bipartite case. In \cite{domokos}, the authors extend their results to all quivers by considering a particular doubling of any quiver to a bipartite quiver. However, the functor from this quiver to the original is not suited to our purposes.  

To show that the $f_{T^\pm}$ span, we instead use the work of Schofield--van den Bergh \cite{schofieldbergh} which consider a slightly different generating set than \cite{domokos}.  In particular, they show that for any quiver $Q$ and dimension vector $\br$, the algebra of semi-invariants can be spanned by the semi-invariants of bipartite quivers (when $Q$ is acyclic, which we have assumed for all quivers in this paper). More precisely, one considers all bipartite quivers $Q'$ with functors $F: \add(Q') \to \add(Q)$ taking arrows in $Q'$ to paths in $Q$; then $\SI(Q,\br)$ is spanned by $F(\SI(Q',\br'))$ ranging over all such $(Q',\br)$.  This is a consequence of the proof of \cite[Theorem 2.3]{schofieldbergh}; see also the discussion above  \cite[Corollary 2.4]{schofieldbergh}.\footnote{In fact, the authors of \cite{schofieldbergh} show something stronger: that it is sufficient to only consider a subset of bipartite quivers and some of their semi-invariants. These give the `standard' semi-invariants defined there.}  Such functors act on the semi-invariants by replacing arrow matrices with the matrices associated to their image paths. Therefore, DZ semi-invariants of $Q'$ (which span its semi-invariant algebra, by \cite{domokos}) are mapped to $f_{T^\pm}$ semi-invariants of $Q$. We therefore obtain:
\begin{cor} The $f_{T^\pm}$ span the semi-invariant ring. 
\end{cor}

 \section{Straightening laws}\label{sec:straightening}
 A tableau is called semi-standard if it is weakly increasing along rows and strictly increasing along columns. 
 
 A foundational fact for Grassmannians and flag varieties is that the semi-invariants indexed by semi-standard tableaux span all semi-invariants. Every leading term of \emph{any} semi-invariant is the leading term of a semi-invariant coming from a semi-standard tableau. Moreover, the leading term of a semi-standard tableau is simply the monomial read off of that tableau. These results are proven by showing that the semi-invariants of these varieties satisfy \emph{straightening laws}: every monomial in the Pl\"ucker coordinates that is not semi-standard can be written as sum of monomials that are semi-standard. For example, the semi-invariant corresponding to the non-semi-standard tableau
 \[ \ytableausetup{centertableaux}
\begin{ytableau}
1 & 2\\
4 & 3 \\
\end{ytableau}
\]
can be written as a linear combination of the semi-invariants of the semi-standard tableaux 
\[\begin{ytableau}
1 & 3\\
2 & 4 \\
\end{ytableau} \hspace{3mm} \begin{ytableau}
1 & 2\\
3 & 4 \\
\end{ytableau}
\]
We now investigate the extent of which this holds for linked tableaux pairs. We will see each such relation gives rise to a relation for linked tableaux pairs; however, the resulting condition on leading terms is weaker.
\subsection{Straightening laws for Pl\"ucker coordinates}
Suppose $A$ is an $r \times N$ matrix. For $I=\{i_1,\dots,i_r\}, i_j \in  \{1,\dots,N\}$, let $p_I$ denote the $r \times r$ minor corresponding to the columns indexed by $I$. The $p_I$ are Pl\"ucker coordinates, and it is well-known that they satisfying straightening laws. That is, if a set of $k$ Pl\"ucker coordinates $I_1,\dots,I_k$ does not correspond to a semi-standard tableau, then
\begin{equation} \label{eq:plstr} p_{I_1} \cdots p_{I_k}=\sum a_{J_1, \dots, J_k} p_{J_1} \cdots p_{J_k},\end{equation}
where each set $J_1,\dots,J_k$ satisfies:
\begin{itemize}
\item The associated $r \times k$ tableau is semi-standard.
\item The $J_1,\dots,J_k$ are a re-partitioning of the elements in the union of $I_1,\dots,I_k.$
\item Under a term ordering on the Pl\"ucker coordinates, the monomial on the left hand side has lower term order than all monomials on the right hand side. 
\end{itemize}
Viewing the data consisting of tuples of subsets as tableaux, we can write relations like \eqref{eq:plstr} as 
\[p_T=\sum_{\alpha \in \Sym(r k)} a_\alpha p_{\alpha \cdot T},\]
where $\alpha$ ranges over permutations of the labels in $T$. Note, we allow permutations that mix the columns and rows, and only some of the coefficients are non-zero. 

These relations are true for any matrix. Fix a vertex $p$ in the quiver, and consider the matrix $A^+_p$ formed by concatenating from left to right the matrices $M_P$ for all paths $P$ such that $t(P)=p.$ We can analogously form $A^-_p$, by concatenating the matrices from top to bottom. 

Now consider a linked pair $(\underline{T}^+,\underline{T}^-)$. Fix a vertex $p$ such that $T^+_p$ is non-empty. Recall that
 \[D(T^+_p)=\prod_{l=1}^{r_p} \det([C_{j^l_m}(M_{P_{i^l_m}})]_{m=1,\dots,r_p}),\] where the labels of the $l^{th}$ column of $T^+_p$ are
 \[i^l_m j^l_m, m=1,\dots,r_p.\]
Thus, $D(T^+_p)$ is a product of minors of $A^+_p.$ Any straightening law of the form above can then be applied to $D(T^+_p)$. Of course, when we act by a permutation $\alpha$, blocks can change rows; this means that $\underline{T}^-$ and  $\alpha \cdot \underline{T}^+$ are no longer linked.  We write $\alpha \cdot \underline{T}^-$ for the induced action on the tableaux in $\underline{T}^-$ given by changing the second entry of the labels. That is, we have a linked pair,
\[\alpha \cdot (\underline{T}^\pm):=(\alpha \cdot \underline{T}^+, \alpha \cdot \underline{T}^-).\]
\begin{lem} Suppose there exists a straightening relation giving 
\[D(T^+_p)=\sum_{\alpha} a_\alpha D(\alpha \cdot T^+_p).\]
Then for any $\beta \in \Sym^-,$ we have a relation
\[D((\beta \cdot T^+)_p)=\sum_{\alpha} a_\alpha D((\beta \cdot \alpha \cdot T^+)_p).\]
\end{lem}
\begin{proof} Using the straightening law, there is a relation:
\[D((\beta \cdot T^+)_p)=\sum_{\alpha} a_\alpha D(\alpha \cdot (\beta \cdot T^+)_p).\]
It suffices to show that 
\[\alpha \cdot (\beta \cdot T^+)_p=(\beta \cdot \alpha \cdot T^+)_p,\]
for all $\alpha$.

We first observe that the right hand side is well-defined. Note that the group $\Sym^-$ for the pair $\alpha \cdot (\underline{T}^\pm)$ is canonically identified with the same group for $(\underline{T}^\pm)$, as $\alpha$ does not change the shape of the tableau. Recall that each copy of $\Sym(r_q)$ in $\Sym^-$ corresponds to a column in $T^-_q$. If the column has labels $i_1 m_1,\dots,i_{r_q} m_q$, these correspond under the link to some set of labels in $\underline{T}^+$, $i_1 1,\dots, i_{r_q} q$, and $\Sym(r_q)$ acts by permuting the second index of these labels. Since $\alpha$ changes only the second index of the labels in $\underline{T}^-$, and does not move the columns, an element of $\beta$ acts in the same way on the labels $i_1 m_1,\dots,i_{r_q} m_q$ appearing in $\underline{T}^+$ as it does on the labels $i_1 m_1,\dots,i_{r_q} m_q$ appearing in $\alpha \cdot \underline{T}^+$, although they may not be in the same location. Finally, these actions commute, for the same reason that the actions of $\Sym^+$ and $\Sym^-$ commute. 
\end{proof}
\begin{thm}[Straightening Laws for semi-invariants of quivers]
Suppose there is a straightening relation giving 
\[D(T^+_p)=\sum_{\alpha} a_\alpha D(\alpha \cdot T^+_p)\] for some target $p$. 
Then
\[f_{\underline{T}^\pm}=\sum_{\alpha} a_\alpha f_{\alpha \cdot \underline{T}^\pm}.\]
In the same way, straightening laws applied to source tableaux give rise to straightening laws for semi-invariants. 
\end{thm}
\begin{proof}
From the lemma, we see that for any $\beta \in \Sym^-,$ there is a relation
\[D((\beta \cdot T^+)_p)=\sum_{\alpha} a_\alpha D((\beta \cdot \alpha \cdot T^+)_p).\]
Note that for $q \neq p$,
\[(\alpha \cdot T^+)_q =T^+_q,\]
as by definition, $\alpha$ acts only on $T^+_p$. 
Similarly, for any $\beta \in \Sym^-$,
\[(\beta \cdot \alpha \cdot T^+)_q =(\alpha \cdot \beta \cdot T^+)_q= (\beta\cdot T^+)_q,\]
where the first equality is given by the commutivity property established in the previous lemma. 

Therefore, for any $\beta \in \Sym^-,$ the relation
\[D(\beta \cdot T^+)=\sum_{\alpha} a_\alpha D(\beta \cdot \alpha \cdot T^+)\]
follows by multiplying 
 \[D((\beta \cdot T^+)_p)=\sum_{\alpha} a_\alpha D((\beta \cdot \alpha \cdot T^+)_p).\]
by
\[\prod_{q \neq p} D(\beta(T^+)_q).\]
By summing over $\beta \in \Sym^-$, we obtain the statement of the theorem.
\end{proof}
\begin{eg} \label{eg:semiinv2} We illustrate the theorem in an example, continued from Example \ref{eg:semiinv}. Consider the semi-invariant coming from the linked pair $\underline{T}^+=(\emptyset,\emptyset,T^+_2)$ and $\underline{T}^-=(T^-_0,T^-_1,\emptyset)$, where
  \[T^+_2= \ytableausetup{centertableaux}
\begin{ytableau}
3 1 & 5 2\\
4 1 & 7 1\\
7 2 &  6 2\\
\end{ytableau}
.\]
Let $\underline{T}^-=(T^-_0,T^-_1,\emptyset)$, where 
  \[T^-_1 = \ytableausetup{centertableaux}
\begin{ytableau}
7 2 \\
7 3 \\
\end{ytableau}
 \hspace{10mm} 
T^-_0= \ytableausetup{centertableaux}
\begin{ytableau}
3 1 & 4 2\\
6 3 & 5 1\\
\end{ytableau}
\]

There is a Pl\"ucker relation that gives
\[ \begin{ytableau}
3 1 & 4 2\\
6 3 & 5 1\\
\end{ytableau}=- \begin{ytableau}
3 1 & 5 1\\
4 2 & 6 3\\
\end{ytableau}
+\begin{ytableau}
3 1 & 4 2\\
5 1 & 6 3\\
\end{ytableau}
\]
where we have omitted the $D(-)$ notation. 
Each of the tableaux on the right hand side give rise to a new $\underline{T}^-$, which corresponds to  some $\underline{T}^+$ with $T^+_2$ respectively 
\[\ytableausetup{centertableaux} 
\begin{ytableau}
3 1 & 5 1\\
4 2 & 7 1\\
7 2 &  6 2\\
\end{ytableau} \hspace{5mm} \textup{ and } \hspace{5mm} 
\begin{ytableau}
3 1 & 5 2\\
4 1 & 7 1\\
7 2 &  6 2\\
\end{ytableau}
\]
Call the two new linked pairs $\underline{U}^\pm$ and $\underline{V}^\pm$ respectively. 

For an element $(\tau_1,\tau_2) \in S_3^2$, there is a relation
\[\ytableausetup{boxsize=3.5em}
\begin{ytableau}
3 \tau_1(1) & 4 \tau_1(2)\\
6 \tau_2(3) & 5 \tau_2(1)\\
\end{ytableau} = -\begin{ytableau}
3 \tau_1(1) & 5 \tau_2(1)\\\
4 \tau_1(2) & 6 \tau_2(3)\\
\end{ytableau}
+\begin{ytableau}
3 \tau_1(1) & 4 \tau_1(2)\\
5 \tau_2(1)\ & 6 \tau_2(3)\\
\end{ytableau}
\]
Summing over all of the elements of $S_3^2$, we obtain that 
\[f_{\underline{T}^\pm}=- f_{\underline{U}^\pm} +f_{\underline{V}^\pm}.\]
\end{eg}
Fix a quiver $Q$ and choose an order on the path set of $Q$. We define an order on the labels of the tableaux:
\[ij \leq kl \text{ if } i \leq k \text{ or } i= k \text{ and } j \leq l.\]

\begin{mydef} A pair of linked tableaux is \emph{semi-standard} if each tableau in $\underline{T}^\pm$ is semi-standard; i.e. it is strictly increasing along the columns and weakly increasing along the rows. 

A pair of linked tableaux is \emph{weakly} semi-standard if, for each tableau in $\underline{T}^\pm$, the first digits of the labels are weakly increasing along columns and weakly increasing along the rows.

We say a semi-invariant $f_{\underline{T}^\pm}$ is \emph{(weakly) semi-standard} if $\underline{T}^\pm$ is a (weakly) semi-standard pair of linked tableaux. 
\end{mydef}
\begin{eg} Both tableaux below are weakly semi-standard, while only the second tableau is semi-standard. 
\[
\ytableausetup{boxsize=normal}
\begin{ytableau}
1 2 & 1 1\\
1 3 & 4 2\\
\end{ytableau} \hspace{5mm} \begin{ytableau}
1 1 & 1 2\\
1 3 & 4 2\\
\end{ytableau}\]
\end{eg}
\begin{thm} Let $(Q,\br)$ be an acyclic quiver and dimension vector. Then weakly semi-standard quiver semi-invariants span $\SI(Q,\br)$.
\end{thm}
\begin{proof} Suppose we have a quiver semi-invariant corresponding to a pair of linked tableaux $\underline{T}^\pm$ that is not semi-standard. We can use the straightening laws on the target vertices to write it as a linear combination of linked tableaux whose target tableaux are all semi-standard. We can therefore assume that the target tableaux are all semi-standard. If one of the source tableau is not semi-standard, we can again use the straightening laws to re-write it as a sum of semi-standard tableaux -- however, in doing so we may have changed the labels of $\underline{T}^+$, so that we can no longer assert that they are semi-standard. Nevertheless, since the changes are all in the second digit of the labels, the tableaux remain weakly semi-standard. \end{proof}
\begin{rem} The proof of the theorem in fact shows something slightly stronger: to obtain a generating set for $\SI(Q,\br)$, we only need to consider pairs of linked tableaux $\underline{T}^\pm$ where $\underline{T}^+$ is semi-standard and $\underline{T^-}$ is weakly semi-standard. 
\end{rem}
\begin{eg} Continuing Example \ref{eg:semiinv2}, we can now write $f_{\underline{U}^\pm}$ and $f_{\underline{V}^\pm}$ as a sum of weakly semi-standard tableaux. For example, for $U^\pm$, note that by noting that
\[U_2^+=\ytableausetup{centertableaux,boxsize=1.5em} 
\begin{ytableau}
3 1 & 5 1\\
4 2 & 7 1\\
7 2 &  6 2\\
\end{ytableau}=-\ytableausetup{centertableaux} 
\begin{ytableau}
3 1 & 5 1\\
4 2 & 6 2\\
7 2 &  7 1\\
\end{ytableau}\]
If we call the tableau on the right hand side $W_2^+$, then $\underline{W}^+=(\emptyset, \emptyset, W_2^+)$ forms a weakly semi-standard linked pair together with
\[\underline{W}^-=\left( \hspace{1mm} \begin{ytableau}
3 1 & 5 1\\
4 2 & 6 2\\
\end{ytableau}, \begin{ytableau}
7 3 \\
7 3 \\
\end{ytableau}, \emptyset\right)\]
and $f_{\underline{W}^\pm}=-f_{\underline{U}^\pm}.$
\end{eg}
In the case of the Kronecker quiver, we now use these results to study the leading terms of semi-invariants. 
\section{The Kronecker quiver}\label{sec:Kronecker}
In this section, we fix a Kronecker quiver,
 \[\begin{tikzcd}
 0 \arrow[r,"K"] & 1
\end{tikzcd} \]
 and a dimension vector $\br=(r_0,r_1).$ The semi-invariant ring of its abelianization is
 \[\CC[\Mat(r_1 \times r_0)^{\oplus K}].\]
 We label generators $x^i_{jk}$, where $x^i_{jk}$ represents the $(j,k)$ entry of the $i^{th}$ matrix, so that the matrix corresponding to the $i^{th}$ arrow is $M_i=[x^i_{j,k}].$
 
 We can concatenate the $M_i$ two ways:
 \[A_R:=\begin{bmatrix} M_1 |&\cdots &| M_K\\
 \end{bmatrix},\]
  \[A_C:=\begin{bmatrix} M_1 \\ \vdots \\ M_K\\
 \end{bmatrix}.\]
 We fix a term order given by ordering the variables by going along the rows of $M_1$, then $M_2$ and so forth, and then using the lexicographic term order.
  \begin{lem}\label{lem:grlt} The leading term of a full-size minor of $A_R$ or $A_C$ is the monomial corresponding to the diagonal.
 \end{lem} 
 \begin{proof} The result follows because both $A_R$ and $A_C$ have the property that the entry of highest weight in a submatrix of any size is the entry in the upper left corner. 
 \end{proof}
 \begin{prop}\label{prop:lt} Let $f$ be a semi-invariant of the Kronecker quiver. Then the leading monomial of $f$, which we denote $\LM(f)$, is 
 \[\Mon(\underline{T}^+)=\Mon(\underline{T}^-)\]
 for some semi-standard pair of linked tableaux $T^\pm.$
 \end{prop}
 \begin{proof} Since $f$ is a semi-invariant of the Kronecker quiver, it is a semi-invariant with respect to both the $\GL(r_0)$ and $\GL(r_1)$ natural actions. It therefore can be written both as a polynomial in the $r_0 \times r_0$ minors of $A_C$ and as a polynomial in the $r_1 \times r_1$ minors of $A_R$. Using standard results for these rings, we know that the leading term of $f$ can be written as
  \[\Mon({T}^+)=\Mon({T}^-),\]
  for two semi-standard tableaux, $T^+$ and $T^-$. It suffices to show that this pair is linked. We construct a link using the leading term of $f$. Consider $x^i_{jk}$ appearing in the leading term of $f$. Then $x^i_{jk}$ corresponds to a label $i k$ appearing in the $j^{th}$ row of $T^+$, and also to a label $i j$ appearing in the $k^{th}$ row of $T^-$.  
 \end{proof}
 \begin{rem} Note that there are potentially multiple links supporting the same pair of semi-standard tableaux $T^\pm$; these lead to different semi-invariants. The ambiguity is up to multiple labels $i j$ appearing in a fixed row $k$ of $T^+$, or equivalently labels $i k$ appearing in row $j$ of $T^-.$ We say a link is \emph{semi-standard} if it maps the boxes labeled $i j$ in row $k$ of $T^+$ to the boxes labeled $i k$ in $T^-$ in an order preserving way. From now on, we include this in the definition of a semi-standard linked pair: both the tableaux and the link are required to be semi-standard.
 \end{rem}
 \begin{cor} Let $(T^-,T^+)$ be a semi-standard linked tableaux pair. Then there is no other semi-standard linked tableaux pair $(T_1^-,T_1^+)$ such that $T^-=T^-_1$ or $T^+=T^+_1$.
 \end{cor}
 
The following reformulation of the link between a pair of linked tableaux will be helpful. Let $(T^-,T^+)$ be a linked pair. Each column $T^-$ corresponds, using the link, to some set of $r_0$ boxes in $T^+$ with labels $i_1 1, \dots, i_{r_0} r_0$. We can connect these boxes in $T^+$ using an arrow (or path). After doing this for each column of $T^-$, all boxes of $T^+$ will be contained in exactly one of these paths. The roles of $T^-$ and $T^+$ can be reversed.

To illustrate this, consider the following linked pair:
\[
T^+=    \begin{ytableau}
    1 1 & 2 3\\
    1 2 & 3 2\\
    2 1 & 3 3 \\
    \end{ytableau}\hspace{5mm} \hspace{5mm} T^-=\begin{ytableau}
1 1 & 2 3 \\
1 2 & 3 2 \\
2 1 & 3 3 \\
\end{ytableau} \]
The arrow diagram for $T^+$ is
\[\begin{tikzpicture}[scale=0.6]
\draw (0,0) rectangle (1,1);
\draw (1,1) rectangle (2,2);
\draw (0,1) rectangle (1,2);
\draw (1,0) rectangle (2,1);
\draw (0,2) rectangle (1,3);
\draw (1,2) rectangle (2,3);
\draw[->] (0.5,2.5) --  (0.5,1.5)--(1.5,2.5);
\draw[->] (0.5,0.5) -- (1.5,1.5) -- (1.5,0.5);
\end{tikzpicture}\]
Then the arrow diagram for $T^-$ the same since $T^+$ is self-dual.

 \begin{prop}\label{prop:ltachieved} Let $(T^+,T^-)$ be a linked, semi-standard pair, and let $f_{T^\pm}$ be the associated semi-invariant. Then
 \[\Mon(T^+)=\Mon(T^-)\]
 is the leading monomial of $f_{T^\pm}.$
 \end{prop}
 \begin{proof} The monomials that appear (with possibly zero coefficient -- there could be cancellations) in $f_{T^\pm}$ are
 \[m_{\sigma^+,\sigma^-}:=\Mon(\sigma^+ \cdot \sigma^- T^+)\]
 for $(\sigma^+,\sigma^-) \in \Sym^+ \times \Sym^-.$
 The element $\sigma^+$ acts by permuting the rows of each column of $T^+$, and the element $\sigma^-$ acts by permuting the second digit in the label of the boxes of $T^+$. 
 
 We observe that for any monomial $m_{\sigma^+,\sigma^-},$ there is another monomial $m_{\tau^+,\tau^-}$ with equal or higher term order such that $\tau^+$ and $\tau^-$ both have the property that they only permute labels or boxes with the same first digit. To see this, observe that $m_{\sigma^+,\sigma^-}$ appears in the product of determinants $D(\sigma^- T^+)$. By Lemma \ref{lem:grlt}, the leading term of this product is the monomial associated to the tableau obtained from $\sigma^- T^+$ by sorting each column, which is $m_{\tau^+,\sigma^-}$ for some $\tau^+$. Since $\sigma^-$ only changes the second digits of labels, and $T^+$ is semi-standard, $\tau^+$ at most changes the order of labels in a column with the same first digit. Thus $m_{\tau^+,\sigma^-}$ has higher term order than $m_{\sigma^+,\sigma^-}$. The monomial $m_{\tau^+,\sigma^-}$ appears in the product of determinants $D(\tau^+ T^-)$, and the same argument shows that the leading term of $D(\tau^+ T^-)$ is some $m_{\tau^+,\tau^-}$ with $\tau^-$ having the required property.
 
 We can therefore restrict our attention to monomials $m_{\tau^+,\tau^-}$ such that $\tau^+$ and $\tau^-$ both have the property that they only permute labels or boxes with the same first digit. Assume that such a monomial satisfies $m_{\tau^+,\tau^-} \geq \Mon(T^+)$. We'll show that $m_{\tau^+,\tau^-} = \Mon(T^+)$, and that $\sign(\tau^+)\sign(\tau^-)=1$. This will prove the proposition. 

Before we proceed, we introduce some language and prove a claim. Consider a label $ij$ in $T^+$ in a box $B_1$. Then $\tau^-$ takes the $j$ and moves it to another box, originally labeled $ik$ for some $k$ in a box $B_2$, and $\tau^+$ takes the new $ij$ and moves it somewhere within the same column as $B_2$, say to a box $B_3$ originally labeled $il$. We say that $\tau^- \tau^+$ moves $ij$ from $B_1$ to $B_3$, via $B_2$. 
That is, the tableau looks like:
\[\begin{tikzpicture}[scale=0.6]
%\draw[] (0,-2) rectangle (1,-1);
\node[] (A) at (0.5,-1.5) {$ij$};
\node[gray] at (0,-2.1) {\tiny $B_1$};
\node[] at (-2.5,-1.5) {row $r_1$};
%\draw (3,0) rectangle (4,1);
\node[] (C) at (3.5, 0.5) {$il$};
\node[gray] at (3,-0.1) {\tiny $B_3$};
\node[] at (7.5,0.5) {row $r_3$};
%\draw (3,-4) rectangle (4,-3);
\node[] (B) at (3.5, -3.5) {$ik$};
\node[gray] at (3,-4.1) {\tiny $B_2$};
\node[] at (7.5,-3.5) {row $r_2$};
\draw (-1,2) rectangle (6,-6);
\draw[->]  (A)-- (B) node[midway,below] {$\tau^+$};
\draw[->]  (B)-- (C) node[midway,right] {$\tau^-$};
\end{tikzpicture}\]
Note that we have drawn $B_3$ to the right of $B_1$, but it could be to the left -- either case is possible.  
\begin{claim} Suppose that $\tau^- \tau^+$ takes $ij$ from box $B_1$ to $B_3$, via $B_2$, where $B_2$ was originally labeled by $ik$ and $B_3$ by $il$, just as in the above picture.   Let $r_i$ be the row of box $B_i$. 

Suppose that the following assumptions hold: 
\begin{enumerate}
\item $r_1 \geq r_3$,
\item $r_2>r_3,$ 
\item $l \geq j.$
\end{enumerate}
Then in fact $l=j$, $r_1=r_3$, and $r_1-r_3=k-j.$
\end{claim}
\begin{claimproof}
Since $T^+$ and $T^-$ are semi-standard, it is easy to see that $r_2-r_1 \geq k-j$ and $r_2-r_3 \leq k-l.$ But then 
\[r_2-r_1 \geq k-j \geq k-l \geq r_2-r_3,\]
so $r_3 \geq r_1$, hence $r_1=r_3$. This implies that the inequalities above are in fact equalities, so $j=l$  and $r_1=r_3.$
\end{claimproof}

That is, the claim states that the tableau pictured above in fact looks like 
\[\begin{tikzpicture}[scale=0.6]
%\draw[] (0,0) rectangle (1,1);
\node[] (A) at (0.5,0.5) {$ij$};
\node[gray] at (0,0) {\tiny $B_1$};
\node[gray] at (3,0) {\tiny $B_3$};
\node[gray] at (3,-4) {\tiny $B_2$};
\node[] at (2,0.5) {$\cdots$};
\node[] at (-3,0.5) {row $r_1=r_3$};
%\draw (3,0) rectangle (4,1);
\node[] (C) at (3.5, 0.5) {$ij$};
\node[]  at (3.5, -0.7) {$i j+1$};
\node[]  at (3.5, -1.7) {$\vdots$};

%\draw (3,-1) rectangle (4,0);
%\draw (3,-4) rectangle (4,-3);
\node[] (B) at (3.5, -3.5) {$ik$};
\node[] at (9,-3.5) {row $r_2=r_1+k-j$};
\draw (-1,2) rectangle (6,-6);
\end{tikzpicture}\]
Again, we have chosen to draw $B_1$ to the left of $B_3$ for convenience, but it could be on the right. We now proceed with the main proof.

To show that  $m_{\tau^+,\tau^-} = \Mon(T^+)$, we will use induction on the labels $ij$. The induction hypothesis at label $ij$ is the following statement: for all $pq<ij$, $\tau^+ \tau^-$ takes all labels $pq$ in $T^+$ to boxes previously occupied by $pq$ in the same row. 

The base case is for the label $11$ (or $i1$ for the smallest $i$ in $T^+$, but by relabeling, we can assume this is $11$). Note that all labels $11$ appear in the first row, and if any label $11$ is moved by $\tau^+ \tau^-$ to a lower row, then   $m_{\tau^+,\tau^-} > \Mon(T^+)$, a contradiction. For any label $11$ moved to a box occupied by label $pq$ in the first row, the assumptions of the claim are automatically satisfied, so $pq=11$ as desired.

Now assume that the induction hypothesis holds for all $pq<ij$. This means that for fixed $pq<ij$, the part of the monomial made up of variables $x^p_{qs}$ for any $s$ is identical in $\Mon(T^+)$ and $m_{\tau^+,\tau^-}$. If $\Mon(T^+)_{ij}$ and $(m_{\tau^+,\tau^-})_{ij}$ are the parts made up of $x^i_{js}$ for all $s$, then as $m_{\tau^+,\tau^-} \leq \Mon(T^+)$, 
\[(m_{\tau^+,\tau^-})_{ij} \leq \Mon(T^+)_{ij}.\]

Now consider where the $ij$ labels move under $\tau^+ \tau^-$. Note that by the induction assumption, if an $ij$ moves to label occupied by $il$, then $l \geq j$. Suppose we have some $ij$ in box $B_1$ moving to a box $B_3$ via $B_2$, where $B_2$ and $B_3$ were originally occupied by $ik$ and $il$ respectively. As above, let $r_i$ be the row of $B_i$. Suppose that $B_3$ is weakly above $B_1$. There are two cases: $B_2$ is strictly below $B_3$, and $B_2$ is strictly above $B_3$. In the first case, by the claim, we get that $j=l$ and that $B_1$ and $B_3$ are in fact in the same row. In the second case, since $T^-$ is semi-standard, we know that $k<j$ and that the path in $T^+$ containing both these labels in fact contains $ik,i(k+1),\dots,ij$. If we look at action of $\tau^-$ on the second digit of the $i$ labels in this path, we obtain there is some set of labels $s_1<\cdots<s_t$ such that $\tau^-$ takes $is_c$ to $is_{c+1}$, and $s_1 \leq k$, $s_{t-1} < j \leq s_t$. Applying the induction assumption to $i s_c$ for $c=1,\dots,{t-1}$, we see that $r_1-r_2=j-k$. Similar reasoning shows that $r_3-r_2=l-k.$ As $r_3-r_2 \leq r_1-r_2$, this implies $l \leq j$. So in fact $l=j$, and $B_1$ and $B_3$ are in the same row. 

We have shown that all $ij$ labels that move weakly up in fact stay in the same row, and replace identical $ij$ labels. Thus if any $ij$ labels move strictly lower,  $(m_{\tau^+,\tau^-})_{ij} > \Mon(T^+)_{ij}$, a contradiction. So the $ij$ labels satisfy the induction hypothesis. We conclude that $m_{\tau^+,\tau^-} = \Mon(T^+)$. The final step is to show that $\sign(\tau^+)\sign(\tau^-)=1$. Notice that we have shown that whenever a label $ij$ moves to the spot occupied by another label $ij$ via a box labeled $il$, $\tau^-$ takes $j$ to $l$, and $\tau^+$ takes $l$ to $j$. So they are in fact inverses of each other, and hence have the same sign.

 \end{proof}
\begin{thm}\label{thm:lm} Semi-invariants corresponding to semi-standard linked tableaux form a vector space basis of the semi-invariant ring.
\end{thm}
\begin{proof} 
We first show that these semi-invariants span. Let $f$ be any semi-invariant. By Proposition \ref{prop:lt}, its leading term corresponds to the monomial of some semi-standard linked pair $T^\pm$ with semi-standard link. Then the leading term of $f-f_{T^\pm}$ is of strictly lower order. By induction, we can therefore write $f$ as a sum of semi-standard semi-invariants. Since there is a unique semi-standard linked pair with semi-standard link, these semi-invariants are linearly independent.
\end{proof}

\section{SAGBI bases for the Kronecker quiver}\label{sec:sagbik}
Recall the definition of a SAGBI basis:
\begin{mydef} Let $R$ be a sub-algebra of a polynomial ring $P=\CC[x_1,\dots,x_m]$, together with a total order on the monomials of $P$. Given a polynomial $f$, let $\LM(f)$ be its leading monomial under this term order. A collection of elements $S \subset R$ is a \emph{SAGBI basis}  for $R$ if 
\[ \CC[\LM(f):f \in R]=\CC[\LM(f):f \in S].\]
\end{mydef}
So far, the results for quiver semi-invariants of a Kronecker quiver follow what is known about the semi-invariants of the Grassmannian. However, to move from a vector space basis of the semi-invariant ring to a SAGBI basis, we need to consider factorization. Let $f$ be a semi-invariant. There are two questions one can ask about factorization:
\begin{enumerate}
\item Does $f$ factor into a product of two semi-invariants?
\item Does the leading monomial, $\LM(f)$,  of $f$ factor into the product of the leading term of two semi-invariants? 
\end{enumerate}
For the Grassmannian, the answer to both questions always coincide. If $T$ is a semi-standard tableau,  $D(T)$ is a product of the determinants corresponding to each of its columns. This implies that the semi-invariant ring of the Grassmannian has a SAGBI basis corresponding to the semi-standard tableaux with a single column.

The following example illustrates how the leading term of a semi-invariant can factor, while the semi-invariant itself does not. 
\begin{eg} Consider the Kronecker quiver with with dimension vector $(2,2)$ and $6$ arrows. The following tableaux are a linked pair:
\[\begin{ytableau}
1 1 &  3 1 & 5 2\\
2 1 & 4 2 & 6 2\\
\end{ytableau} \hspace{5mm} \begin{ytableau}
1 1 & 2 2 & 3 1\\
4 2 & 5 1 & 6 2\\
\end{ytableau}\]
Then the leading term of the associated semi-invariant is the product of the leading terms of the semi-invariants of the associated linked pairs:
\[\begin{ytableau}
1 1 & 5 2\\
2 1  & 6 2\\
\end{ytableau} \hspace{5mm} \begin{ytableau}
1 1 & 2 2 \\
 5 1 & 6 2\\
\end{ytableau}\]
\[\begin{ytableau}
3 1 \\
4 2 \\
\end{ytableau} \hspace{5mm} \begin{ytableau}
3 1 \\
4 2 \\
\end{ytableau}\]
However, the semi-invariant of the original pair is a not a product of the lower degree semi-invariants. 

In contrast, the linked pair 
\[\begin{ytableau}
1 1 &  3 2 & 5 2\\
2 1 & 4 1 & 6 2\\
\end{ytableau} \hspace{5mm} \begin{ytableau}
1 1 & 2 2 & 4 2\\
3 1 & 5 1 & 6 2\\
\end{ytableau}\]
has a leading term that cannot be written as a product of leading terms of any semi-invariants. 
\end{eg}
In general, the link 
prevents such a factorization property for quiver semi-invariants. We consider the requirements on the link for such a factorization to exist.

\begin{mydef}\label{defn:prim} Let ${T^\pm}$ be a semi-standard linked tableaux pair.  If there are two semi-standard linked tableaux pairs ${T_1^\pm}, {T_2^\pm}$ such that 
\[\LM(f_{T^\pm})= \LM(f_{T_1^\pm})\LM(f_{T_2^\pm}),\]
then we say that $T_1^\pm$ and $T_2^\pm$ induce a \emph{splitting} of $T^\pm$. If there are no splittings of $T^\pm$, then we say it is \emph{primitive}. 

\end{mydef}
The motivation for this definition is the following immediate corollary of Theorem \ref{thm:lm}.
\begin{cor} The semi-invariants corresponding to primitive, semi-standard linked tableaux are a (possibly infinite) SAGBI basis of the semi-invariant ring of the Kronecker quiver. 
\end{cor}
Giving a combinatorial description of primitive semi-standard tableaux pairs is thus the key step to explicitly writing down SAGBI bases of Kronecker quivers. This seems to be a non-trivial problem. A related question is whether the set of primitive, semi-standard tableaux is finite for a fixed $r_1, r_2,$ and $K$. We know of no examples where it is infinite. In the following section, we show that the set is finite when $r_1=r_2=2$ for any $K$ and give an explicit, combinatorial description. Before proving this, we give some small examples where we have used the computer algebra program MAGMA to determine the set of primitive, semi-standard linked tableaux.

\begin{eg}Consider the Kronecker quiver with $r_0=2$ and $r_1=3$, and $2$ arrows between the vertices. Using MAGMA, we determine that there is exactly one primitive semi-standard tableaux pair of degree $(-3,2)$ and no primitive semi-standard tableaux pairs of degree $(-6,4)$. The primitive tableaux pair is
\[T^+=\begin{ytableau}
1 1 & 1 1\\
1 2 & 2 1\\
2 2 & 2 2 \\
\end{ytableau}\hspace{8mm} 
T^-=\begin{ytableau}
1 1 & 1 1 & 2 2\\
1 2 & 2 3 & 2 3\\
\end{ytableau}\]
This is clearly a SAGBI basis, as it is the only generator of the semi-invariant ring. 
\end{eg}

\begin{eg}\label{eg:233} Consider the Kronecker quiver with $r_0=2$ and $r_1=3$, and $3$ arrows between the vertices. Using MAGMA, we determine that there are 20 primitive semi-standard tableaux pairs of degree $(-3,2)$ and no primitive semi-standard tableaux pairs of degree $[-6,4]$. The $T^+$ of the 20 primitive tableaux pairs are:
\[
\begin{ytableau}
1 1 & 2 2\\
3 1 & 3 1\\
3 2 & 3 2 \\
\end{ytableau}\hspace{3mm}
\begin{ytableau}
1 1 & 2 1\\
1 2 & 3 1\\
3 2 & 3 2 \\
\end{ytableau}\hspace{3mm}
\begin{ytableau}
1 1 & 2 1\\
1 2 & 3 1\\
2 2 & 3 2 \\
\end{ytableau}\hspace{3mm}
\begin{ytableau}
1 1 & 2 1\\
2 2 & 3 1\\
3 2 & 3 2 \\
\end{ytableau}\hspace{3mm}
\begin{ytableau}
1 1 & 2 2\\
2 1 & 3 1\\
3 2 & 3 2 \\
\end{ytableau}\]
\[\begin{ytableau}
1 1 & 2 2\\
2 1 & 3 1\\
2 2 & 3 2 \\
\end{ytableau}\hspace{3mm}
\begin{ytableau}
1 1 & 1 1\\
1 2 & 2 1\\
3 2 & 3 2 \\
\end{ytableau}\hspace{3mm}
\begin{ytableau}
1 1 & 1 1\\
1 2 & 2 1\\
2 2 & 3 2 \\
\end{ytableau}\hspace{3mm}
\begin{ytableau}
1 1 & 1 1\\
1 2 & 2 2\\
2 1 & 3 2 \\
\end{ytableau}\hspace{3mm}
\begin{ytableau}
1 1 & 1 1\\
1 2 & 1 2\\
2 1 & 3 2 \\
\end{ytableau}\]
\[\begin{ytableau}
1 1 & 1 1\\
2 1 & 2 2\\
3 2 & 3 2 \\
\end{ytableau}\hspace{3mm}
\begin{ytableau}
1 1 & 1 1\\
1 2 & 2 1\\
2 2 & 2 2 \\
\end{ytableau}\hspace{3mm}
    \begin{ytableau}
    1 1 & 1 1\\
    2 1 & 2 2\\
    2 2 & 3 2 \\
    \end{ytableau}\hspace{3mm}
\begin{ytableau}
1 1 & 1 1\\
1 2 & 3 1\\
3 2 & 3 2 \\
\end{ytableau}\hspace{3mm}
\begin{ytableau}
1 1 & 1 1\\
1 2 & 3 1\\
2 2 & 3 2 \\
\end{ytableau}\]
\[\begin{ytableau}
1 1 & 1 1\\
2 2 & 3 1\\
3 2 & 3 2 \\
\end{ytableau}\hspace{3mm}
\begin{ytableau}
1 1 & 2 1\\
1 2 & 2 2\\
2 1 & 3 2 \\
\end{ytableau}\hspace{3mm}
\begin{ytableau}
1 1 & 2 1\\
2 1 & 2 2\\
3 2 & 3 2 \\
\end{ytableau}\hspace{3mm}
\begin{ytableau}
1 1 & 2 1\\
2 1 & 2 2\\
2 2 & 3 2 \\
\end{ytableau}\hspace{3mm}
\begin{ytableau}
2 1 & 2 1\\
2 2 & 3 1\\
3 2 & 3 2 \\
\end{ytableau}\]
Using MAGMA, we computed the semi-invariants associated to each of the tableaux pairs, and checked that they form a SAGBI basis. If we increase the number of arrows to $4$, then there are 232 primitive semi-standard tableaux pairs of degree $d(-3,2),$ for $1\leq d \leq 3$, but there may be more in higher degree. 
\end{eg}
\begin{eg}\label{eg:334} Consider the Kronecker quiver with $r_0=3$ and $r_1=3$, and $3$ arrows between the vertices. Using MAGMA, we determine that there are 18 primitive semi-standard tableaux pairs of degree $(-d,d)$ for $1 \leq d \leq 3$ and no primitive semi-standard tableaux pairs of degree $[-4,4]$. The $T^+$ of the 18 primitive tableaux pairs are:
\[
\begin{ytableau}
i 1 \\
j 2 \\
k 3 \\
\end{ytableau}\hspace{3mm} \text{ where }
1 \leq i \leq j \leq k \leq 3,\]
and
\[    \begin{ytableau}
    1 1 & 1 1\\
    1 2 & 2 3\\
    2 2 & 3 3 \\
    \end{ytableau}\hspace{3mm}
    \begin{ytableau}
    1 1 & 2 3\\
    1 2 & 3 2\\
    2 1 & 3 3 \\
    \end{ytableau}\hspace{3mm}
    \begin{ytableau}
    1 1 & 2 2\\
    2 1 & 3 2\\
    2 3 & 3 3 \\
    \end{ytableau}\hspace{3mm}
    \begin{ytableau}
    1 1 & 2 2\\
    2 1 & 3 2\\
    3 3 & 3 3 \\
    \end{ytableau}\hspace{3mm}
    \begin{ytableau}
    1 1 & 2 1\\
    1 2 & 2 3\\
    2 2 & 3 3 \\
    \end{ytableau}\hspace{3mm}
    \begin{ytableau}
    1 1 & 2 2\\
    2 1 & 2 3\\
    2 2 & 3 3 \\
    \end{ytableau}\]
\[
    \begin{ytableau}
    1 1 & 1 1 & 2 3\\
    1 2 & 2 1 & 3 2\\
    2 2 & 3 3 & 3 3\\
    \end{ytableau}\hspace{3mm}
    \begin{ytableau}
    1 1 & 1 1 & 2 2\\
    1 2 & 2 3 & 3 2\\
    2 1 & 3 3 & 3 3\\
    \end{ytableau}\]
\end{eg}

\begin{eg}Consider the Kronecker quiver with $r_0=2$ and $r_1=3$, and $3$ arrows between the vertices. The primitive semi-standard tableaux pairs of Example \ref{eg:233} all have one of the following arrow diagrams:
\[\begin{tikzpicture}[scale=0.6]
\draw (0,0) rectangle (1,1);
\draw (1,1) rectangle (2,2);
\draw (0,1) rectangle (1,2);
\draw (1,0) rectangle (2,1);
\draw (0,2) rectangle (1,3);
\draw (1,2) rectangle (2,3);
\draw[->] (0.5,2.5) -- (1.5,2.5);
\draw[->] (0.5,1.5) -- (0.5,0.5);
\draw[->] (1.5,1.5) -- (1.5,0.5);
\end{tikzpicture}
 \hspace{5mm}
\begin{tikzpicture}[scale=0.6]
\draw (0,0) rectangle (1,1);
\draw (1,1) rectangle (2,2);
\draw (0,1) rectangle (1,2);
\draw (1,0) rectangle (2,1);
\draw (0,2) rectangle (1,3);
\draw (1,2) rectangle (2,3);
\draw[->] (0.5,2.5) -- (0.5,1.5);
\draw[->] (1.5,2.5) -- (0.5,0.5);
\draw[->] (1.5,1.5) -- (1.5,0.5);
\end{tikzpicture}
\hspace{5mm}
\begin{tikzpicture}[scale=0.6]
\draw (0,0) rectangle (1,1);
\draw (1,1) rectangle (2,2);
\draw (0,1) rectangle (1,2);
\draw (1,0) rectangle (2,1);
\draw (0,2) rectangle (1,3);
\draw (1,2) rectangle (2,3);
\draw[->] (0.5,2.5) -- (0.5,1.5);
\draw[->] (1.5,2.5) -- (1.5,1.5);
\draw[->] (0.5,0.5) -- (1.5,0.5);
\end{tikzpicture}
 \hspace{5mm}
\begin{tikzpicture}[scale=0.6]
\draw (0,0) rectangle (1,1);
\draw (1,1) rectangle (2,2);
\draw (0,1) rectangle (1,2);
\draw (1,0) rectangle (2,1);
\draw (0,2) rectangle (1,3);
\draw (1,2) rectangle (2,3);
\draw[->] (0.5,2.5) -- (1.5,1.5);
\draw[->] (0.5,1.5) -- (1.5,0.5);
\draw[->] (1.5,2.5) -- (0.5,0.5);
\end{tikzpicture}\] 
\end{eg}

\section{The $r_0=r_1=2$ case and finiteness}\label{sec:2case}
In this section, we give an explicit description of the primitive semi-standard tableaux pairs of the Kronecker quiver, when $r_0=r_1=2$, for any number of arrows $K$ between the vertices.  Our strategy is as follows: in the lemmas and propositions, we rule out some basic configurations of arrows in arrow diagrams. This gives strong control over the possible shapes of primitive semi-standard tableaux. We then use an inductive argument to show finiteness; in fact, we give an explicit description of the primitive semi-standard tableaux.

As $r_0=r_1=2$, any semi-standard linked tableaux $T^-$ and $T^+$ will consist of two rows, and each of the $a \cdot r_2$ arrows in the arrow configuration of $T^\pm$ will be of length one. 

\begin{lem}\label{lem:nobarr} Let $T^\pm$ be a semi-standard, primitive, linked pair. Then the arrow configurations satisfy the following conditions. 
\begin{enumerate}\item Arrows in $T^\pm$ cannot point \emph{backwards}, i.e for an arrow $i1\to j2$, $j2$ is not positioned to the left of $i1$ in $T^\pm$.
\item Arrows in $T^\pm$ cannot point \emph{downwards}, i.e there is no arrow $i1\to j2$ where $i1$ is positioned on the first row of $T^\pm$ and $j2$ on the second row of $T^\pm$.
\end{enumerate}
\end{lem}
\begin{proof}

\begin{enumerate}
\item We wish to exclude the following four configurations:

\[\begin{tikzpicture}[scale=0.6,every node/.style={scale=0.8}]
\draw [dashed, color=gray] (1,0) rectangle (2,1);
\draw [dashed,color=gray] (1,-1) rectangle (2,0);
\draw (0,0) rectangle (1,1);
\draw (0,-1) rectangle (1,0);

\draw (2,0) rectangle (3,1);
\draw (2,-1) rectangle (3,0);

\draw[->] (2.2,0.5) -- (0.8,-0.5);
\node[] at (0.5,-0.5) {$j2$};
\node[] at (2.5,0.5) {$i1$};
\end{tikzpicture}
\hspace{5mm}
\begin{tikzpicture}[scale=0.6,every node/.style={scale=0.8}]
\draw [dashed, color=gray] (1,0) rectangle (2,1);
\draw [dashed,color=gray] (1,-1) rectangle (2,0);
\draw (0,0) rectangle (1,1);
\draw (0,-1) rectangle (1,0);

\draw (2,0) rectangle (3,1);
\draw (2,-1) rectangle (3,0);

\draw[->] (2.2,0.5) -- (0.8,0.5);
\node[] at (0.5,0.5) {$j2$};
\node[] at (2.5,0.5) {$i1$};
\end{tikzpicture}
 \hspace{5mm}
\begin{tikzpicture}[scale=0.6,every node/.style={scale=0.8}]
\draw [dashed, color=gray] (1,0) rectangle (2,1);
\draw [dashed,color=gray] (1,-1) rectangle (2,0);
\draw (0,0) rectangle (1,1);
\draw (0,-1) rectangle (1,0);

\draw (2,0) rectangle (3,1);
\draw (2,-1) rectangle (3,0);

\draw[->] (2.2,-0.5) -- (0.8,-0.5);
\node[] at (0.5,-0.5) {$j2$};
\node[] at (2.5,-0.5) {$i1$};
\end{tikzpicture}
\hspace{5mm}
\begin{tikzpicture}[scale=0.6,every node/.style={scale=0.8}]
\draw [dashed, color=gray] (1,0) rectangle (2,1);
\draw [dashed,color=gray] (1,-1) rectangle (2,0);
\draw (0,0) rectangle (1,1);
\draw (0,-1) rectangle (1,0);

\draw (2,0) rectangle (3,1);
\draw (2,-1) rectangle (3,0);

\draw[->] (2.2,-0.5) -- (0.8,0.5);
\node[] at (0.5,0.5) {$j2$};
\node[] at (2.5,-0.5) {$i1$};
\end{tikzpicture}\] 
Without loss of generality, we can assume these configurations are in $T^+$. Since $T^-$ is semi-standard, we have that $i1<j2$. Since $T^+$ is also semi-standard, every configuration except the first one violates this inequality. We show that the first configuration also does not occur, as it cannot be contained in any primitive tableau. More precisely, assume we have the following:
\[\begin{tikzpicture}[scale=0.6,every node/.style={scale=0.8}]
\draw [dashed, color=gray] (1,0) rectangle (2,1);
\draw [dashed,color=gray] (1,-1) rectangle (2,0);
\draw [dashed, color=gray] (-1,-1) rectangle (0,0);
\draw [dashed, color=gray] (-1,0) rectangle (0,1);
\draw [dashed,color=gray] (3,-1) rectangle (4,0);
\draw [dashed,color=gray] (3,0) rectangle (4,1);
\draw (0,0) rectangle (1,1);
\draw (0,-1) rectangle (1,0);
\draw (2,0) rectangle (3,1);
\draw (2,-1) rectangle (3,0);

\draw[->] (2.2,0.5) -- (0.8,-0.5);
\node[] at (0.5,0.5) {$k*$};
\node[] at (2.5,-0.5) {$l*$};
\node[] at (2.5,0.5) {$i1$};
\node[] at (0.5,-0.5) {$j2$};
\node[color=gray] at (1.5,0.5) {$A$};
\node[color=gray] at (1.5,-0.5) {$B$};

\end{tikzpicture}\]
where we denote the intermediate blocks by $A=(A_1,...,A_n)$ and $B=(B_1,...,B_n)$ and $k*\in\{k1,k2\}$,  $l*\in\{l1,l2\}$. We claim that this tableau splits into a $2\times 1$ block formed by the $i1$ and $j2$ cells, and a second tableau containing the rest of $T^+$ as follows:
\[\begin{tikzpicture}[scale=0.6,every node/.style={scale=0.8}]
\draw (0,0) rectangle (1,1);
\draw (0,-1) rectangle (1,0);
\node[] at (0.5,-0.5) {$j2$};
\node[] at (0.5,0.5) {$i1$};
\end{tikzpicture}
\hspace{5mm} \textup{ and } \hspace{5mm}
\begin{tikzpicture}[scale=0.6,every node/.style={scale=0.8}]
\draw [dashed, color=gray] (-1,-1) rectangle (0,0);
\draw [dashed, color=gray] (-1,0) rectangle (0,1);
\draw [dashed,color=gray] (2,-1) rectangle (3,0);
\draw [dashed,color=gray] (2,0) rectangle (3,1);
\draw (1,0) rectangle (2,1);
\draw (1,-1) rectangle (2,0);
\draw (0,0) rectangle (1,1);
\draw (0,-1) rectangle (1,0);

\node[] at (0.5,0.5) {$k*$};
\node[] at (0.5,-0.5) {$B$};
\node[] at (1.5,0.5) {$A$};
\node[] at (1.5,-0.5) {$l*$};
\end{tikzpicture}\]
Indeed, both factors are semi-standard. The $i1-j2$ block is semi-standard by hypothesis. The semi-standardness of $T^+$ implies the following inequalities 
\[k*\leq A_1\leq \ldots \leq A_n \leq i1<j2\leq B_1\leq \ldots \leq B_n \leq l*.\] One can now see that the right-hand tableau is semi-standard, as both non-strict and strict inequalities required in the definition are immediately deduced from the sequence above: in particular, we have that $A_i<B_{i+1}$, which addresses the ``shift" that appears when removing the $i1-j2$ block from $T^+$.
The splitting exists, therefore any tableau with a backwards-pointing arrow is not primitive.
\item Here we wish to exclude only one configuration, namely

\[\begin{tikzpicture}[scale=0.6,every node/.style={scale=0.8}]
\draw [dashed, color=gray] (1,0) rectangle (2,1);
\draw [dashed,color=gray] (1,-1) rectangle (2,0);
\draw [dashed, color=gray] (-1,-1) rectangle (0,0);
\draw [dashed, color=gray] (-1,0) rectangle (0,1);
\draw [dashed,color=gray] (3,-1) rectangle (4,0);
\draw [dashed,color=gray] (3,0) rectangle (4,1);
\draw (0,0) rectangle (1,1);
\draw (0,-1) rectangle (1,0);
\node[color=gray] at (1.5,-0.5) {$B$};
\node[color=gray] at (1.5,0.5) {$A$};
\node[] at (2.5,0.5) {$k*$};
\node[] at (0.5,-0.5) {$l*$};

\draw (2,0) rectangle (3,1);
\draw (2,-1) rectangle (3,0);

\draw[->] (0.7,0.5) -- (2.2,-0.5);
\node[] at (0.5,0.5) {$i1$};
\node[] at (2.5,-0.5) {$j2$};
\end{tikzpicture}\]
where the notation is as before. We now compare the relative positions of $k*$ and $l*$ to exclude all possible configurations. 

If $k<l$ then removing the $i1-j2$ block splits $T^+$ into two SSYTs. Indeed, there is nothing to check about $T^-$, while in $T^+$ we have 
\[i1\leq A_1\leq \ldots \leq A_n \leq k*<l*\leq B_1\leq \ldots \leq B_n \leq j2.\] After splitting the $i1-j2$ block, what remains is an SSYT due to the strict inequality $A_s<B_t$ for all $s, t$ (in particular if $s=t+1)$, in addition to $A_1<l*$ and $k*<B_n$. 

If $k\geq l$, since there are no backwards arrows in either $T^\pm$, the configuration in $T^-$ is:
\[\begin{tikzpicture}[scale=0.6,every node/.style={scale=0.8}]
\draw [] (1,0) rectangle (2,1);
\draw [] (1,-1) rectangle (2,0);
\draw [dashed, color=gray] (-1,-1) rectangle (0,0);
\draw [dashed, color=gray] (-1,0) rectangle (0,1);
\draw [dashed,color=gray] (3,-1) rectangle (4,0);
\draw [dashed,color=gray] (3,0) rectangle (4,1);
\draw [dashed,color=gray] (0,0) rectangle (1,1);
\draw [dashed,color=gray] (0,-1) rectangle (1,0);
\draw [dashed,color=gray] (2,0) rectangle (3,1);
\draw [dashed,color=gray] (2,-1) rectangle (3,0);

\node[] at (1.5,-0.5) {$j2$};
\node[] at (1.5,0.5) {$i1$};
\node[style=circle, fill=white] at (0,0) {$k1$};
\node[style=circle, fill=white] at (3,0) {$l2$};

\draw[->] (0.3,0) -- (1.2,-0.5);
\draw[->] (1.8,0.5) -- (2.7,0);
\end{tikzpicture}\]
There are four cases:
\begin{itemize}
\item We have $k*=k1$ and $l*=l1$ in $T^+$. In other words, both $k1$ and $l1$ are on the first row of $T^-$. This implies that the condition $k1\leq i1 \leq l2$ becomes an equality, which is a contradiction since $i1<l1$ in $T^+$.
\item A similar argument works for $k*=k2$ and $l*=l2$ in $T^+$.
\item If $k*=k1$ and $l*=l2$, then $k1\leq i1<j2\leq l2$ from $T^-$, so that $i=j=k=l$. Even further, in $T^+$ we obtain $(A_1,\ldots, A_n)=(i_1,\ldots, i_1)$ and $(B_1,\ldots, B_n)=(i_2,\ldots, i_2)$. It is then immediate that we can split off the initial $i1-j2$ (now $i1-i2$) block without violating the strict inequalities in the remaining tableau.
\item $k*=k2$, $l*=l1$ in  $T^+$, in which case $T^-$ becomes: 
\[\begin{tikzpicture}[scale=0.6,every node/.style={scale=0.8}]
\draw [] (1,0) rectangle (2,1);
\draw [] (1,-1) rectangle (2,0);
\draw [] (-1,-1) rectangle (0,0);
\draw [dashed, color=gray] (-1,0) rectangle (0,1);
\draw [dashed,color=gray] (3,-1) rectangle (4,0);
\draw [] (3,0) rectangle (4,1);
\draw [dashed,color=gray] (0,0) rectangle (1,1);
\draw [dashed,color=gray] (0,-1) rectangle (1,0);
\draw [dashed,color=gray] (2,0) rectangle (3,1);
\draw [dashed,color=gray] (2,-1) rectangle (3,0);

\node[] at (1.5,-0.5) {$j2$};
\node[] at (1.5,0.5) {$i1$};
\node[] at (-0.5,-0.5) {$k1$};
\node[] at (3.5,0.5) {$l2$};
%\node[color=gray] at () {$$};

\draw[->] (-0.2,-0.5) -- (1.2,-0.5);
\draw[->] (1.8,0.5) -- (3.2,0.5);
\end{tikzpicture}\]
We claim that we can split the $2\times 2$ block containing $i, j, k$ and $l$. In fact, we only need to show that the following sub-block of $T^-$
\[\begin{tikzpicture}[scale=0.6,every node/.style={scale=0.8}]
\draw (1,0) rectangle (2,1);
\draw (1,-1) rectangle (2,0);
\draw (0,0) rectangle (1,1);
\draw (0,-1) rectangle (1,0);

\node[] at (0.5,0.5) {$i1$};
\node[] at (0.5,-0.5) {$k1$};
\node[] at (1.5,0.5) {$l2$};
\node[] at (1.5,-0.5) {$j2$};
\end{tikzpicture}\]
is semi-standard: its complement in $T^-$ is semi-standard by construction, and we automatically obtain that $T^+$ splits into two rectangular SSYTs. The above $2\times 2$ block is not semi-standard only if $i=k$ or $l=j$. But from $T^+$ and our hypothesis we know that $j> k\geq l > i$, so neither of these cases occur.
\end{itemize}
 We are thus able to conclude that there is no primitive linked SSYT-pair in which one tableau contains a \emph{downwards} arrow.
\end{enumerate}

\end{proof}
\begin{lem} \label{lem:lem2} Let $(T^\pm)$ be a pair of primitive, semi-standard linked tableaux. Consider labels $a1$ and $d2$ that appear in the first and second row respectively of $T^+$ (and hence also, respectively, in the first and second row of $T^-$). Then either $a1$ is strictly the left of $d2$ in both $T^+$ and $T^-$, or $a1$ is strictly to the right of $d2$ in both $T^+$ and $T^-$. 
\end{lem}
\begin{proof} We first rule out the case where $a1$ is strictly above $d2$ in $T^+$. Then there must be a downwards arrow in $T^-$, which is impossible by Lemma \ref{lem:nobarr}.  

Suppose that $a1$ is strictly to the left of $d2$ in $T^-$, and strictly to the right of $d2$ in $T^+$. That is, the configuration in $T^-$ is of the following form, where the second digits of the labels above and below $a1$ and $d2$ are determined by using Lemma \ref{lem:nobarr}.
\[\begin{tikzpicture}[scale=0.6,every node/.style={scale=0.8}]
\draw [] (0,0) rectangle (1,1);
\draw [] (0,1) rectangle (1,2);
\draw [] (4,0) rectangle (5,1);
\draw [] (4,1) rectangle (5,2);
\draw [dashed, color=gray] (-1,0) rectangle (6,1);
\draw [dashed, color=gray] (-1,1) rectangle (6,2);

\node[] at (0.5,1.5) {$a1$};
\node[] at (0.5,0.5) {$b1$};

\node[] at (4.5,0.5) {$d2$};
\node[] at (4.5,1.5) {$c2$};

\end{tikzpicture}\] 
In $T^+$, again by using Lemma \ref{lem:nobarr}, the configuration must therefore be:
\[\begin{tikzpicture}[scale=0.6,every node/.style={scale=0.8}]
\draw [] (0,0) rectangle (1,1);
\draw [] (-2,0) rectangle (-1,1);
\draw [] (6,1) rectangle (7,2);
\draw [] (4,1) rectangle (5,2);
\draw [dashed, color=gray] (-3,0) rectangle (8,1);
\draw [dashed, color=gray] (-3,1) rectangle (8,2);

\node[] at (4.5,1.5) {$a1$};
\node[] at (6.5,1.5) {$b2$};
\draw[->] (4.8,1.5) -- (6.2,1.5);

\node[] at (0.5,0.5) {$d2$};
\node[] at (-1.5,0.5) {$c1$};
\draw[->] (-1.2,0.5) -- (0.2,0.5);

\end{tikzpicture}\] 
If $c<b$, then split $T^\pm$ into the tableaux pair 
\[ \begin{tikzpicture}[scale=0.6,every node/.style={scale=0.8}]
\draw [] (0,0) rectangle (1,1);
\draw [] (0,1) rectangle (1,2);
\node[] at (0.5,1.5) {$a1$};
\node[] at (0.5,0.5) {$d2$};
\end{tikzpicture} \hspace{5mm}
\begin{tikzpicture}[scale=0.6,every node/.style={scale=0.8}]
\draw [] (0,0) rectangle (1,1);
\draw [] (0,1) rectangle (1,2);
\node[] at (0.5,1.5) {$a1$};
\node[] at (0.5,0.5) {$d2$};
\end{tikzpicture}
\] 
and the tableaux $T^\pm_1$ obtained by removing $a1$ and $d2$ from $T^\pm$. Since $a1<d2$, the tableaux above are a semi-standard linked tableaux pair. $T^-_1$ is clearly semi-standard and $T^\pm_1$ are a linked pair. It remains to check that $T^+_1$ is semi-standard. Since $c<d$, every label in  the first row of $T^+$, to the right of $a1$ and the left of $d2$ is strictly smaller than every label in  the second row of $T^+$, to the right of $a1$ and the left of $d2$. So even after removing $a1$ and $d2$, the tableau remains semi-standard. Therefore, $T^\pm$ is not primitive. 

If $c \geq b$, split $T^\pm$ into the tableaux pair
\[ \begin{tikzpicture}[scale=0.6,every node/.style={scale=0.8}]
\draw [] (0,0) rectangle (1,1);
\draw [] (0,1) rectangle (1,2);
\draw [] (1,0) rectangle (2,1);
\draw [] (1,1) rectangle (2,2);

\node[] at (0.5,0.5) {$b1$};
\node[] at (0.5,1.5) {$a1$};
\node[] at (1.5,0.5) {$d2$};
\node[] at (1.5,1.5) {$c2$};
\end{tikzpicture} \hspace{5mm}
\begin{tikzpicture}[scale=0.6,every node/.style={scale=0.8}]
\draw [] (0,0) rectangle (1,1);
\draw [] (0,1) rectangle (1,2);
\draw [] (1,0) rectangle (2,1);
\draw [] (1,1) rectangle (2,2);

\node[] at (0.5,0.5) {$c1$};
\node[] at (0.5,1.5) {$a1$};
\node[] at (1.5,0.5) {$d2$};
\node[] at (1.5,1.5) {$b2$};
\end{tikzpicture}
\] 
and the tableaux $T^\pm_1$ obtained by removing these four labels from $T^\pm$. The tableaux above are semi-standard as
\[a<b \leq c,\]
\[b \leq c <d.\]
It is clear by the shape of what we have removed that $T^\pm_1$ are also semi-standard. So $T^\pm$ is not primitive.

By symmetry we can swap the roles of $T^+$ and $T^-$ and use the same argument in the last remaining case.

\end{proof}
The proof of the next proposition involves multiple cases, yet is essential to the main theorem of this section.

\begin{prop}\label{prop:cases} Let $T^\pm$ be a primitive, semi-standard linked tableaux pair. Then neither $T^+$ nor $T^-$ contains a configuration of the following form:
\[\begin{tikzpicture}[scale=0.6,every node/.style={scale=0.8}]
\draw [] (0,1) rectangle (1,2);
\draw [] (-2,0) rectangle (-1,1);
\draw [] (4,0) rectangle (5,1);
\draw [thick,double] (1,0)--(1,2);
\draw [dashed, color=gray] (-3,0) rectangle (8,1);
\draw [dashed, color=gray] (-3,1) rectangle (8,2);

\node[] at (0.5,1.5) {$a1$};
\node[] at (-1.5,0.5) {$b_1 1$};
\node[] at (4.5,0.5) {$d2$};
\draw[->] (-1.2,0.5) -- (4.2,0.5);
\end{tikzpicture}\] 
We assume that $b_1 1$ is to the left of, and $d2$ to the right of, the vertical double line. 
\end{prop}
\begin{proof}
To prove the statement, we need a bit more notation. Without loss of generality, we assume that the offending configuration appears in $T^+$. Let $b_2 1$ be the label under $a1$ in $T^+$. Let $c_2 2$ be label above $d2$ in $T^+$. Then the configuration in $T^+$ is
\[\begin{tikzpicture}[scale=0.6,every node/.style={scale=0.8}]
\draw [] (0,1) rectangle (1,2);
\draw [] (-2,0) rectangle (-1,1);
\draw [] (0,0) rectangle (1,1);
\draw [] (4,0) rectangle (5,1);
\draw [] (4,1) rectangle (5,2);

\draw [dashed, color=gray] (-3,0) rectangle (8,1);
\draw [dashed, color=gray] (-3,1) rectangle (8,2);

\node[] at (0.5,1.5) {$a1$};
\node[] at (0.5,0.5) {$b_2 1$};
\node[] at (-1.5,0.5) {$b_1 1$};
\node[] at (4.5,0.5) {$d2$};
\node[] at (4.5,1.5) {$c_2 2$};

\draw[->] (-1.2,0.5) to [ bend right=45] (4.2,0.5);
%\draw[->] (0.8,1.5) to [ bend left=45] (6.1,1.6);

\end{tikzpicture}\] 

In $T^-$, by Lemma \ref{lem:lem2}, $a1$ is again to the left of $d2$, i.e. $b_1 \geq a$. Let $c_1 1$ be the label below $a 1$ in $T^-$. There are two cases, which we consider separately:
\begin{enumerate}
\item $c_2 \leq c_1$,
\item $c_2 > c_1.$
\end{enumerate}
Suppose first that $c_2 \leq c_1$. Then, by changing the link trivially if necessary, we have the configurations:
\[\begin{tikzpicture}[scale=0.6,every node/.style={scale=0.8}]
\draw [] (-2,0) rectangle (-1,1);
\draw [] (0,1) rectangle (1,2);
\draw [] (0,0) rectangle (1,1);
\draw [] (4,0) rectangle (5,1);
\draw [] (4,1) rectangle (5,2);
\draw [] (6,1) rectangle (7,2);

\draw [dashed, color=gray] (-3,0) rectangle (8,1);
\draw [dashed, color=gray] (-3,1) rectangle (8,2);

\node[] at (0.5,1.5) {$a1$};
\node[] at (0.5,0.5) {$c_1 1$};
\node[] at (-1.5,0.5) {$c_2 1$};
\node[] at (4.5,1.5) {$b_1 2$};
\node[] at (4.5,0.5) {$d2$};
\node[] at (6.5,1.5) {$b_2 2$};
\node[] at (-4,1) {$T^-=$};

\draw[->] (0.8,1.5) to [ bend left=45] (6.1,1.6);
\draw[->] (-1.2,0.5) to [ bend right=45] (4.1,0.6);

\end{tikzpicture}\]
\[\begin{tikzpicture}[scale=0.6,every node/.style={scale=0.8}]
\draw [] (0,1) rectangle (1,2);
\draw [] (-2,0) rectangle (-1,1);
\draw [] (0,0) rectangle (1,1);
\draw [] (4,0) rectangle (5,1);
\draw [] (4,1) rectangle (5,2);
\draw [] (6,1) rectangle (7,2);
\draw [dashed, color=gray] (-3,0) rectangle (8,1);
\draw [dashed, color=gray] (-3,1) rectangle (8,2);

\node[] at (0.5,1.5) {$a1$};
\node[] at (0.5,0.5) {$b_2 1$};
\node[] at (-1.5,0.5) {$b_1 1$};
\node[] at (4.5,0.5) {$d2$};
\node[] at (4.5,1.5) {$c_2 2$};
\node[] at (6.5,1.5) {$c_1 2$};

\draw[->] (-1.2,0.5) to [ bend right=45] (4.2,0.5);
\draw[->] (0.8,1.5) to [ bend left=45] (6.1,1.6);
\node[] at (-4,1) {$T^+=$};

\end{tikzpicture}\]  
We have four inequalities:
\[\begin{tikzpicture}[scale=0.6,every node/.style={scale=0.8}]
\draw [] (-1,1) rectangle (1,2);
\draw [] (1,0) rectangle (3,1);
\draw [] (-1,0) rectangle (1,1);
\draw [] (1,1) rectangle (3,2);

\node[] at (0,1.5) {$a \leq b_1$};
\node[] at (0,0.5) {$d \geq c_1$};
\node[] at (2,1.5) {$a \leq c_2$};
\node[] at (2,0.5) {$d \geq b_2$};
\end{tikzpicture}\]  
We consider three subcases, where some of these inequalities are strict or equalities, that together cover all possibilities:
\begin{itemize}
\item Suppose the inequalities are of one of the forms:
\[\begin{tikzpicture}[scale=0.6,every node/.style={scale=0.8}]
\draw [] (0,1) rectangle (1,2);
\draw [] (1,0) rectangle (2,1);
\draw [] (0,0) rectangle (1,1);
\draw [] (1,1) rectangle (2,2);
\node[] at (0.5,1.5) {$<$};
\node[] at (0.5,0.5) {$>$};
\node[] at (1.5,1.5) {$\leq$};
\node[] at (1.5,0.5) {$\geq$};
\end{tikzpicture} \hspace{5mm}
\begin{tikzpicture}[scale=0.6,every node/.style={scale=0.8}]
\draw [] (0,1) rectangle (1,2);
\draw [] (1,0) rectangle (2,1);
\draw [] (0,0) rectangle (1,1);
\draw [] (1,1) rectangle (2,2);
\node[] at (1.5,1.5) {$<$};
\node[] at (1.5,0.5) {$>$};
\node[] at (0.5,1.5) {$\leq$};
\node[] at (0.5,0.5) {$\geq$};
\end{tikzpicture}\]  
In the first case, we remove the highlighted boxes below:
\[\begin{tikzpicture}[scale=0.6,every node/.style={scale=0.8}]
\draw [] (-2,0) rectangle (-1,1);
\draw [fill=yellow] (0,1) rectangle (1,2);
\draw [fill=yellow] (0,0) rectangle (1,1);
\draw [fill=yellow] (4,0) rectangle (5,1);
\draw [fill=yellow] (4,1) rectangle (5,2);
\draw [] (6,1) rectangle (7,2);

\draw [dashed, color=gray] (-3,0) rectangle (8,1);
\draw [dashed, color=gray] (-3,1) rectangle (8,2);

\node[] at (0.5,1.5) {$a1$};
\node[] at (0.5,0.5) {$c_1 1$};
\node[] at (-1.5,0.5) {$c_2 1$};
\node[] at (4.5,1.5) {$b_1 2$};
\node[] at (4.5,0.5) {$d2$};
\node[] at (6.5,1.5) {$b_2 2$};
\node[] at (-4,1) {$T^-=$};

\draw[->] (0.8,1.5) to [ bend left=45] (6.1,1.6);
\draw[->] (-1.2,0.5) to [ bend right=45] (4.1,0.6);

\end{tikzpicture} \hspace{5mm} \begin{tikzpicture}[scale=0.6,every node/.style={scale=0.8}]
\draw [fill=yellow] (0,1) rectangle (1,2);
\draw [fill=yellow] (-2,0) rectangle (-1,1);
\draw [] (0,0) rectangle (1,1);
\draw [fill=yellow] (4,0) rectangle (5,1);
\draw [] (4,1) rectangle (5,2);
\draw [fill=yellow] (6,1) rectangle (7,2);
\draw [dashed, color=gray] (-3,0) rectangle (8,1);
\draw [dashed, color=gray] (-3,1) rectangle (8,2);

\node[] at (0.5,1.5) {$a1$};
\node[] at (0.5,0.5) {$b_2 1$};
\node[] at (-1.5,0.5) {$b_1 1$};
\node[] at (4.5,0.5) {$d2$};
\node[] at (4.5,1.5) {$c_2 2$};
\node[] at (6.5,1.5) {$c_1 2$};

\draw[->] (-1.2,0.5) to [ bend right=45] (4.2,0.5);
\draw[->] (0.8,1.5) to [ bend left=45] (6.1,1.6);
\node[] at (-4,1) {$T^+=$};

\end{tikzpicture}\]  
Using the inequalities, this gives a splitting of $T^\pm$, which therefore is not primitive. We leave it to the reader to check that the new tableaux are semi-standard. 

In the second case, we remove the highlighted boxes below:
\[\begin{tikzpicture}[scale=0.6,every node/.style={scale=0.8}]
\draw [fill=yellow] (-2,0) rectangle (-1,1);
\draw [fill=yellow] (0,1) rectangle (1,2);
\draw [] (0,0) rectangle (1,1);
\draw [fill=yellow] (4,0) rectangle (5,1);
\draw [] (4,1) rectangle (5,2);
\draw [fill=yellow] (6,1) rectangle (7,2);

\draw [dashed, color=gray] (-3,0) rectangle (8,1);
\draw [dashed, color=gray] (-3,1) rectangle (8,2);

\node[] at (0.5,1.5) {$a1$};
\node[] at (0.5,0.5) {$c_1 1$};
\node[] at (-1.5,0.5) {$c_2 1$};
\node[] at (4.5,1.5) {$b_1 2$};
\node[] at (4.5,0.5) {$d2$};
\node[] at (6.5,1.5) {$b_2 2$};
\node[] at (-4,1) {$T^-=$};

\draw[->] (0.8,1.5) to [ bend left=45] (6.1,1.6);
\draw[->] (-1.2,0.5) to [ bend right=45] (4.1,0.6);

\end{tikzpicture} \hspace{5mm} \begin{tikzpicture}[scale=0.6,every node/.style={scale=0.8}]
\draw [fill=yellow] (0,1) rectangle (1,2);
\draw [] (-2,0) rectangle (-1,1);
\draw [fill=yellow] (0,0) rectangle (1,1);
\draw [fill=yellow] (4,0) rectangle (5,1);
\draw [fill=yellow] (4,1) rectangle (5,2);
\draw [] (6,1) rectangle (7,2);
\draw [dashed, color=gray] (-3,0) rectangle (8,1);
\draw [dashed, color=gray] (-3,1) rectangle (8,2);

\node[] at (0.5,1.5) {$a1$};
\node[] at (0.5,0.5) {$b_2 1$};
\node[] at (-1.5,0.5) {$b_1 1$};
\node[] at (4.5,0.5) {$d2$};
\node[] at (4.5,1.5) {$c_2 2$};
\node[] at (6.5,1.5) {$c_1 2$};

\draw[->] (-1.2,0.5) to [ bend right=45] (4.2,0.5);
\draw[->] (0.8,1.5) to [ bend left=45] (6.1,1.6);
\node[] at (-4,1) {$T^+=$};

\end{tikzpicture}\]  

\item Suppose the inequalities are of one of the forms:
\[\begin{tikzpicture}[scale=0.6,every node/.style={scale=0.8}]
\draw [] (0,1) rectangle (1,2);
\draw [] (1,0) rectangle (2,1);
\draw [] (0,0) rectangle (1,1);
\draw [] (1,1) rectangle (2,2);
\node[] at (0.5,1.5) {$\leq$};
\node[] at (0.5,0.5) {$=$};
\node[] at (1.5,1.5) {$=$};
\node[] at (1.5,0.5) {$\geq$};
\end{tikzpicture} \hspace{5mm}
\begin{tikzpicture}[scale=0.6,every node/.style={scale=0.8}]
\draw [] (0,1) rectangle (1,2);
\draw [] (1,0) rectangle (2,1);
\draw [] (0,0) rectangle (1,1);
\draw [] (1,1) rectangle (2,2);
\node[] at (0.5,1.5) {$=$};
\node[] at (0.5,0.5) {$\geq$};
\node[] at (1.5,1.5) {$\leq$};
\node[] at (1.5,0.5) {$=$};
\end{tikzpicture}\]  
In either case, we remove the highlighted boxes below:
\[\begin{tikzpicture}[scale=0.6,every node/.style={scale=0.8}]
\draw [] (-2,0) rectangle (-1,1);
\draw [fill=yellow] (0,1) rectangle (1,2);
\draw [] (0,0) rectangle (1,1);
\draw [fill=yellow] (4,0) rectangle (5,1);
\draw [] (4,1) rectangle (5,2);
\draw [] (6,1) rectangle (7,2);

\draw [dashed, color=gray] (-3,0) rectangle (8,1);
\draw [dashed, color=gray] (-3,1) rectangle (8,2);

\node[] at (0.5,1.5) {$a1$};
\node[] at (0.5,0.5) {$c_1 1$};
\node[] at (-1.5,0.5) {$c_2 1$};
\node[] at (4.5,1.5) {$b_1 2$};
\node[] at (4.5,0.5) {$d2$};
\node[] at (6.5,1.5) {$b_2 2$};
\node[] at (-4,1) {$T^-=$};

\draw[->] (0.8,1.5) to [ bend left=45] (6.1,1.6);
\draw[->] (-1.2,0.5) to [ bend right=45] (4.1,0.6);

\end{tikzpicture} \hspace{5mm} \begin{tikzpicture}[scale=0.6,every node/.style={scale=0.8}]
\draw [fill=yellow] (0,1) rectangle (1,2);
\draw [] (-2,0) rectangle (-1,1);
\draw [] (0,0) rectangle (1,1);
\draw [fill=yellow] (4,0) rectangle (5,1);
\draw [] (4,1) rectangle (5,2);
\draw [] (6,1) rectangle (7,2);
\draw [dashed, color=gray] (-3,0) rectangle (8,1);
\draw [dashed, color=gray] (-3,1) rectangle (8,2);

\node[] at (0.5,1.5) {$a1$};
\node[] at (0.5,0.5) {$b_2 1$};
\node[] at (-1.5,0.5) {$b_1 1$};
\node[] at (4.5,0.5) {$d2$};
\node[] at (4.5,1.5) {$c_2 2$};
\node[] at (6.5,1.5) {$c_1 2$};

\draw[->] (-1.2,0.5) to [ bend right=45] (4.2,0.5);
\draw[->] (0.8,1.5) to [ bend left=45] (6.1,1.6);
\node[] at (-4,1) {$T^+=$};

\end{tikzpicture}\]  
It is clear that the tableaux removed from $T^\pm$ are semi-standard. The equalities can be used to show that what remains is semi-standard. For example, in the first case, this follows from the inequalities
\[c_2=a<b_2, c_1=d>b_1.\]
The second case is similar. 

\item Suppose the inequalities are of one of the forms:
\[\begin{tikzpicture}[scale=0.6,every node/.style={scale=0.8}]
\draw [] (0,1) rectangle (1,2);
\draw [] (1,0) rectangle (2,1);
\draw [] (0,0) rectangle (1,1);
\draw [] (1,1) rectangle (2,2);
\node[] at (0.5,1.5) {$<$};
\node[] at (0.5,0.5) {$\geq$};
\node[] at (1.5,1.5) {$<$};
\node[] at (1.5,0.5) {$\geq$};
\end{tikzpicture} \hspace{5mm}
\begin{tikzpicture}[scale=0.6,every node/.style={scale=0.8}]
\draw [] (0,1) rectangle (1,2);
\draw [] (1,0) rectangle (2,1);
\draw [] (0,0) rectangle (1,1);
\draw [] (1,1) rectangle (2,2);
\node[] at (1.5,1.5) {$\leq$};
\node[] at (1.5,0.5) {$>$};
\node[] at (0.5,1.5) {$\leq$};
\node[] at (0.5,0.5) {$>$};
\end{tikzpicture}\]  

In the first case, we remove the highlighted boxes below:
\[\begin{tikzpicture}[scale=0.6,every node/.style={scale=0.8}]
\draw [fill=yellow] (-2,0) rectangle (-1,1);
\draw [fill=yellow] (0,1) rectangle (1,2);
\draw [] (0,0) rectangle (1,1);
\draw [fill=yellow] (4,0) rectangle (5,1);
\draw [fill=yellow] (4,1) rectangle (5,2);
\draw [] (6,1) rectangle (7,2);

\draw [dashed, color=gray] (-3,0) rectangle (8,1);
\draw [dashed, color=gray] (-3,1) rectangle (8,2);

\node[] at (0.5,1.5) {$a1$};
\node[] at (0.5,0.5) {$c_1 1$};
\node[] at (-1.5,0.5) {$c_2 1$};
\node[] at (4.5,1.5) {$b_1 2$};
\node[] at (4.5,0.5) {$d2$};
\node[] at (6.5,1.5) {$b_2 2$};
\node[] at (-4,1) {$T^-=$};

\draw[->] (0.8,1.5) to [ bend left=45] (6.1,1.6);
\draw[->] (-1.2,0.5) to [ bend right=45] (4.1,0.6);

\end{tikzpicture} \hspace{5mm}
 \begin{tikzpicture}[scale=0.6,every node/.style={scale=0.8}]
\draw [fill=yellow] (0,1) rectangle (1,2);
\draw [fill=yellow] (-2,0) rectangle (-1,1);
\draw [] (0,0) rectangle (1,1);
\draw [fill=yellow] (4,0) rectangle (5,1);
\draw [fill=yellow] (4,1) rectangle (5,2);
\draw [] (6,1) rectangle (7,2);
\draw [dashed, color=gray] (-3,0) rectangle (8,1);
\draw [dashed, color=gray] (-3,1) rectangle (8,2);

\node[] at (0.5,1.5) {$a1$};
\node[] at (0.5,0.5) {$b_2 1$};
\node[] at (-1.5,0.5) {$b_1 1$};
\node[] at (4.5,0.5) {$d2$};
\node[] at (4.5,1.5) {$c_2 2$};
\node[] at (6.5,1.5) {$c_1 2$};

\draw[->] (-1.2,0.5) to [ bend right=45] (4.2,0.5);
\draw[->] (0.8,1.5) to [ bend left=45] (6.1,1.6);
\node[] at (-4,1) {$T^+=$};
\end{tikzpicture}\]   

In the second case, we remove the highlighted boxes below:
\[\begin{tikzpicture}[scale=0.6,every node/.style={scale=0.8}]
\draw [] (-2,0) rectangle (-1,1);
\draw [fill=yellow] (0,1) rectangle (1,2);
\draw [fill=yellow] (0,0) rectangle (1,1);
\draw [fill=yellow] (4,0) rectangle (5,1);
\draw [] (4,1) rectangle (5,2);
\draw [fill=yellow] (6,1) rectangle (7,2);

\draw [dashed, color=gray] (-3,0) rectangle (8,1);
\draw [dashed, color=gray] (-3,1) rectangle (8,2);

\node[] at (0.5,1.5) {$a1$};
\node[] at (0.5,0.5) {$c_1 1$};
\node[] at (-1.5,0.5) {$c_2 1$};
\node[] at (4.5,1.5) {$b_1 2$};
\node[] at (4.5,0.5) {$d2$};
\node[] at (6.5,1.5) {$b_2 2$};
\node[] at (-4,1) {$T^-=$};

\draw[->] (0.8,1.5) to [ bend left=45] (6.1,1.6);
\draw[->] (-1.2,0.5) to [ bend right=45] (4.1,0.6);

\end{tikzpicture} \hspace{5mm} \begin{tikzpicture}[scale=0.6,every node/.style={scale=0.8}]
\draw [fill=yellow] (0,1) rectangle (1,2);
\draw [] (-2,0) rectangle (-1,1);
\draw [fill=yellow] (0,0) rectangle (1,1);
\draw [fill=yellow] (4,0) rectangle (5,1);
\draw [] (4,1) rectangle (5,2);
\draw [fill=yellow] (6,1) rectangle (7,2);
\draw [dashed, color=gray] (-3,0) rectangle (8,1);
\draw [dashed, color=gray] (-3,1) rectangle (8,2);

\node[] at (0.5,1.5) {$a1$};
\node[] at (0.5,0.5) {$b_2 1$};
\node[] at (-1.5,0.5) {$b_1 1$};
\node[] at (4.5,0.5) {$d2$};
\node[] at (4.5,1.5) {$c_2 2$};
\node[] at (6.5,1.5) {$c_1 2$};

\draw[->] (-1.2,0.5) to [ bend right=45] (4.2,0.5);
\draw[->] (0.8,1.5) to [ bend left=45] (6.1,1.6);
\node[] at (-4,1) {$T^+=$};

\end{tikzpicture}\]  
We leave it to the reader to check that the new tableaux are semi-standard. 
\end{itemize}

This shows that the lemma holds when $c_1 \geq c_2$. We now consider the case when $c_2>c_1$. The configuration is
\[\begin{tikzpicture}[scale=0.6,every node/.style={scale=0.8}]
\draw [] (2,0) rectangle (3,1);
\draw [] (0,1) rectangle (1,2);
\draw [] (0,0) rectangle (1,1);
\draw [] (4,0) rectangle (5,1);
\draw [] (4,1) rectangle (5,2);
\draw [] (6,1) rectangle (7,2);

\draw [dashed, color=gray] (-3,0) rectangle (8,1);
\draw [dashed, color=gray] (-3,1) rectangle (8,2);

\node[] at (0.5,1.5) {$a1$};
\node[] at (0.5,0.5) {$c_1 1$};
\node[] at (2.5,0.5) {$c_2 1$};
\node[] at (4.5,1.5) {$b_1 2$};
\node[] at (4.5,0.5) {$d2$};
\node[] at (6.5,1.5) {$b_2 2$};
\node[] at (-4,1) {$T^-=$};

\draw[->] (0.8,1.5) to [ bend left=45] (6.1,1.6);
\draw[->] (2.8,0.5) to [] (4.1,0.5);

\end{tikzpicture}\]
\[\begin{tikzpicture}[scale=0.6,every node/.style={scale=0.8}]
\draw [] (0,1) rectangle (1,2);
\draw [] (-2,0) rectangle (-1,1);
\draw [] (0,0) rectangle (1,1);
\draw [] (4,0) rectangle (5,1);
\draw [] (4,1) rectangle (5,2);
\draw [] (2,1) rectangle (3,2);
\draw [dashed, color=gray] (-3,0) rectangle (8,1);
\draw [dashed, color=gray] (-3,1) rectangle (8,2);

\node[] at (0.5,1.5) {$a1$};
\node[] at (0.5,0.5) {$b_2 1$};
\node[] at (-1.5,0.5) {$b_1 1$};
\node[] at (4.5,0.5) {$d2$};
\node[] at (4.5,1.5) {$c_2 2$};
\node[] at (2.5,1.5) {$c_1 2$};

\draw[->] (-1.2,0.5) to [ bend right=45] (4.2,0.5);
\draw[->] (0.8,1.5) to  (2.2,1.5);
\node[] at (-4,1) {$T^+=$};

\end{tikzpicture}\]  
We consider the possible ordering of the pairs $(c_2,b_2)$ and $(c_1,b_1)$. If $b_1<c_1$ and $c_2<b_2$, then it is clear that we can remove $a1$ and $d2$, so the tableaux pair is not primitive. If only one of these inequalities hold, it is also easy to see that the tableaux pair splits: for example, if $b_1<c_1$ and $c_2 \geq b_2$, then we can remove the highlighted labels:
\[\begin{tikzpicture}[scale=0.6,every node/.style={scale=0.8}]
\draw [fill=yellow] (2,0) rectangle (3,1);
\draw [fill=yellow] (0,1) rectangle (1,2);
\draw [] (0,0) rectangle (1,1);
\draw [fill=yellow] (4,0) rectangle (5,1);
\draw [] (4,1) rectangle (5,2);
\draw [fill=yellow] (6,1) rectangle (7,2);

\draw [dashed, color=gray] (-3,0) rectangle (8,1);
\draw [dashed, color=gray] (-3,1) rectangle (8,2);

\node[] at (0.5,1.5) {$a1$};
\node[] at (0.5,0.5) {$c_1 1$};
\node[] at (2.5,0.5) {$c_2 1$};
\node[] at (4.5,1.5) {$b_1 2$};
\node[] at (4.5,0.5) {$d2$};
\node[] at (6.5,1.5) {$b_2 2$};
\node[] at (-4,1) {$T^-=$};

\draw[->] (0.8,1.5) to [ bend left=45] (6.1,1.6);
\draw[->] (2.8,0.5) to [] (4.1,0.5);

\end{tikzpicture}\]
\[\begin{tikzpicture}[scale=0.6,every node/.style={scale=0.8}]
\draw [fill=yellow] (0,1) rectangle (1,2);
\draw [] (-2,0) rectangle (-1,1);
\draw [fill=yellow] (0,0) rectangle (1,1);
\draw [fill=yellow] (4,0) rectangle (5,1);
\draw [fill=yellow] (4,1) rectangle (5,2);
\draw [] (2,1) rectangle (3,2);
\draw [dashed, color=gray] (-3,0) rectangle (8,1);
\draw [dashed, color=gray] (-3,1) rectangle (8,2);

\node[] at (0.5,1.5) {$a1$};
\node[] at (0.5,0.5) {$b_2 1$};
\node[] at (-1.5,0.5) {$b_1 1$};
\node[] at (4.5,0.5) {$d2$};
\node[] at (4.5,1.5) {$c_2 2$};
\node[] at (2.5,1.5) {$c_1 2$};

\draw[->] (-1.2,0.5) to [ bend right=45] (4.2,0.5);
\draw[->] (0.8,1.5) to  (2.2,1.5);
\node[] at (-4,1) {$T^+=$};

\end{tikzpicture}\]  
The strict inequality $b_1<c_1$ ensures that what remains in $T^-$ is semi-standard, and as $d>c_2 \geq b_2$, what is removed from $T^-$ is also semi-standard. The situation when $b_1 \geq c_1$ and $c_2<b_2$ is analogous. 

We now consider the situation when  $b_1 \geq c_1$ and $c_2 \geq b_2$. Since $c_1 < c_2$ and $b_2 \geq b_1$, we have the chain of inequalities:
\[d>c_2 \geq b_2 \geq b_1 \geq c_1 > a,\]
and at least one of the middle three inequalities must be strict. 

If $b_2>b_1$ or both $c_2>b_2$ and $b_1 >c_1$ then we can remove the six highlighted boxes below to obtain a splitting of $T^\pm.$ 
\[\begin{tikzpicture}[scale=0.6,every node/.style={scale=0.8}]
\draw [fill=yellow] (2,0) rectangle (3,1);
\draw [fill=yellow] (0,1) rectangle (1,2);
\draw [fill=yellow] (0,0) rectangle (1,1);
\draw [fill=yellow] (4,0) rectangle (5,1);
\draw [fill=yellow] (4,1) rectangle (5,2);
\draw [fill=yellow] (6,1) rectangle (7,2);

\draw [dashed, color=gray] (-3,0) rectangle (8,1);
\draw [dashed, color=gray] (-3,1) rectangle (8,2);

\node[] at (0.5,1.5) {$a1$};
\node[] at (0.5,0.5) {$c_1 1$};
\node[] at (2.5,0.5) {$c_2 1$};
\node[] at (4.5,1.5) {$b_1 2$};
\node[] at (4.5,0.5) {$d2$};
\node[] at (6.5,1.5) {$b_2 2$};
\node[] at (-4,1) {$T^-=$};

\draw[->] (0.8,1.5) to [ bend left=45] (6.1,1.6);
\draw[->] (2.8,0.5) to [] (4.1,0.5);

\end{tikzpicture}\]
\[\begin{tikzpicture}[scale=0.6,every node/.style={scale=0.8}]
\draw [fill=yellow] (0,1) rectangle (1,2);
\draw [fill=yellow] (-2,0) rectangle (-1,1);
\draw [fill=yellow] (0,0) rectangle (1,1);
\draw [fill=yellow] (4,0) rectangle (5,1);
\draw [fill=yellow] (4,1) rectangle (5,2);
\draw [fill=yellow] (2,1) rectangle (3,2);
\draw [dashed, color=gray] (-3,0) rectangle (8,1);
\draw [dashed, color=gray] (-3,1) rectangle (8,2);

\node[] at (0.5,1.5) {$a1$};
\node[] at (0.5,0.5) {$b_2 1$};
\node[] at (-1.5,0.5) {$b_1 1$};
\node[] at (4.5,0.5) {$d2$};
\node[] at (4.5,1.5) {$c_2 2$};
\node[] at (2.5,1.5) {$c_1 2$};

\draw[->] (-1.2,0.5) to [ bend right=45] (4.2,0.5);
\draw[->] (0.8,1.5) to  (2.2,1.5);
\node[] at (-4,1) {$T^+=$};

\end{tikzpicture}\]  
It remains to investigate what happens with $b_2=b_1$ and either $c_2=b_2$ or $b_1=c_1$. We check the case where $b_2=b_1=c_1<c_2$, and leave the second case to the reader, as it is symmetric. In this case, we need to consider the label below $b 2$ in $T^+$. It is either of the form $d_2 2$ or of the form $b_3 1$. We consider these separately:
\begin{itemize}
\item Suppose the label is of the form $d_2 2$. Then the configuration is below and the highlighted boxes induce a splitting. 
\[\begin{tikzpicture}[scale=0.6,every node/.style={scale=0.8}]
\draw [fill=yellow] (2,0) rectangle (3,1);
\draw [fill=yellow] (0,1) rectangle (1,2);
\draw [fill=yellow] (0,0) rectangle (1,1);
\draw [] (4,0) rectangle (5,1);
\draw [fill=yellow] (4,1) rectangle (5,2);
\draw [] (6,1) rectangle (7,2);

\draw [dashed, color=gray] (-3,0) rectangle (8,1);
\draw [dashed, color=gray] (-3,1) rectangle (8,2);

\node[] at (0.5,1.5) {$a1$};
\node[] at (0.5,0.5) {$b 1$};
\node[] at (2.5,0.5) {$d_2 2$};
\node[] at (4.5,1.5) {$b 2$};
\node[] at (4.5,0.5) {$d2$};
\node[] at (6.5,1.5) {$b 2$};
\node[] at (-4,1) {$T^-=$};

\draw[->] (0.8,1.5) to [ bend left=45] (6.1,1.6);
%\draw[->] (2.8,0.5) to [] (4.1,0.5);

\end{tikzpicture}\]
\[\begin{tikzpicture}[scale=0.6,every node/.style={scale=0.8}]
\draw [fill=yellow] (0,1) rectangle (1,2);
\draw [] (-2,0) rectangle (-1,1);
\draw [fill=yellow] (0,0) rectangle (1,1);
\draw [] (4,0) rectangle (5,1);
\draw [] (4,1) rectangle (5,2);
\draw [fill=yellow] (2,1) rectangle (3,2);
\draw [fill=yellow] (2,0) rectangle (3,1);
\draw [dashed, color=gray] (-3,0) rectangle (8,1);
\draw [dashed, color=gray] (-3,1) rectangle (8,2);

\node[] at (0.5,1.5) {$a1$};
\node[] at (0.5,0.5) {$b 1$};
\node[] at (-1.5,0.5) {$b 1$};
\node[] at (4.5,0.5) {$d2$};
\node[] at (4.5,1.5) {$c_2 2$};
\node[] at (2.5,1.5) {$b 2$};
\node[] at (2.5,0.5) {$d_2 2$};

\draw[->] (-1.2,0.5) to [ bend right=45] (4.2,0.5);
\draw[->] (0.8,1.5) to  (2.2,1.5);
\node[] at (-4,1) {$T^+=$};

\end{tikzpicture}\]  
\item Suppose the label is of the form $b_3 1$. If $b_3<d$, then the configuration is below and the highlighted boxes induce a splitting. 
\[\begin{tikzpicture}[scale=0.6,every node/.style={scale=0.8}]
\draw [fill=yellow] (2,0) rectangle (3,1);
\draw [fill=yellow] (0,1) rectangle (1,2);
\draw [fill=yellow] (0,0) rectangle (1,1);
\draw [fill=yellow] (4,0) rectangle (5,1);
\draw [] (4,1) rectangle (5,2);
\draw [fill=yellow] (6,1) rectangle (7,2);
\draw [fill=yellow] (8,1) rectangle (9,2);

\draw [dashed, color=gray] (-3,0) rectangle (10,1);
\draw [dashed, color=gray] (-3,1) rectangle (10,2);

\node[] at (0.5,1.5) {$a1$};
\node[] at (0.5,0.5) {$b 1$};
\node[] at (2.5,0.5) {$c_2 1$};
\node[] at (4.5,1.5) {$b 2$};
\node[] at (4.5,0.5) {$d2$};
\node[] at (6.5,1.5) {$b 2$};
\node[] at (8.5,1.5) {$b_3 2$};

\node[] at (-4,1) {$T^-=$};

\draw[->] (0.8,1.5) to [ bend left=45] (6.1,1.6);
%\draw[->] (2.8,0.5) to [] (4.1,0.5);

\end{tikzpicture}\]
\[\begin{tikzpicture}[scale=0.6,every node/.style={scale=0.8}]
\draw [fill=yellow] (0,1) rectangle (1,2);
\draw [] (-2,0) rectangle (-1,1);
\draw [fill=yellow] (0,0) rectangle (1,1);
\draw [fill=yellow] (4,0) rectangle (5,1);
\draw [fill=yellow] (4,1) rectangle (5,2);
\draw [fill=yellow] (2,1) rectangle (3,2);
\draw [fill=yellow] (2,0) rectangle (3,1);
\draw [dashed, color=gray] (-3,0) rectangle (8,1);
\draw [dashed, color=gray] (-3,1) rectangle (8,2);

\node[] at (0.5,1.5) {$a1$};
\node[] at (0.5,0.5) {$b 1$};
\node[] at (-1.5,0.5) {$b 1$};
\node[] at (4.5,0.5) {$d2$};
\node[] at (4.5,1.5) {$c_2 2$};
\node[] at (2.5,1.5) {$b 2$};
\node[] at (2.5,0.5) {$b_3 1$};

\draw[->] (-1.2,0.5) to [ bend right=45] (4.2,0.5);
\draw[->] (0.8,1.5) to  (2.2,1.5);
\node[] at (-4,1) {$T^+=$};

\end{tikzpicture}\]  
 If $b_3 \geq d$, then in fact $b_3=d>c_2$ and the highlighted boxes below induce a splitting.
 \[\begin{tikzpicture}[scale=0.6,every node/.style={scale=0.8}]
\draw [] (2,0) rectangle (3,1);
\draw [fill=yellow] (0,1) rectangle (1,2);
\draw [fill=yellow] (0,0) rectangle (1,1);
\draw [fill=yellow] (4,0) rectangle (5,1);
\draw [] (4,1) rectangle (5,2);
\draw [fill=yellow] (6,1) rectangle (7,2);
\draw [] (8,1) rectangle (9,2);

\draw [dashed, color=gray] (-3,0) rectangle (10,1);
\draw [dashed, color=gray] (-3,1) rectangle (10,2);

\node[] at (0.5,1.5) {$a1$};
\node[] at (0.5,0.5) {$b 1$};
\node[] at (2.5,0.5) {$c_2 1$};
\node[] at (4.5,1.5) {$b 2$};
\node[] at (4.5,0.5) {$d2$};
\node[] at (6.5,1.5) {$b 2$};
\node[] at (8.5,1.5) {$b_3 2$};

\node[] at (-4,1) {$T^-=$};

\draw[->] (0.8,1.5) to [ bend left=45] (6.1,1.6);
%\draw[->] (2.8,0.5) to [] (4.1,0.5);

\end{tikzpicture}\]
\[\begin{tikzpicture}[scale=0.6,every node/.style={scale=0.8}]
\draw [fill=yellow] (0,1) rectangle (1,2);
\draw [] (-2,0) rectangle (-1,1);
\draw [fill=yellow] (0,0) rectangle (1,1);
\draw [fill=yellow] (4,0) rectangle (5,1);
\draw [] (4,1) rectangle (5,2);
\draw [fill=yellow] (2,1) rectangle (3,2);
\draw [] (2,0) rectangle (3,1);
\draw [dashed, color=gray] (-3,0) rectangle (8,1);
\draw [dashed, color=gray] (-3,1) rectangle (8,2);

\node[] at (0.5,1.5) {$a1$};
\node[] at (0.5,0.5) {$b 1$};
\node[] at (-1.5,0.5) {$b 1$};
\node[] at (4.5,0.5) {$d2$};
\node[] at (4.5,1.5) {$c_2 2$};
\node[] at (2.5,1.5) {$b 2$};
\node[] at (2.5,0.5) {$b_3 1$};

\draw[->] (-1.2,0.5) to [ bend right=45] (4.2,0.5);
\draw[->] (0.8,1.5) to  (2.2,1.5);
\node[] at (-4,1) {$T^+=$};

\end{tikzpicture}\]  
\end{itemize}

\end{proof}
There is a version for $d$ of Proposition \ref{prop:cases}:
\begin{prop}\label{prop:cases2} Let $T^\pm$ be a primitive, semi-standard linked tableaux pair. Then neither $T^+$ nor $T^-$ contain a configuration of the following form:
\[\begin{tikzpicture}[scale=0.6,every node/.style={scale=0.8}]
\draw [] (0,0) rectangle (1,1);
\draw [] (-4,1) rectangle (-3,2);
\draw [] (2,1) rectangle (3,2);
\draw [thick,double] (0,0)--(0,2);
\draw [dashed, color=gray] (-5,0) rectangle (6,1);
\draw [dashed, color=gray] (-5,1) rectangle (6,2);

\node[] at (0.5,0.5) {$d2$};
\node[] at (-3.5,1.5) {$a 1$};
\node[] at (2.5,1.5) {$b 2$};
\draw[->] (-3.2,1.5) -- (2.2,1.5);
\end{tikzpicture}\] 
We assume that $a 1$ is to the left of, and $b2$ to the right of, the vertical double line. 
\end{prop}
\begin{proof}
Since the tableaux set up is symmetric across the first and second rows, this follows from Proposition \ref{prop:cases}.
\end{proof}

In preparation of the proof of finiteness, we introduce some notation. Let $T^\pm$ be a primitive, semi-standard linked pair. We label the entries of the first row of $T^-$ with second digit $1$,
\[a_1 1 \leq \cdots \leq a_k 1,\]
and the entries of the first row with second digit $2$
\[b_1 2 \leq \cdots \leq b_l 2.\]
For the second row, we label the entries
\[c_1 1 \leq \cdots \leq c_l 1,\]
and 
\[d_1 2 \leq \cdots \leq d_k 2.\]
It is not hard to see that the number of $a$s equal the number of $d$s, and the number of $b$s equal the number of $c$s, so the indexing above is justified. Because there are no backwards arrows by Lemma \ref{lem:nobarr}, as we move through the columns of $T^\pm$ from left to right, we have always passed strictly more $a$ labels than $d$ labels, until we reach the last column. Now consider the configuration near some $a_j$. Define 
\begin{itemize}
\item $s$ to be the number of $d$s  to the left of $a_j$ in $T^+$ and  hence also in $T^-$, by Lemma \ref{lem:lem2};
\item $t=j-s > 0$, by the above observation;
\item $y$ to be the number of $b$s  to the left of $a_j$ in $T^-$;
\item $x$ to be the number of $c$s to the left of $a_j$ in $T^+$. 
\end{itemize} 
Then near $a_j$ the tableaux are
\[\begin{tikzpicture}[scale=1,every node/.style={scale=0.8}]
\draw [] (2,1) rectangle (3,2);
\draw [] (0,1) rectangle (1,2);
\draw [] (0,0) rectangle (1,1);
\draw [] (4,1) rectangle (5,2);

\draw [dashed, color=gray] (-1,0) rectangle (6,1);
\draw [dashed, color=gray] (-1,1) rectangle (6,2);

\node[] at (0.5,1.5) {$a_j1$};
\node[] at (0.5,0.5) {$c_{y+t}1$};
\node[] at (2.5,1.5) {$b_{x+t} 2$};
\node[] at (4.5,1.5) {$b_{z_-} 2$};
\node[] at (-2,1) {$T^-=$};

\draw[->] (0.8,0.8) to [] (4.2,1.3);
\draw[->] (0.8,1.5) to [] (2,1.5);

\end{tikzpicture}\]

\[\begin{tikzpicture}[scale=1,every node/.style={scale=0.8}]
\draw [] (2,1) rectangle (3,2);
\draw [] (0,1) rectangle (1,2);
\draw [] (0,0) rectangle (1,1);
\draw [] (4,1) rectangle (5,2);

\draw [dashed, color=gray] (-1,0) rectangle (6,1);
\draw [dashed, color=gray] (-1,1) rectangle (6,2);

\node[] at (0.5,1.5) {$a_j1$};
\node[] at (0.5,0.5) {$b_{x+t}1$};
\node[] at (2.5,1.5) {$c_{y+t} 2$};
\node[] at (4.5,1.5) {$c_{z_+}2$};
\node[] at (-2,1) {$T^+=$};

\draw[->] (0.8,0.8) to [] (4.2,1.3);
\draw[->] (0.8,1.5) to [] (2,1.5);

\end{tikzpicture}\]
The labels have been filled in using the following observations:
\begin{itemize}
\item The label $a_j 1$ is the $(j+y)^{th}$ box in the first row of $T^-$. Since there are $s$ $d$s to the left of $a_j 1$, the label $b_p 1$ below $a_j 1$ in $T^-$ must have $p=j+y-s=y+t$. Similar reasoning gives that the label below $a_j 1$ in $T^+$ is $c_{x+t} 1$.
\item  It follows from Proposition \ref{prop:cases} that the target of the arrow in the configuration above in $T^+$ with source $b_{x+t} 1$ is on the first row. The analogous statement holds for $c_{y+t} 2$ in $T^-$. 

\end{itemize}
\begin{lem}\label{lem:abs} Using the above notation,
\[ |x-y| < t.\]
\end{lem}
\begin{proof}
By definition $b_y 2$ is strictly the left of $a_j 1$ in $T^-$. Since there are no backwards arrows, and $a_j 1$ maps to $b_{x+t}2$, it follows that 
\[x+t>y.\]
Looking at $T^+$, we similarly obtain that $c_{y+t}2$ is strictly to the right of $c_x 2$, and so
\[y+t>x.\]
This proves the claim. 
\end{proof}

We now define an \emph{initial set} for each $a_j$. Assume $x \leq y$. If $y<x$, then the definition is completely analogous by symmetry. 
\begin{mydef} Using the notation above, we define the \emph{initial set} of $a_j$ to be
\begin{itemize}
\item $I^a_j:=\{a_j,c_{x+t},b_{x+t}\}$ if $a_j<c_{x+t},$
\item $I^a_j:=\{a_j,c_{y+t},b_{y+t}\}$ if $a_j\geq c_{x+t}.$
\end{itemize}
\end{mydef}
The initial set is defined so that it begins the description of a sub-tableaux pair of $T^\pm$ that induces a splitting.
\begin{lem}\label{lem:splitting} Near $a_j$, the initial set defines the part of sub-tableaux pair that induces a splitting of $T^\pm$.
\end{lem}
\begin{proof}
Consider the first case, when $a_j<c_{x+t}$. Note by Lemma \ref{lem:abs}, $x+t>y$, and therefore $b_{x+t} 2$ is to the right of $a_j 1$ in $T^-$.  Then if we highlight the initial set, plus extra tableaux to indicate the pattern, both what remains and what is removed is semi-standard:
\[\begin{tikzpicture}[scale=1,every node/.style={scale=0.8}]
\draw [fill=yellow] (-2,0) rectangle (-1,1);
\draw [fill=yellow] (2,1) rectangle (3,2);
\draw [fill=yellow] (2,0) rectangle (3,1);
\draw [fill=yellow] (0,1) rectangle (1,2);
\draw [] (0,0) rectangle (1,1);
\draw [] (4,1) rectangle (5,2);

\draw [dashed, color=gray] (-3,0) rectangle (6,1);
\draw [dashed, color=gray] (-3,1) rectangle (6,2);

\node[] at (0.5,1.5) {$a_j1$};
\node[] at (0.5,0.5) {$c_{y+t}1$};
\node[] at (-1.5,0.5) {$c_{x+t}1$};

\node[] at (2.5,1.5) {$b_{x+t} 2$};
\node[] at (4.5,1.5) {$b_{z_-} 2$};
\node[] at (-4,1) {$T^-=$};

%\draw[->] (0.8,0.8) to [] (4.2,1.3);
%\draw[->] (0.8,1.5) to [] (2,1.5);

\end{tikzpicture}\]

\[\begin{tikzpicture}[scale=1,every node/.style={scale=0.8}]
\draw [fill=yellow] (2,1) rectangle (3,2);
\draw [] (4,1) rectangle (5,2);
\draw [fill=yellow] (0,1) rectangle (1,2);
\draw [fill=yellow] (0,0) rectangle (1,1);
\draw [fill=yellow] (2,0) rectangle (3,1);
\draw [] (6,1) rectangle (7,2);

\draw [dashed, color=gray] (-1,0) rectangle (8,1);
\draw [dashed, color=gray] (-1,1) rectangle (8,2);

\node[] at (0.5,1.5) {$a_j1$};
\node[] at (0.5,0.5) {$b_{x+t}1$};
\node[] at (2.5,1.5) {$c_{x+t} 2$};
\node[] at (4.5,1.5) {$c_{y+t} 2$};
\node[] at (6.5,1.5) {$c_{z_+}2$};
\node[] at (-2,1) {$T^+=$};

%\draw[->] (0.8,0.8) to [] (4.2,1.3);
%\draw[->] (0.8,1.5) to [] (2,1.5);

\end{tikzpicture}\]

Consider the second case, when $a_j \geq c_{x+t}$. Since by assumption, $c_{x+t} \leq a_j$, in fact $c_{x+t}=a_j$. Applying Proposition \ref{prop:cases} twice, we obtain that $c_{x+t} 2$ is above a label of the form $b_p 1$ and that there are no $d$ type labels between $b_{x+t}1$ and $b_p1$ in $T^+$. Therefore 
\[p \geq x+2t=(x+t)+t>y+t,\]
 applying Lemma \ref{lem:abs}. We conclude that in $T^+$, $c_{x+t}$ is strictly to the right of $b_{y+t}$.  Again, if we highlight the initial set, plus extra tableaux to indicate the pattern, both what remains and what is removed is semi-standard. What remains is semi-standard because of the inequality:
 \[c_{x+t}=a_j<b_{x+t}.\]
 We illustrate this below:
\[\begin{tikzpicture}[scale=1,every node/.style={scale=0.8}]
\draw [] (-2,0) rectangle (-1,1);
\draw [fill=yellow] (2,1) rectangle (3,2);
\draw [fill=yellow] (2,0) rectangle (3,1);
\draw [fill=yellow] (0,1) rectangle (1,2);
\draw [fill=yellow] (0,0) rectangle (1,1);
\draw [] (4,1) rectangle (5,2);

\draw [dashed, color=gray] (-3,0) rectangle (6,1);
\draw [dashed, color=gray] (-3,1) rectangle (6,2);

\node[] at (0.5,1.5) {$a_j1$};
\node[] at (0.5,0.5) {$c_{y+t}1$};
\node[] at (-1.5,0.5) {$c_{x+t}1$};

\node[] at (2.5,1.5) {$b_{y+t} 2$};
\node[] at (4.5,1.5) {$b_{z_-} 2$};
\node[] at (-4,1) {$T^-=$};

%\draw[->] (0.8,0.8) to [] (4.2,1.3);
%\draw[->] (0.8,1.5) to [] (2,1.5);

\end{tikzpicture}\]

\[\begin{tikzpicture}[scale=1,every node/.style={scale=0.8}]
\draw [fill=yellow] (2,0) rectangle (3,1);
\draw [fill=yellow] (6,1) rectangle (7,2);
\draw [fill=yellow] (0,1) rectangle (1,2);
\draw [] (0,0) rectangle (1,1);
\draw [fill=yellow] (6,0) rectangle (7,1);
\draw [] (4,1) rectangle (5,2);
\draw [] (6,1) rectangle (7,2);
\draw [] (8,1) rectangle (9,2);

\draw [dashed, color=gray] (-1,0) rectangle (10,1);
\draw [dashed, color=gray] (-1,1) rectangle (10,2);

\node[] at (0.5,1.5) {$a_j1$};
\node[] at (0.5,0.5) {$b_{x+t}1$};
\node[] at (2.5,0.5) {$b_{y+t} 2$};
\node[] at (4.5,1.5) {$c_{x+t} 2$};
\node[] at (6.5,1.5) {$c_{y+t} 2$};
\node[] at (8.5,1.5) {$c_{z_+}2$};
\node[] at (-2,1) {$T^+=$};

%\draw[->] (0.8,0.8) to [] (4.2,1.3);
%\draw[->] (0.8,1.5) to [] (2,1.5);
\end{tikzpicture}\]
\end{proof}

We define the notion of an initial set for each $d_{j}$ -- it is the analogue of the definition for the $a_j$. Set $\hat{j}=k-j+1$, so that $d_j$ is the $\hat{j}^{th}$ $d$ from the right. Define 
\begin{itemize}
\item $s$ to be the number of $a$s  to the right of $d_j$ in $T^+$ and  hence also in $T^-$, by Lemma \ref{lem:lem2};
\item $t=\hat{j}-s> 0$,
\item $y$ to be the number of $b$s  to the right of $d_{j}$ in $T^+$, and $\hat{y}=:l-y+1$ (recall that $l$ is the total number of $b$s, which is equal to the total number of $c$s),
\item $x$ to be the number of $c$s to the right of $d_l$ in $T^-$, and $\hat{x}:=l-x+1$.
\end{itemize} 
Then near $d_l$ the tableaux look like
\[\begin{tikzpicture}[scale=1,every node/.style={scale=0.8}]
\draw [] (2,0) rectangle (3,1);
\draw [] (0,0) rectangle (1,1);
\draw [] (4,1) rectangle (5,2);
\draw [] (4,0) rectangle (5,1);

\draw [dashed, color=gray] (-1,0) rectangle (6,1);
\draw [dashed, color=gray] (-1,1) rectangle (6,2);

\node[] at (4.5,0.5) {$d_l 2$};
\node[] at (4.5,1.5) {$b_{\hat{y}-t}2$};
\node[] at (2.5,0.5) {$c_{\hat{x}-t} 1$};
\node[] at (0.5,0.5) {$c_{z_-} 2$};
\node[] at (-2,1) {$T^-=$};

\draw[->] (2.8,0.5) to [] (4.2,0.5);
\draw[->] (0.8,0.8) to [] (4.1,1.3);

\end{tikzpicture}\]
\[\begin{tikzpicture}[scale=1,every node/.style={scale=0.8}]
\draw [] (2,0) rectangle (3,1);
\draw [] (0,0) rectangle (1,1);
\draw [] (4,1) rectangle (5,2);
\draw [] (4,0) rectangle (5,1);

\draw [dashed, color=gray] (-1,0) rectangle (6,1);
\draw [dashed, color=gray] (-1,1) rectangle (6,2);

\node[] at (4.5,0.5) {$d_l 2$};
\node[] at (4.5,1.5) {$c_{\hat{x}-t}2$};
\node[] at (2.5,0.5) {$b_{\hat{y}-t} 1$};
\node[] at (0.5,0.5) {$b_{z_+} 1$};
\node[] at (-2,1) {$T^+=$};

\draw[->] (2.8,0.5) to [] (4.2,0.5);
\draw[->] (0.8,0.8) to [] (4.1,1.3);

\end{tikzpicture}\]
Again, it holds that $|x-y|=|\hat{x}-\hat{y}|<t.$

We now define an \emph{initial set} for each $d_j$. Assume $x \leq y$. If $y<x$, then the definition is completely analogous using symmetry between $T^+$ and $T^-$.  
\begin{mydef} Using the notation above, we define the \emph{initial set} of $d_j$ to be
\begin{itemize}
\item $I^d_j:=\{d_j,c_{x-t},b_{x-t}\}$ if $d_j \leq c_{y-t},$
\item $I^d_j:=\{d_j,c_{y-t},b_{y-t}\}$ if $d_j > c_{y-t}.$
\end{itemize}
\end{mydef}
As in the $a_j$ case, one can check that the initial sets define the first few boxes of sub-tableaux that give a semi-standard splitting of $T^\pm$: that is, Lemma \ref{lem:splitting} holds for $I^d_j$ as well. 

Before we state and prove the main theorem of this section, we illustrate the core idea of the theorem as well as the role of initial sets in the example below.
\begin{eg}\label{eg:thm} Consider a semi-standard linked pair where labels $a_i,b_i,c_i,d_i$ are such that the tableaux are semi-standard. Suppose that $a_4<b_4$, $c_7<d_1$, and $c_{12} \geq d_2$ -- this determines the initial sets. As described in Theorem \ref{thm:2thm}, one can find a splitting of this tableaux pair by starting with the initial set of the last $a$. In this case, this initial set is $\{a_4,c_4, b_4\}$. Add to this set labels $c_{4+4\alpha}, b_{4+4\alpha}$, as long as no $d$s are passed; in case we add only $b_8$ and $c_8$. When $d_1$ is passed, we start adding multiples of $3$ as long as we have not passed $d_2$: our set is now
\[\{a_4,c_4,b_4,c_8,b_8,c_{11},b_{11}\}.\]
As we pass $d_2$, we start adding multiples of $2$, until we reach $d_3$ and $d_4$ (simultaneously in this case). We add $d_4$ to complete the last column, so the labels are:
\[\{a_4,c_4,b_4,c_8,b_8,c_{11},b_{11},c_{13},b_{13},c_{15},b_{15},d_4\}.\]
We highlight these labels in yellow below. 
\[\begin{tikzpicture}[scale=0.8,every node/.style={scale=0.7}]
 \foreach \x in {0,...,18}
 \draw [] (\x,0) rectangle (\x+1,1);
 \foreach \x in {0,...,18}
 \draw [] (\x,1) rectangle (\x+1,2);
 
  \foreach \x in {3,7,11,14,16,18}
 \draw [fill=yellow] (\x,1) rectangle (\x+1,2);
 \foreach \x in {4,8,12,14,17,19}
 \draw [fill=yellow] (\x-1,0) rectangle (\x,1);
 \draw [fill=orange] (13,1) rectangle (14,2);
\draw [fill=orange] (15,0) rectangle (16,1);
 
\foreach \x in {1,...,10}
\node[] at (\x-1/2,0.5) {$c_{\x}1$};

\node[] at (10.5,0.5) {$d_1 2$};
\node[] at (11.5,0.5) {$c_{11} 1$};
\node[] at (12.5,0.5) {$c_{12} 1$};
\node[] at (13.5,0.5) {$c_{13} 1$};
\node[] at (14.5,0.5) {$c_{14} 1$};
\node[] at (15.5,0.5) {$d_2 2$};
\node[] at (16.5,0.5) {$c_{15} 1$};
\node[] at (17.5,0.5) {$d_3 2$};
\node[] at (18.5,0.5) {$d_4 2$};

\node[] at (0.5,1.5) {$a_11$};
\node[] at (1.5,1.5) {$a_21$};
\node[] at (2.5,1.5) {$a_31$};
\node[] at (4.5,1.5) {$b_1 2$};
\node[] at (3.5,1.5) {$a_41$};
\foreach \x in {2,...,15}
\node[] at (\x+3.5,1.5) {$b_{\x}2$};

\node[] at (-1,1) {$T^-=$};

\end{tikzpicture}\]
\[\begin{tikzpicture}[scale=0.8,every node/.style={scale=0.7}]
 \foreach \x in {0,...,18}
 \draw [] (\x,0) rectangle (\x+1,1);
 \foreach \x in {0,...,18}
 \draw [] (\x,1) rectangle (\x+1,2);

 \foreach \x in {6,8,12,15,17,19}
 \draw [fill=yellow] (\x-1,1) rectangle (\x,2);
 \foreach \x in {4,8,12,15,17,19}
 \draw [fill=yellow] (\x-1,0) rectangle (\x,1);

\draw [fill=orange] (13,1) rectangle (14,2);
\draw [fill=orange] (13,0) rectangle (14,1);

\foreach \x in {1,...,8}
\node[] at (\x-1/2,0.5) {$b_{\x}1$};

\node[] at (8.5,0.5) {$d_1 2$};
\foreach \x in {9,...,12}
\node[] at (\x+0.5,0.5) {$b_{\x}1$};
\node[] at (13.5,0.5) {$d_2 2$};
\foreach \x in {13,...,15}
\node[] at (\x+1.5,0.5) {$b_{\x}1$};
\node[] at (17.5,0.5) {$d_3 2$};
\node[] at (18.5,0.5) {$d_4 2$};

\node[] at (0.5,1.5) {$a_11$};
\node[] at (1.5,1.5) {$c_1 2$};
\node[] at (2.5,1.5) {$a_21$};
\node[] at (3.5,1.5) {$a_31$};
\node[] at (4.5,1.5) {$c_2 2$};
\node[] at (5.5,1.5) {$a_41$};

\foreach \x in {3,...,15}
\node[] at (\x+3.5,1.5) {$c_{\x}2$};

\node[] at (-1,1) {$T^+=$};

\end{tikzpicture}\]
Whenever we have highlighted entire columns, there is nothing to check as far as semi-standardness goes. In $T^+$, there is one location where we do not have an entire column: at the beginning, where Lemma \ref{lem:splitting} ensures that semi-standardness is preserved. In $T^-$, there is the pair $b_{11} 2$ and $c_{13} 1$ to consider. Note that these labels fall between the elements of the initial set of $d_2$, highlighted in orange. In fact, Lemma \ref{lem:splitting} (or rather its analogue for $d$s) once again ensures that the tableau that is removed is semi-standard. We check this explicitly: by assumption, $c_{12} \geq d_2$ and $c_{12} \leq d_2$, so $c_{12}=d_2$, hence
\[b_{11} \leq b_{12}<d_2=c_{12} \leq c_{13}.\]
\end{eg}

\begin{thm}\label{thm:2thm} The semi-invariant ring of the Kronecker quiver with dimension vector $(2,2)$ and $K$ arrows has a finite SAGBI basis indexed by linked pairs of semi-standard tableaux of the form:
\[\begin{tikzpicture}[scale=1,every node/.style={scale=0.8}]
\draw [] (0,1) rectangle (1,2);
\draw [] (0,0) rectangle (1,1);
\draw [] (1,1) rectangle (2,2);
\draw [] (1,0) rectangle (2,1);
\draw [] (2,1) rectangle (3,2);
\draw [] (2,0) rectangle (3,1);
\draw [] (3,1) rectangle (4,2);
\draw [] (3,0) rectangle (4,1);

\draw [dashed, color=gray] (3,0) rectangle (8,1);
\draw [dashed, color=gray] (3,1) rectangle (8,2);
\draw [] (10,1) rectangle (11,2);
\draw [] (10,0) rectangle (11,1);
\draw [] (7,1) rectangle (8,2);
\draw [] (7,0) rectangle (8,1);
\draw [] (8,1) rectangle (9,2);
\draw [] (8,0) rectangle (9,1);
\draw [] (9,1) rectangle (10,2);
\draw [] (9,0) rectangle (10,1);

\node[] at (0.5,1.5) {$a 1$};
\node[] at (0.5,0.5){$c_1 1$};
\node[] at (1.5,1.5) {$b_1 2$};
\node[] at (1.5,0.5){$c_2 1$};
\node[] at (2.5,1.5) {$b_2 2$};
\node[] at (2.5,0.5){$c_3 1$};
\node[] at (3.5,1.5) {$b_3 2$};
\node[] at (3.5,0.5){$c_4 1$};

\node[] at (7.5,1.5) {$b_{l-3} 2$};
\node[] at (7.5,0.5){$c_{l-2} 1$};
\node[] at (8.5,1.5) {$b_{l-2} 2$};
\node[] at (8.5,0.5){$c_{l-1} 1$};
\node[] at (9.5,1.5) {$b_{l-1} 2$};
\node[] at (9.5,0.5){$c_l 1$};
\node[] at (10.5,1.5) {$b_l 2$};
\node[] at (10.5,0.5){$d 2$};

\node[] at (-1,1) {$T^-=$};

\draw[->] (0.8,1.5) to [] (1.2,1.5);
\draw[->] (9.8,0.5) to [] (10.2,0.5);
\draw[->] (0.8,0.8) to [] (2.2,1.4);
\draw[->] (1.8,0.8) to [] (3.2,1.4);
\draw[->] (2.8,0.8) to [] (4.2,1.4);
\draw[->] (5.8,0.8) to [] (7.2,1.4);
\draw[->] (6.8,0.8) to [] (8.2,1.4);
\draw[->] (7.8,0.8) to [] (9.2,1.4);

\end{tikzpicture}\]
\[\begin{tikzpicture}[scale=1,every node/.style={scale=0.8}]
\draw [] (0,1) rectangle (1,2);
\draw [] (0,0) rectangle (1,1);
\draw [] (1,1) rectangle (2,2);
\draw [] (1,0) rectangle (2,1);
\draw [] (2,1) rectangle (3,2);
\draw [] (2,0) rectangle (3,1);
\draw [] (3,1) rectangle (4,2);
\draw [] (3,0) rectangle (4,1);

\draw [dashed, color=gray] (3,0) rectangle (8,1);
\draw [dashed, color=gray] (3,1) rectangle (8,2);
\draw [] (10,1) rectangle (11,2);
\draw [] (10,0) rectangle (11,1);
\draw [] (7,1) rectangle (8,2);
\draw [] (7,0) rectangle (8,1);
\draw [] (8,1) rectangle (9,2);
\draw [] (8,0) rectangle (9,1);
\draw [] (9,1) rectangle (10,2);
\draw [] (9,0) rectangle (10,1);

\node[] at (0.5,1.5) {$a 1$};
\node[] at (0.5,0.5){$b_1 1$};
\node[] at (1.5,1.5) {$c_1 2$};
\node[] at (1.5,0.5){$b_2 1$};
\node[] at (2.5,1.5) {$c_2 2$};
\node[] at (2.5,0.5){$b_3 1$};
\node[] at (3.5,1.5) {$c_3 2$};
\node[] at (3.5,0.5){$b_4 1$};

\node[] at (7.5,1.5) {$c_{l-3} 2$};
\node[] at (7.5,0.5){$b_{l-2} 1$};
\node[] at (8.5,1.5) {$c_{l-2} 2$};
\node[] at (8.5,0.5){$b_{l-1} 1$};
\node[] at (9.5,1.5) {$c_{l-1} 2$};
\node[] at (9.5,0.5){$b_l 1$};
\node[] at (10.5,1.5) {$c_l 2$};
\node[] at (10.5,0.5){$d 2$};

\node[] at (-1,1) {$T^+=$};

\draw[->] (0.8,1.5) to [] (1.2,1.5);
\draw[->] (9.8,0.5) to [] (10.2,0.5);
\draw[->] (0.8,0.8) to [] (2.2,1.4);
\draw[->] (1.8,0.8) to [] (3.2,1.4);
\draw[->] (2.8,0.8) to [] (4.2,1.4);
\draw[->] (5.8,0.8) to [] (7.2,1.4);
\draw[->] (6.8,0.8) to [] (8.2,1.4);
\draw[->] (7.8,0.8) to [] (9.2,1.4);

\end{tikzpicture}\]
%where the labels satisfy:
%\[1 \leq a<c_1<b_2<c_3<\cdots<b_l<d \leq K, \hspace{5mm} 1 \leq a<b_1<c_2<b_3<\cdots<c_l<d \leq K.\]
\end{thm}
\begin{proof}
Note that every primitive semi-standard tableaux pair $T^\pm$ with $k=1$ is of the form in the theorem (recall that $k$ is the number $a$s, i.e. the number of labels with second digit $1$ in the first row of $T^+$). We first show that the set as described in the theorem is finite. Note that as the tableaux are semi-standard, the labels satisfy two chains of inequalities:
\[a<c_1<b_2<c_3<\cdots<b_l<d, \hspace{5mm} a<b_1<c_2<b_3<\cdots<c_l<d.\]
Since every label lies between $1$ and $K$, as there are $K$ arrows in the quiver, this is only possible if $l+2 \leq K$. Therefore, only tableaux pairs of dimensions $2 \times j$, $j \leq K-1$ can possibly be primitive. This implies finiteness of the SAGBI basis.

It remains to show that if $T^\pm$ is a primitive semi-standard tableaux pair, then $k=1$. This uses the construction of the initial set above. Suppose that $k>1$. Begin with the furthest right $a$ label, $a_k$, and take its initial set. Near $a_k$, the tableaux pair looks like:
\[\begin{tikzpicture}[scale=1,every node/.style={scale=0.8}]
\draw [] (2,1) rectangle (3,2);
\draw [] (0,1) rectangle (1,2);
\draw [] (0,0) rectangle (1,1);
\draw [] (4,1) rectangle (5,2);

\draw [dashed, color=gray] (-1,0) rectangle (6,1);
\draw [dashed, color=gray] (-1,1) rectangle (6,2);

\node[] at (0.5,1.5) {$a_k1$};
\node[] at (0.5,0.5) {$c_{y+t}1$};
\node[] at (2.5,1.5) {$b_{x+t} 2$};
\node[] at (4.5,1.5) {$b_{z_-} 2$};
\node[] at (-2,1) {$T^-=$};

\draw[->] (0.8,0.8) to [] (4.2,1.3);
\draw[->] (0.8,1.5) to [] (2,1.5);

\end{tikzpicture}\]

\[\begin{tikzpicture}[scale=1,every node/.style={scale=0.8}]
\draw [] (2,1) rectangle (3,2);
\draw [] (0,1) rectangle (1,2);
\draw [] (0,0) rectangle (1,1);
\draw [] (4,1) rectangle (5,2);

\draw [dashed, color=gray] (-1,0) rectangle (6,1);
\draw [dashed, color=gray] (-1,1) rectangle (6,2);

\node[] at (0.5,1.5) {$a_k1$};
\node[] at (0.5,0.5) {$b_{x+t}1$};
\node[] at (2.5,1.5) {$c_{y+t} 2$};
\node[] at (4.5,1.5) {$c_{z_+}2$};
\node[] at (-2,1) {$T^+=$};

\draw[->] (0.8,0.8) to [] (4.2,1.3);
\draw[->] (0.8,1.5) to [] (2,1.5);

\end{tikzpicture}\]
In $T^-$, if $p1$ is a label weakly to left of $c_{y+t}$, then $p \leq c$. There are $y+t+k-1$ such labels, and $y+2t-1$ of them have an arrow to the first row. If $p1$ is a label strictly to the right of $c_{y+t}$, then $p1$ must appear in the second row of $T^-$, and hence $p \geq c$. Therefore $z_-=y+2t$.  Similar reasoning demonstrates that $z_+=x+2t$. 

As long as only $b$s and $c$s appear, this pattern continues: the target of the arrow from $b_p 1$ in $T^+$ is $c_{p+t} 2$, and the target of the arrow from $c_p 1$ in $T^-$ is $b_{p+t} 2$.

Until a $d$ type box is passed in either $T^\pm$, we add boxes to our initial set by adding $t$ to the $b$ and $c$ labels appearing in $I^a_k$. This determines a sub-tableau that, other than for the initial set, contains columns $t$ positions apart. That is, we obtain either the set
\[\{a_k,c_{x+t},b_{x+t}\} \cup \{c_{x+\alpha t},b_{x+\alpha t}: 2 \leq \alpha \leq A\} \]
or
\[\{a_k,c_{y+t},b_{y+t}\} \cup \{c_{y+\alpha t},b_{y+\alpha t}: 2 \leq \alpha \leq A\}, \]
where $A$ is an integer that ensures that all elements are to the left of $d_{s+1}$. Recall that  $d_{s+1}$ is the first $d$ on the right of $a_k$ We illustrate this in the case where $I^a_k:=\{a_k,c_{x+t},b_{x+t}\}$ and $x \leq y$:
\[\begin{tikzpicture}[scale=1,every node/.style={scale=0.8}]
\draw [fill=yellow] (-2,0) rectangle (-1,1);
\draw [fill=yellow] (2,1) rectangle (3,2);
\draw [fill=yellow] (2,0) rectangle (3,1);
\draw [fill=yellow] (0,1) rectangle (1,2);
\draw [] (0,0) rectangle (1,1);

\draw [fill=yellow] (4,0) rectangle (5,1);
\draw [fill=yellow] (4,1) rectangle (5,2);
\draw [fill=yellow] (6,0) rectangle (7,1);
\draw [fill=yellow] (6,1) rectangle (7,2);

\draw [dashed, color=gray] (-3,0) rectangle (6,1);
\draw [dashed, color=gray] (-3,1) rectangle (6,2);

\node[] at (0.5,1.5) {$a_k1$};
\node[] at (0.5,0.5) {$c_{y+t}1$};
\node[] at (-1.5,0.5) {$c_{x+t}1$};

\node[] at (2.5,1.5) {$b_{x+t} 2$};
\node[] at (2.5,0.5) {$c_{x+2t} 2$};

\node[] at (4.5,1.5) {$b_{x+2t} 2$};
\node[] at (4.5,0.5) {$c_{x+3t} 2$};
\node[] at (6.5,1.5) {$b_{x+3t} 2$};
\node[] at (6.5,0.5) {$c_{x+4t} 2$};

\node[] at (-4,1) {$T^-=$};

%\draw[->] (0.8,0.8) to [] (4.2,1.3);
%\draw[->] (0.8,1.5) to [] (2,1.5);

\end{tikzpicture}\]

\[\begin{tikzpicture}[scale=1,every node/.style={scale=0.8}]
\draw [fill=yellow] (2,1) rectangle (3,2);
\draw [] (4,1) rectangle (5,2);
\draw [fill=yellow] (0,1) rectangle (1,2);
\draw [fill=yellow] (0,0) rectangle (1,1);
\draw [fill=yellow] (2,0) rectangle (3,1);
\draw [fill=yellow] (2,1) rectangle (3,2);
\draw [fill=yellow] (4,0) rectangle (5,1);
\draw [fill=yellow] (4,1) rectangle (5,2);
\draw [fill=yellow] (6,0) rectangle (7,1);
\draw [fill=yellow] (6,1) rectangle (7,2);

\draw [] (6,1) rectangle (7,2);

\draw [dashed, color=gray] (-1,0) rectangle (8,1);
\draw [dashed, color=gray] (-1,1) rectangle (8,2);

\node[] at (0.5,1.5) {$a_k1$};
\node[] at (0.5,0.5) {$b_{x+t}1$};
\node[] at (2.5,1.5) {$c_{x+t} 2$};
\node[] at (2.5,0.5) {$b_{x+2t} 2$};

\node[] at (4.5,1.5) {$c_{x+2t} 2$};
\node[] at (4.5,0.5) {$b_{x+3t} 2$};

\node[] at (6.5,1.5) {$c_{x+3t} 2$};
\node[] at (6.5,0.5) {$b_{x+4t} 2$};
\node[] at (-2,1) {$T^+=$};

%\draw[->] (0.8,0.8) to [] (4.2,1.3);
%\draw[->] (0.8,1.5) to [] (2,1.5);

\end{tikzpicture}\]

There is similar start to sub-tableaux beginning with initial set of $d_{s+1}$, and moving to the left. We set $\underline{x},\underline{y},\underline{t},\underline{s}$ to be indices associated with $d_{s+1}$. As $\underline{s}=0$, note that
\[\underline{t}=k-(s+1)+1-\underline{s}=k-s=t.\]
Therefore the columns in the sub-tableaux starting at $d_{s+1}$ are also $t$ apart. They either completely coincide with the columns from $a_k$ or are disjoint. If they coincide, we have a sub-tableaux pair of $T^\pm$, which, by Lemma \ref{lem:splitting} is semi-standard and, once removed, what remains is also semi-standard.  This contradicts primitivity. 

Otherwise, the columns from the initial set of $a_k$ have no common labels with the initial set of $d_{s+1}$. In this case, we do the following. While adding boxes to our initial set starting at $a_k$, when we reach a box that lies between $d_{s+1}$ and one of the other two labels in $I^d_{s+1}$, we add the other box not in the column of $T^\pm$ but in the column that would appear when $I^d_{s+1}$ is removed. We illustrate this in the case where  $I^a_k:=\{a_k,c_{x+t},b_{x+t}\}$, $I^d_{s+1}=\{d_{s+1},c_{\hat{\underline{x}}-t},b_{\hat{\underline{x}}-t}\}$, and $\hat{\underline{x}} \leq \hat{\underline{y}}$. We can define $A$ now more precisely as the smallest integer satisfying 
\[x+At \geq \hat{x}-t,\]
(this could be expressed using the floor function). Note that the inequality is necessarily strict as we have assumed that $x \not \equiv \hat{x}$ modulo $t$.

\[\begin{tikzpicture}[scale=1.4,every node/.style={scale=0.7}]
\draw [fill=yellow] (2,0) rectangle (3,1);
\draw [fill=yellow] (2,1) rectangle (3,2);
\draw [fill=orange] (4,1) rectangle (5,2);
\draw [fill=orange] (4,0) rectangle (5,1);
\draw [fill=orange] (6,1) rectangle (7,2);
\draw [fill=orange] (11,0) rectangle (12,1);
\draw [fill=yellow] (9,1) rectangle (10,2);
\draw [] (9,0) rectangle (10,1);
\draw [fill=yellow] (8,0) rectangle (9,1);

\draw [dashed, color=gray] (1.5,0) rectangle (12.5,1);
\draw [dashed, color=gray] (1.5,1) rectangle (12.5,2);

\node[] at (2.5,0.5) {$b_{x+(A-1)t} 2$};
\node[] at (2.5,1.5) {$c_{x+A t} 1$};

\node[] at (4.5,1.5) {$b_{\hat{\underline{x}}-2t} 2$};
\node[] at (4.5,0.5) {$c_{\hat{\underline{x}}-t} 1$};

\node[] at (9.5,1.5) {$b_{x+At} 2$};
\node[] at (9.5,0.5) {$c_{x+(A+1)t} 1$};
\node[] at (8.5,0.5) {$c_{x+At+(t-1)} 1$};

\node[] at (6.5,1.5) {$b_{\hat{\underline{x}}-t}2$};
\node[] at (11.5,0.5) {$d_{s+1}2$};

\node[] at (1,1) {$T^-=$};

\end{tikzpicture}\]
\[\begin{tikzpicture}[scale=1.4,every node/.style={scale=0.7}]
\draw [fill=yellow] (2,0) rectangle (3,1);
\draw [fill=yellow] (2,1) rectangle (3,2);
\draw [fill=orange] (4,1) rectangle (5,2);
\draw [fill=orange] (4,0) rectangle (5,1);
\draw [fill=orange] (6,1) rectangle (7,2);
\draw [fill=orange] (6,0) rectangle (7,1);
\draw [fill=yellow] (8,1) rectangle (9,2);
\draw [fill=yellow] (8,0) rectangle (9,1);

\draw [dashed, color=gray] (1.5,0) rectangle (10,1);
\draw [dashed, color=gray] (1.5,1) rectangle (10,2);

\node[] at (2.5,0.5) {$b_{x+At} 2$};
\node[] at (2.5,1.5) {$c_{x+(A-1) t} 1$};

\node[] at (4.5,1.5) {$c_{\hat{\underline{x}}-2t} 2$};
\node[] at (4.5,0.5) {$b_{\hat{\underline{x}}-t} 1$};

\node[] at (6.5,1.5) {$c_{\hat{\underline{x}}-t} 2$};
\node[] at (6.5,0.5) {$d_{s+1}2$};

\node[] at (8.5,1.5) {$c_{x+At} 2$};
\node[] at (8.5,0.5) {$b_{x+At+(t-1)} 1$};

\node[] at (1,1) {$T^+=$};

\end{tikzpicture}\]
Here we have highlighted in yellow the sub-tableaux starting at $a_k$, and in orange the sub-tableaux starting at $d_{s+1}$. Until we pass $d_{s+2}$, we can keep adding boxes to the sub-tableaux starting at $a_k$ so that columns are now $t-1$ apart. That is, in the case illustrated above, the sub-tableaux contain labels
\[\{a_k,c_{x+t},b_{x+t}\} \cup \{c_{x+\alpha t},b_{x+\alpha t}: 2 \leq \alpha \leq A,\} \cup  \{c_{x+A t+\alpha (t-1)},b_{x+At + \alpha (t-1)}: 2 \leq \alpha \leq B\},\]
where $B$ is an integer that ensures that all elements are to the left of $d_{s+2}$. 
If $I^a_k=\{a_k,c_{y+t},b_{y+t}\},$ then the sub-tableaux contains the labels
\[\{a_k,c_{y+t},b_{y+t}\} \cup \{c_{y+\alpha t},b_{y+\alpha t}: 2 \leq \alpha \leq A,\}  \cup \{c_{y+A t+\alpha (t-1)},b_{y+At + \alpha (t-1)}: 2 \leq \alpha \leq B\} .\] 
The situation when $I^d_{s+1}=\{d_{s+1},c_{\hat{\underline{\hat{y}}}-t},b_{\hat{\underline{y}}-t}\}$ and $\hat{\underline{x}} \leq \hat{\underline{y}}$ is very similar: the only change is in the definition of $A$ as the smallest integer satisfying
 \[x+At \geq \underline{\hat{y}}-t.\]
 
 This process continues for each $d_{s+p}$: either we end the sub-tableaux using the initial set of $d_{s+p}$, or we by-pass it using the same method as for $d_{s+1}$, and end up with a sub-tableau where the columns at this stage are now $t-p$ apart, and the $\underline{t}$ for $d_{s+p+1}$ is $t-p$. If we pass $d_{k-1}1$, then the columns are $1$ apart, and so we necessarily end our sub-tableaux using the initial set coming from $d_k$. 
 
The final result is a sub-tableaux pair that induces a semi-standard splitting of $T^\pm$. The fact that all tableaux in the splitting are semi-standard follows from Lemma \ref{lem:splitting} and its $d$-analogue. This concludes the proof of the theorem. 
\end{proof}

\begin{rem} Minimal sets of generators of $\SI(Q,(2,2))$, where $Q$ is the Kronecker quiver with $K$, have been considered in \cite{domokos2} in characteristic zero, in \cite{domokoschar} in any characteristic, and as a special case of \cite{lopatin}. The sets presented in these papers are of course smaller than the SAGBI basis presented in Theorem \ref{thm:2thm}. Interestingly, in characteristic 2, the generating set presented in \cite{domokoschar} has some similarities with the SAGBI basis. In degree one and two the generators are the same. The higher degree generators, $\xi(x_{k_1},\dots,x_{k_{2q}})$, for $1 \leq k_1<\cdots<k_{2q} \leq K$, correspond to the semi-invariants arising from pairs of linked tableaux described as follows.  In $T^+$, the labels with first digit $k_i$ and $k_{q+i}$ appear in the same column, for $i=1,\dots,q$. In $T^-$, the labels with first digit $i$ and $q+i-1$ appear in the same column for $i=2,\dots,q$, and $1$ and $2q$ appear in the same column. Different choices of $T^\pm$ satisfying these conditions will give the same semi-invariants up to sign. 
\end{rem}

\section{An example of a toric degeneration}\label{sec:example}
In this section, we describe a toric degeneration of a Fano quiver moduli space using the SAGBI basis described above. Let $Q$ be the Kronecker quiver with three arrows and dimension vector $\br=(2,3)$. 

Let $\theta:=(-9,6)$ be the anti-canonical stability condition. Then the unstable locus with respect to $\theta$ has co-dimension at least 2 by \cite[Proposition 5.2]{reineke}. By \cite{fanoquiver}, the stable and semi-stable locus coincide. The GIT quotient $M_\theta(Q,\br)$ is a six-dimensional smooth Fano variety with Fano index $3$. Its Cox ring coincides with the semi-invariant ring, and therefore a SAGBI basis is given by the linked tableaux in Example \ref{eg:233}. Therefore, there exists a toric degeneration from $M_\theta(Q,\br)$ to the toric variety given by the leading terms of the semi-invariants. This toric variety corresponds to the Fano polytope $P$ with vertices:
\begin{gather}
\begin{aligned} \label{eqn:verts}\{
    (2, -1, 2, 0, -1, 0),
    (2, -1, 2, 1, -1, -1),
    (1, 0, 1, 0, -1, 0),
    (1, 0, 1, 1, -1, -1),\\
    (0, 0, 0, 0, 0, 1),
    (0, 0, 0, 0, 1, 0),
    (0, 0, 0, 1, 0, 0),
    (0, 0, 1, 0, 0, 0),
    (1, -1, 0, -1, 0, 0),\\
     (2, -1, 1, 0, 0, 0),
    (1, -1, 1, -1, -1, 0),
    (-1, 1, -1, 0, 1, 0),
    (-1, 0, -1, 0, 0, 0)
\}.\end{aligned}\end{gather}

We interpret this degeneration in the context of mirror symmetry for Fano varieties. For some background on mirror symmetry for Fano varieties, see \cite{fanomanifolds, CoatesCortiGalkinKasprzyk2016}. The mirror to a deformation class of $n$-dimensional smooth Fano varieties, also known as its (weak) \emph{Landau--Ginzburg model}, is a  mutation class of certain Laurent polynomials in $n$ variables. The Laurent polynomials that are mirror to smooth Fano varieties are, conjecturally, \emph{rigid maximally mutable} Laurent polynomials \cite{Coates_2021}. 

To check that representatives $X$ and $f$ belong to mirror classes, one shows that two power series coincide. The power series associated to a Fano variety $X$ is the \emph{quantum period}, which is built out of genus zero Gromov--Witten invariants, and hence deformation invariant. The power series associated to the Laurent polynomial $f$ is called the \emph{classical period}; it is easy to compute and is mutation invariant. See \cite{CoatesCortiGalkinKasprzyk2016} for a definition of both periods. Checking that the two periods coincide can be very difficult, as it is often hard or impossible to find a closed formula for the quantum period of a Fano variety. For Fano toric complete intersections, this is a consequence of the celebrated mirror theorem \cite{givental, lian}.  

The Newton polytope $P$ of a Laurent polynomial $f$ which is a \emph{rigid maximally-mutable} Laurent polynomial \cite{Coates_2021} is a Fano polytope, i.e. $P$ spans the fan of a (singular) toric Fano variety $X_P$. Part of the mirror symmetry conjectures is that $X_P$ can smooth to a Fano variety $X$ which is mirror to $f$. Thus, if we want to find a conjectural mirror to a Fano variety $X$, we find a toric degeneration of $X$ to some $X_P$. This determines the monomials of $f$, and coefficients are then chosen to ensure that $f$ is rigid maximally mutable. 

The toric variety associated to the face-fan of the polytope $P$ with vertices \eqref{eqn:verts} is Gorenstein Fano, with terminal singularities. Since $P$ contains no other lattice points except for the origin and its vertices, there exists a unique rigid maximally-mutable Laurent polynomial supported on $P$. In other words,
\begin{align*}
f = x_1^2 x_3^2 x_4/(x_2 x_5 x_6) + x_1^2 x_3^2/(x_2 x_5) + x_1^2 x_3/x_2 + x_1 x_3 x_4/(x_5 x_6) \\
+ x_1 x_3/x_5 + x_1 x_3/(x_2 x_4 x_5) + x_1/(x_2 x_4) + x_3 + x_4 + x_5 + x_6 + x_2 x_5/(x_1 x_3) + 
1/(x_1 x_3),\end{align*}
where each monomial corresponds to one of the 13 vertices of $P$ and $f$ has coefficient $1$ for each of these. We can now consider the mirror condition for $M_\theta(Q,\br)$ and $f$.

 The first 20 terms of the period sequence of this Laurent polynomial are
\[[ 1, 0, 0, 18, 0, 0, 4590, 0, 0, 1728720, 0, 0, 876610350, 0, 0, 520461209268, 
0, 0, 343838539188144, 0].\]
As a sanity check, the periodicity of pairs of zeroes in this sequence confirms the Fano index $3$ of $M_\theta(Q,\br)$. Conjecturally, the sequence should coincide with the quantum period of $M_\theta(Q,\br)$. The Fano variety $M_\theta(Q,\br)$ can also be described \cite{belmans} as the zero locus of a generic section of $E^* \otimes \det(E)=\wedge^5E$ on $\Gr(8,6)$, where $E$ is the rank six tautological quotient bundle. Then, terms of the quantum period of  $M_\theta(Q,\br)$ can be computed using the Abelian/non-Abelian correspondence.  In private communication, P. Belmans has verified that the first $20$ terms of this computation coincide with the $20$ terms in the classical period of $f$ listed above. 
 
%\section{Declarations}
%\textbf{Funding and/or Conflicts of interests/Competing interests: }
%
%The authors have no relevant financial or non-financial interests to disclose.
%
%The authors have no conflicts of interest to declare that are relevant to the content of this article.
%
%All authors certify that they have no affiliations with or involvement in any organization or entity with any financial interest or non-financial interest in the subject matter or materials discussed in this manuscript.
%
%The authors have no financial or proprietary interests in any material discussed in this article.

\bibliographystyle{amsplain}
\bibliography{epsrc}
\end{document}